\documentclass[10pt,reqno]{amsart}
\usepackage{amssymb,mathrsfs,graphicx}
\usepackage{ifthen}
\usepackage[colorlinks=true, pdfstartview=FitV, linkcolor=blue,citecolor=deepred, urlcolor=RoyalBlue]{hyperref}

\usepackage[margin=1in]{geometry}
\usepackage{caption}
\usepackage{rotating,dsfont}
\usepackage{enumitem}
\DeclareMathOperator*{\esssup}{ess\,sup}

\usepackage{soul}
\usepackage{cancel}
\usepackage[dvipsnames]{xcolor}
\definecolor{deepred}{rgb}{0.7, 0, 0} 
\newcommand{\delE}{\delta{\mathscr E}}
\renewcommand{\geq}{\geqslant}
\renewcommand{\ge}{\geqslant}
\renewcommand{\leq}{\leqslant}
\renewcommand{\le}{\leqslant}

\usepackage{soul}
\usepackage{cancel}

 \usepackage{tikz}
 \usetikzlibrary{arrows.meta,positioning,calc}

\usepackage{colortbl}
\definecolor{black}{rgb}{0.0, 0.0, 0.0}
\definecolor{red}{rgb}{1.0, 0.5, 0.5}
\provideboolean{shownotes} 
\setboolean{shownotes}{true} 
\newcommand{\margnote}[1]{
\ifthenelse{\boolean{shownotes}}%
{\marginpar{\raggedright\tiny\texttt{#1}}}%
{}%
}
\newcommand{\hole}[1]{
\ifthenelse{\boolean{shownotes}}%
{\begin{center} \fbox{ \rule {.25cm}{0cm} \rule[-.1cm]{0cm}{.4cm}
\parbox{.85\textwidth}{\begin{center} \texttt{#1}\end{center}} \rule
{.25cm}{0cm}}\end{center}} {} }

\title[Lagrangian formulation of alignment dynamics]{Lagrangian formulation and Eulerian closure\\in alignment  dynamics}

\author[Carrillo]{Jos\'{e} A. Carrillo}
\address[Jos\'{e} A. Carrillo]{\newline Mathematical Institute, University of Oxford, Oxford OX2 6GG, United Kingdom}
\email{jose.carrillo@maths.ox.ac.uk}

\author[Choi]{Young-Pil Choi}
\address[Young-Pil Choi]{\newline Department of Mathematics, Yonsei University, Seoul 03722, Republic of Korea}
\email{ypchoi@yonsei.ac.kr}

\author[Tadmor]{Eitan Tadmor}
\address[Eitan Tadmor]{\newline Department of Mathematics and IPST, University of Maryland, College Park, MD 20742, USA}
\email{tadmor@umd.edu}

\numberwithin{equation}{section}

\newtheorem{theorem}{Theorem}[section]
\newtheorem{lemma}{Lemma}[section]

\newtheorem{proposition}{Proposition}[section]
\newtheorem{remark}{Remark}[section]

\newcommand{\R}{\mathbb R}

\newcommand{\N}{\mathbb N}

\newcommand{\bq}{\begin{equation}}
\newcommand{\eq}{\end{equation}}
\newcommand{\e}{\varepsilon}
\newcommand{\lt}{\left}
\newcommand{\rt}{\right}

\newcommand{\pa}{\partial}

\newcommand{\intr}{\int_{\R^d}}

\newcommand{\intrr}{\iint_{\R^d \times \R^d}}

\newcommand{\calA}{\mathcal A}

\newcommand{\calD}{\mathcal D}
\newcommand{\calE}{\mathcal E}

\newcommand{\calI}{\mathcal I}

\newcommand{\calK}{\mathcal K}
\newcommand{\calL}{\mathcal L}
\newcommand{\calM}{\mathcal M}

\newcommand{\calP}{\mathcal P}

\newcommand{\calR}{\mathcal R}

\newcommand{\rd}{\textnormal{d}}
\newcommand{\dx}{\textnormal{d}x}
\newcommand{\dr}{\textnormal{d}r}
\newcommand{\dt}{\textnormal{d}t}
\newcommand{\dv}{\textnormal{d}v}
\newcommand{\dy}{\textnormal{d}y}
\newcommand{\dw}{\textnormal{d}w}
\newcommand{\dz}{\textnormal{d}z}

\newcommand{\ds}{\textnormal{d}s}
\newcommand{\tr}{\textnormal{tr}}

\makeatletter
\def\moverlay{\mathpalette\mov@rlay}
\def\mov@rlay#1#2{\leavevmode\vtop{%
   \baselineskip\z@skip \lineskiplimit-\maxdimen
   \ialign{\hfil$\m@th#1##$\hfil\cr#2\crcr}}}
\newcommand{\charfusion}[3][\mathord]{
    #1{\ifx#1\mathop\vphantom{#2}\fi
        \mathpalette\mov@rlay{#2\cr#3}
      }
    \ifx#1\mathop\expandafter\displaylimits\fi}
\makeatother

\newcommand{\supp}{\text{supp}}
\newcommand{\diam}{\text{diam}}

\begin{document}
\allowdisplaybreaks

\date{\today}


\begin{abstract}  
We investigate a continuum Lagrangian $p$-alignment  system given by a nonlocal mean-field system of ordinary differential equations for interacting agents with weak initial data. We first establish global well-posedness of the Lagrangian dynamics and derive quantitative flocking estimates. We next construct Eulerian variables from the possibly non-injective Lagrangian flow via pushforward and disintegration, which leads to an Euler--Reynolds--alignment system featuring a nonnegative Reynolds stress and, for $p>2$, a nonlinear defect force induced by microscopic velocity fluctuations. Assuming only heavy-tailed interaction, we then show that these defect terms vanish asymptotically, leading to asymptotic mono-kinetic closure in the long-time limit. In the linear case $p=2$, we further obtain global weak solutions to the Euler--alignment system, including a sharp one-dimensional critical-threshold characterization and a global result in higher dimensions under a large-coupling condition.  Finally, we establish a uniform-in-time mean-field stability estimate for the particle Cucker--Smale system in the linear regime and deduce uniform-in-time convergence toward the mono-kinetic Eulerian limit; for general $p\ge2$, we also obtain a finite-time mean-field convergence result toward the associated kinetic/Lagrangian alignment dynamics.
\end{abstract}

\maketitle \centerline{\date}

\setcounter{tocdepth}{1}
\tableofcontents

%
%
%
%
\section{Introduction}
We consider the following infinite-dimensional, measure-dependent system of ordinary differential equations describing the Lagrangian evolution of a continuum ensemble of interacting agents:
\begin{align}\label{Lag_p}
\begin{aligned}
\pa_t \eta_t &= v_t, \quad (t,x) \in \R_+ \times \supp(\rho_0),\cr
\pa_t v_t &=\kappa \intr \phi(\eta_t(x) - \eta_t(y)) G_p(v_t(y) - v_t(x))\,\rho_0(\dy), \quad  G_p(\xi):=|\xi|^{p-2}\xi,
\end{aligned}
\end{align}
which we call the \emph{Lagrangian $p$-alignment formulation}. When $p=2$, the interaction becomes linear in the velocity difference, and we refer to \eqref{Lag_p} as the \emph{Lagrangian alignment formulation}. Here $\kappa>0$ denotes the coupling strength, $\rho_0 \in \calP(\R^d)$ is a given probability measure describing the initial spatial distribution of particles, and the initial data are prescribed as
\bq\label{ini:Lag_p}
(\eta_t(x), v_t(x))|_{t=0} = (x, u_0(x)), \quad x \in \supp(\rho_0).
\eq
Throughout the paper, we assume that the communication kernel $\phi \in C^1 \cap W^{1,\infty}(\R^d;\R_+)$ is radially symmetric and non-increasing. In particular, $\phi(z)=\phi(|z|)$ and $\phi$ is even, ensuring symmetry of
pairwise interactions.

The system \eqref{Lag_p} can be viewed as a \emph{nonlocal mean-field ODE} posed on the space of measurable mappings $(\eta_t,v_t): \supp(\rho_0)\to\R^d$, where each trajectory $x \mapsto (\eta_t(x),v_t(x))$ evolves according to a velocity field depending nonlocally on the entire configuration through $\rho_0$. Viewed from this perspective, the Lagrangian system \eqref{Lag_p} provides a natural intermediate framework that will allow us to connect discrete alignment models with their Eulerian continuum descriptions. Moreover, as will be seen later, it also admits a natural kinetic representation in phase space, further linking the Lagrangian and kinetic viewpoints on alignment dynamics.

%
%
%
%
%
\medskip
\noindent {\bf Cucker--Smale flocking model and continuum limits.} The celebrated Cucker--Smale (CS) model describes the emergence of collective alignment in a system of $N$ interacting agents, such as birds in a flock or individuals in a crowd \cite{CS07}. In its classical particle formulation, the dynamics are governed by
\begin{align}\label{CS}
\begin{aligned}
\dot{x}_i &= v_i, \quad i=1,\dots, N, \quad t > 0,\cr
\dot{v}_i &= \frac\kappa N \sum_{j=1}^N\phi (x_j - x_i)G_p(v_j - v_i),
\end{aligned}
\end{align}
where $x_i,v_i\in\R^d$ denote the position and velocity of the $i$-th agent at $t\geq0$. Under mild regularity and decay assumptions on $\phi$, the CS model exhibits \emph{velocity alignment}: the velocity diameter decays to zero, often resulting in asymptotic consensus or flocking behavior. For the classical case $p=2$, this phenomenon and its large-time behavior have been extensively studied, see for instance \cite{CS07, HL09, HT08}, while nonlinear velocity couplings $p\neq2$ were investigated in \cite{CRC10,HHK10}.

In the mean-field limit $N\to\infty$, the empirical measure
\[
\mu_t^N:=\frac1N\sum_{i=1}^N \delta_{(x_i(t),v_i(t))},
\]
associated to \eqref{CS}, is expected to converge to a kinetic distribution $f_t(x,v)$ solving the Vlasov-type equation
\bq\label{VA}
\pa_t f_t + v \cdot \nabla_x f_t + \nabla_v \cdot (F_p[f_t]f_t) =0,
\eq
where the mean-field alignment force is given by
\[
F_p[f_t](x,v) = \kappa \intrr \phi(x-y)G_p(w-v)f_t(\dy,\dw).
\]

The kinetic formulation \eqref{VA} provides a statistical description of the collective dynamics, retaining the full velocity distribution at each spatial location while averaging out individual particle labels. It thus serves as a natural continuum limit of the particle system \eqref{CS}, allowing one to investigate alignment mechanisms and large-time behavior at the level of phase-space measures, without imposing any \emph{a priori} concentration or mono-kinetic assumptions. From this perspective, the well-posedness and qualitative properties of kinetic alignment equations
of the form \eqref{VA} have been extensively studied; we refer to \cite{CCH14, CFRT10, CJu23, CZ21, HL09,   KMT13,  MMPZ19} and the references therein. The rigorous derivation of \eqref{VA} from the particle dynamics \eqref{CS} has also been the subject of extensive research. Mean-field limits have been established for linear velocity coupling in, for instance, \cite{CCHS19, CFRT10, HL09}.  Related quantitative mean-field limits and propagation of chaos estimates for flocking models with nonlinear velocity couplings have also been obtained recently, see \cite{NSTpre}.

At the hydrodynamic level, assuming local velocity concentration $f_t(x,v)\simeq \rho_t(x)\otimes \delta_{u_t(x)}(v)$, one formally obtains the macroscopic (mono-kinetic) \emph{Eulerian $p$-alignment system} \cite{LuTa23, Tad23}
\begin{align}\label{EA}
\begin{aligned}
&\pa_t \rho_t + \nabla \cdot (\rho_t u_t) = 0, \quad t>0, \ x \in \R^d, \cr
&\pa_t (\rho_t u_t) + \nabla \cdot (\rho_t u_t \otimes u_t) = \rho_t\calA_p[\rho_t,u_t],
\end{aligned}
\end{align}
where
\[
\calA_p[\rho_t,u_t](x) := \kappa \intr \phi(x-y) G_p(u_t(y) - u_t(x))\,\rho_t(\dy).
\]
This system describes the evolution of the macroscopic density $\rho$ and the mean velocity field $u$, incorporating nonlocal alignment effects through $\phi$. 

The rigorous derivation of hydrodynamic equations of the form \eqref{EA} from underlying kinetic models has attracted considerable attention in recent years. Depending on the modeling assumptions and the relaxation mechanisms involved, such limits may lead to Euler-type systems with or without additional pressure terms. For linear velocity coupling ($p=2$), hydrodynamic limits have been derived from kinetic equations with local alignment or diffusion effects, yielding Euler--alignment systems possibly augmented by pressure, see for instance \cite{CCJ21, CH24, CK23, FK19, KMT15}. Related results for nonlinear velocity couplings have been obtained more recently, including models with local alignment interactions, cf. \cite{BT25}. Complementary to the kinetic approach, mean-field limits directly connecting particle systems \eqref{CS} to hydrodynamic equations \eqref{EA} have also been established, both for regular communication kernels \cite{ACPSpre, CC21, Shv21} and for singular kernels, covering linear and nonlinear velocity couplings \cite{CFP25, FP24}. 

Beyond derivation results, the Euler--alignment system \eqref{EA} has been the subject of an extensive well-posedness and large-time analysis. Global existence, uniqueness, and asymptotic flocking behavior have been studied under various assumptions on the communication kernel and the initial configuration; see, among many others, \cite{CWZ19, CCTZ25, Cho19, CJu24, DMPW19, DKRT18, HT08, HKK15, KT18, LLST22,  LS19,  ST20, ST21,  ST17, ST18, Tad21, Tad23, TT14, WZpre}. A particularly striking feature of alignment models is the presence of
\emph{critical threshold phenomena}, whereby the global regularity or finite-time breakdown of solutions is determined by delicate structural conditions on the initial data; we refer to \cite{CCTT16, LTWpre, Tadpre, TT14, Tan20} and references therein. We also note that Lagrangian-based approaches have been explored in related
one-dimensional settings, including studies on Lagrangian trajectories \cite{Les20} and sticky particle dynamics for the one-dimensional Euler--alignment system \cite{Gal25, LT23}.

For a broader perspective on alignment models and their continuum descriptions, including particle, kinetic, and hydrodynamic viewpoints,  as well as related multiscale passages between microscopic, mesoscopic, and macroscopic descriptions, we refer to \cite{CCP17, CFTV10, CHL17, MT14,  MMPZ19, PTpre, Shv21, Shv24, Tad21, Tad23}, and the references therein.


 \subsection{Main results}
 The viewpoint adopted in this work differs in a fundamental way from the classical Eulerian approaches to alignment dynamics. Traditionally, Lagrangian trajectories have been introduced \emph{a posteriori} as characteristics associated with a given Eulerian velocity field, primarily as a tool to analyze long-time behavior. In this setting, to make sense of a characteristic description, one typically assumes that the velocity field is bounded and Lipschitz, and a substantial part of the analysis is therefore devoted to ensuring the well-posedness of the Eulerian system so that such a Lagrangian description is meaningful.
 
In contrast, we take the continuum Lagrangian system \eqref{Lag_p} as the primary object. Given an initial configuration $(\rho_0,u_0)$, the Lagrangian alignment formulation is naturally well-defined as a nonlocal mean-field ODE, even for very weak initial data: the spatial distribution $\rho_0$ may be an arbitrary probability measure,
and for the basic well-posedness theory, it suffices to assume $u_0\in L^\infty(\rho_0)$. Stronger regularity assumptions on $u_0$, such as $u_0\in W^{1,\infty}(\rho_0)$, will be imposed later when studying finer properties of the induced Eulerian dynamics, including injectivity of the flow and related closure mechanisms. In this way, the existence of Lagrangian solutions follows directly from the structure of \eqref{Lag_p}, without requiring \emph{a priori} regularity of an Eulerian velocity field. Thus, for example, our Lagrangian formulation enables us to trace the alignment dynamics of a ``blob'' subject to a discontinuous initial configuration $\displaystyle \rho_0=\frac{1}{|\Omega|}{\mathds 1}_\Omega$ supported in $\Omega \subset \R^d$.
 

 \subsubsection{Existence and long-time dynamics of Lagrangian $p$-alignment}
We begin by establishing global well-posedness and alignment properties of the continuum Lagrangian $p$-alignment system \eqref{Lag_p}. Our first result shows that, for essentially bounded initial velocities, the Lagrangian dynamics admit a unique global solution and satisfy quantitative bounds that describe their long-time flocking behavior.

We introduce the spatial and velocity diameters associated with solutions of \eqref{Lag_p} by
 \[
 \rd_\eta(t): = \esssup_{x,y \in \supp(\rho_0)}|\eta_t(x) - \eta_t(y)|, \quad  \rd_v(t): = \esssup_{x,y \in \supp(\rho_0)}|v_t(x) - v_t(y)|.
 \]
 
 The following theorem provides global well-posedness of the Lagrangian system and describes its long-time alignment behavior under suitable conditions.

 \begin{theorem}\label{thm:ext} Let  $p\ge2$ and suppose that $\phi:\R^d\to[0,\infty)$ is bounded and Lipschitz continuous. 
For any $u_0\in  L^\infty(\rho_0; \R^d)$, the Lagrangian $p$-alignment system \eqref{Lag_p}--\eqref{ini:Lag_p} admits a unique global solution
\bq\label{reg_Lag}
(\eta - {\rm id},v)\in C^1([0,\infty);L^\infty(\rho_0)\times  L^\infty(\rho_0))
\eq
satisfying
\[
\|v_t\|_{L^\infty(\rho_0)} \le \|u_0\|_{L^\infty(\rho_0)},  \quad  \|\eta_t-{\rm id}\|_{L^\infty(\rho_0)} \le t \|u_0\|_{L^\infty(\rho_0)} \quad\text{for all }t\ge0.
\]
 Moreover, assume that $\rho_0$ is compactly supported and the communication kernel $\phi$ is ``heavy-tailed'' in the sense that
\bq\label{condi_phi}
\int \limits^\infty \phi(r)\,\dr = \infty.
\eq
Then the velocity diameter asymptotically vanish, i.e.,
$\rd_v(t) \to 0 \  \text{as } t \to \infty$.
 \end{theorem}

 \begin{remark}If $\supp\rho_0$ is bounded, then ${\rm id}\in L^\infty(\rho_0;\R^d)$ and hence 
\[
\eta={\rm id}+(\eta-{\rm id})\in C^1([0,\infty);L^\infty(\rho_0)).
\] 
In this case, one can state the result equivalently as 
\[
(\eta,v)\in   C^1([0,\infty);L^\infty(\rho_0)\times  L^\infty(\rho_0)),
\]
with the same a priori bounds.
\end{remark}

\begin{remark} The well-posedness result in Theorem \ref{thm:ext} is established entirely within the Lagrangian framework, without employing any Eulerian regularity theory. The solution class \eqref{reg_Lag} involves only time regularity and essential boundedness with respect to the Lagrangian variable $x$. In particular, no spatial derivatives with respect to the Lagrangian coordinate (i.e., no assumptions on $\nabla_x\eta$ or $\nabla_x v$) are required or generated in the analysis.

All terms in \eqref{Lag_p} are well-defined for essentially bounded Lagrangian fields, since the alignment operator depends only on velocity differences $v(y)-v(x)$ and the composed interaction kernel $\phi(\eta(x)-\eta(y))$, integrated against $\rho_0(\dy)$. Thus, the construction does not rely on the flow map $\eta(t,\cdot)$ being differentiable or invertible.

In this sense, the solution of Theorem \ref{thm:ext} should be regarded as a \emph{classical solution in the ODE sense}: the curves $t\mapsto \eta_t-{\rm id}$ and $t\mapsto v_t$ belong to $C^1([0,\infty);L^\infty(\rho_0))$, and the Lagrangian system is satisfied pointwise for $\rho_0$-a.e. initial label $x$.
\end{remark}

\begin{remark} 
The Lipschitz continuity of $\phi$ in Theorem \ref{thm:ext} is used only to guarantee that the alignment operator is locally Lipschitz with respect to the Lagrangian variables $(\eta,v)$, which in turn yields uniqueness via the Picard--Lindel\"of theorem. If the kernel $\phi$ is assumed to be merely bounded (or even just measurable), all a priori estimates in the proof of Theorem \ref{thm:ext} remain valid, since they rely only on the boundedness of $\phi$ and the monotone structure of the alignment operator. In particular, the maximum principle for the velocity and the global-in-time
$L^\infty$ bounds for $(\eta,v)$ continue to hold.

However, without Lipschitz regularity of $\phi$, the alignment operator is no longer locally Lipschitz in $(\eta,v)$; hence classical ODE uniqueness theory does not apply.  In this context of ODE dynamics ``beyond-Lipschitz'', it would be interesting to investigate  the broader Sobolev/BV framework of DiPerna--Lions theory and its later extensions \cite{DL89,Amb04}, and its implications in the nonlocal setting of \eqref{Lag_p} for non-Lipschitz kernels. Since this lies beyond the scope of the present work, we do not pursue it here.
\end{remark}

\begin{remark}
For $p\ge3$, the flocking estimate in Theorem \ref{thm:ext} does not, in general, yield a uniform-in-time bound on the spatial diameter $\rd_\eta(t)$. In particular, $\rd_\eta(t)$ may grow in time; see, for instance, \cite{BT24,NSTpre}. Accordingly, for a general communication kernel, Theorem \ref{thm:ext} only yields the ``qualitative'' alignment
\[
\rd_v(t)\to0 \quad\text{as }t\to\infty,
\]
without an explicit decay rate in terms of $t$.

On the other hand, when $p \in [2,3)$, if the initial data satisfy
  \bq\label{eq:conditional-flocking}
\frac1{3-p} \rd_v^{3-p}(0) <  2^{2-p}\kappa \int_{\rd_\eta(0)}^\infty \phi(r)\,\dr,
 \eq
 then there exists $\rd_\eta^\infty > 0$ uniquely determined by
 \bq\label{rel_deta}
\frac1{3-p} \rd_v^{3-p}(0) =  2^{2-p}\kappa \int_{\rd_\eta(0)}^{\rd_\eta^\infty} \phi(r)\,\dr
 \eq
such that
\[
\sup_{t \geq 0} \rd_\eta(t) \leq \rd_\eta^\infty
\]
and the velocity diameter decays with the explicit bounds 
 \bq\label{decay_v_sub3}
\rd_v(t) \leq  \left\{ \begin{array}{ll}
\displaystyle \rd_v(0) e^{- \kappa \phi(\rd_\eta^\infty) t} & \textrm{if $p=2$},\\[2mm]
\displaystyle  \lt(\rd_v^{2-p}(0) + (p-2) 2^{2-p} \kappa \phi(\rd_\eta^\infty) t \rt)^{-\frac1{p-2}}& \textrm{if $p\in(2,3)$}.
  \end{array} \right.
 \eq
Thus, for $p\in[2,3)$, the above threshold condition already yields both velocity alignment and a uniform-in-time bound on the spatial diameter, without requiring the heavy-tail assumption \eqref{condi_phi}. The latter simply secures \emph{unconditional flocking} in the sense that \eqref{eq:conditional-flocking} and hence \eqref{rel_deta} hold for all initial configurations, independent of how large $\rd_v(0)$ is. In particular, the heavy-tail condition is only a convenient sufficient condition in this regime. See Remark \ref{rmk:p-sub3} for a more detailed discussion.
 
More generally, for $p>2$, the decay rate is governed by the growth of
\[
\int_0^t \phi(\rd_\eta(s))\,\ds.
\]
In particular, if $\phi$ is bounded from below by a positive constant on $[0,\infty)$, then
\[
\int_0^t \phi(\rd_\eta(s))\,\ds \gtrsim t,
\quad\text{and hence}\quad
\rd_v(t)\lesssim (1+t)^{-\frac1{p-2}}.
\]
For more slowly decaying kernels, for instance $\phi(r)=(1+r)^{-\beta}$ with $\beta<1$, one correspondingly obtains slower algebraic rates, such as
\begin{equation}\label{eq:dv_decay}
\rd_v(t)\lesssim (1+t)^{-\frac{1-\beta}{p-2}}.
\end{equation}
This is consistent with previously known algebraic flocking estimates for mono-kinetic $p$-alignment systems; see, for instance, \cite{Tad23, NSTpre}, and the references therein. Indeed, introducing
\[
\delE(t):=\intrr |v_t(x)-v_t(y)|^2\,\rho_0(\dx)\rho_0(\dy),
\]
one obtains from \cite[Proposition 5.2]{Tad23} an algebraic decay estimate for $\delE(t)$ under heavy-tail communication kernels, which is compatible with the algebraic decay estimate for the velocity diameter in \eqref{eq:dv_decay} appearing in \cite{NSTpre}.
\end{remark}



 \subsubsection{Lagrangian--Eulerian correspondence: Euler--Reynolds--alignment system}\label{ssec:ERA}
 
Our second main result concerns the Eulerian description canonically induced by the global Lagrangian flow constructed in Theorem \ref{thm:ext}.  The guiding principle is that Eulerian
quantities should be obtained \emph{without} imposing any \emph{a priori} spatial regularity on the flow map.  Instead, they are defined through two purely measure-theoretic operations: pushing forward the reference measure $\rho_0$ by the flow map $\eta_t$, and disintegrating $\rho_0$ along the fibres of $\eta_t$.

This viewpoint is essential in the present setting, since the global Lagrangian flow is constructed only in an $L^\infty$ framework and may fail to be injective. As a consequence, classical pointwise Eulerian formulas are no longer available. The disintegration of $\rho_0$ provides a fundamental way to recover Eulerian quantities from the Lagrangian dynamics, even in the presence of such non-injectivity. We develop this correspondence in detail in Section \ref{ssec:LEc}; here we only recall the minimal definitions needed to state the result. We emphasize that all objects introduced above are uniquely determined by the Lagrangian flow $(\eta_t,v_t)$ up to $\rho_t$-null sets, and no additional closure assumptions are imposed at the Eulerian level.

Given the Lagrangian variables $(\eta_t,v_t)$, we define the Eulerian density and momentum by pushforward:
\[
\rho_t := (\eta_t)_\#\rho_0, \quad m_t := (\eta_t)_\#(v_t\,\rho_0),
\]
so that $m_t\ll \rho_t$ and the barycentric velocity
\[
u_t := \frac{\rd m_t}{\rd\rho_t}
\]
is well defined. To account for the possible multiplicity of Lagrangian labels mapping to the same Eulerian position, we disintegrate $\rho_0$ along the fibres of $\eta_t$ (in the sense of \cite[Theorem 5.3.1]{AGS08}), yielding a family of probability measures $\{\nu_{t,z}\}_{z}$ concentrated on $\eta_t^{-1}(\{z\})$.  This allows us to represent $u_t$ as the fibre average of $v_t$ and to define a nonnegative Reynolds stress as the corresponding fibre variance:
\[
\tau_t := \rho_t \theta_t, \quad \theta_t(z):=\intr (v_t(x)-u_t(z))\otimes(v_t(x)-u_t(z))\,\nu_{t,z}(\dx) \succeq 0.
\]
In particular, the Reynolds stress $\tau_t$ is \emph{not} an additional unknown: it is uniquely induced by the Lagrangian flow (modulo $\rho_t$-null sets), and $\tau_{t=0}=0$.

The next theorem shows that the Eulerian triple $(\rho_t,u_t,\tau_t)$ satisfies an Euler--Reynolds--alignment (ERA) system.  Besides the Reynolds stress, a genuinely nonlinear \emph{defect force} appears for $p>2$, reflecting the interaction between the nonlinearity $G_p(\xi)=|\xi|^{p-2}\xi$ and the fibre fluctuations of $v_t$.
 
\begin{theorem} \label{thm:ERA}
Assume $p \geq 2$ and the hypotheses of Theorem \ref{thm:ext}. Let $(\rho_t, u_t, \tau_t)$ be the Eulerian variables induced by the Lagrangian solution $(\eta_t,v_t)$ through pushforward and disintegration as described above. Then $(\rho_t, u_t, \tau_t)$ is a global solution to 
\begin{align}\label{eq:ERA}
\begin{aligned}
\pa_t\rho_t+\nabla \cdot(\rho_t u_t)&=0, \\
\pa_t(\rho_t u_t)+\nabla \cdot(\rho_t u_t\otimes u_t + \tau_t)
&=\rho_t\lt(\calA_p[\rho_t,u_t]+\calR_p[\rho_t,u_t]\rt), 
\end{aligned}
\end{align}
in the sense of distributions on $(0,\infty)\times\R^d$, with initial data 
\[
(\rho_t, u_t, \tau_t)|_{t=0} = (\rho_0, u_0, 0),
\]
and satisfying\footnote{The continuity statements 
\[
\rho\in C([0,\infty);\calP(\R^d)),\quad
m\in C([0,\infty);\calM(\R^d;\R^d)) 
\]
refer to the narrow topology on $\calP(\R^d)$ and the weak$^\ast$
topology on the indicated spaces of Radon measures. In particular, a strong time-regularity of $u_t$ or $\theta_t$ is not implied.
}
\[
\rho \in C([0,\infty);\calP(\R^d)),\quad m=\rho u \in C([0,\infty);\calM(\R^d;\R^d)), \quad \tau \in L^\infty_{\rm loc}([0,\infty);\calM(\R^d;\R^{d \times d})).
\]
Here, the nonlinear defect force is
\[
\calR_p[\rho_t,u_t](z) :=\kappa\intr\phi(z-\zeta) \calK_t(z,\zeta)\,\rho_t(\rd\zeta),
\]
where, denoting by $\{\nu_{t,z}\}$ any disintegration of $\rho_0$ along $\eta_t$, and $\omega_t(x):=v_t(x)-u_t(\eta_t(x))$ which has zero fibre mean, i.e., $\intr \omega_t\,\nu_{t,z}(\dx)=0$, we set
\bq\label{calK}
\calK_t(z,\zeta) := \intrr   G_p(u_t(\zeta)-u_t(z)+\omega_t(y)-\omega_t(x)) \,\nu_{t,\zeta}(\dy) \nu_{t,z}(\dx) - G_p(u_t(\zeta)-u_t(z)).
\eq

Moreover, $(\rho,u,\tau)$ satisfies the global energy inequality:
for all $t\geq 0$,
\begin{align}\label{ERA_energy}
\begin{aligned}
&\frac12 \intr \left(|u_t|^2+{\rm tr}\,\theta_t(z)\right)\rho_t(\dz)\cr
&\quad +\frac{\kappa}{2} \int_0^t\iiiint_{\R^d \times \R^d \times \R^d \times \R^d} \phi(z-\zeta)\,|v_s(y)-v_s(x)|^p \,\nu_{s,z}(\dx)\nu_{s,\zeta}(\dy) \rho_s(\dz)\rho_s(\rd\zeta)\, \ds \cr
&\qquad \le \frac12\intr |u_0|^2 \,\rho_0(\dz).
\end{aligned}
\end{align}
In the class of Eulerian triples arising from the Lagrangian flow of Theorem \ref{thm:ext}, this solution is unique.
\end{theorem}

In the classical theory of Euler equations, the appearance of a Reynolds-type stress tensor is reminiscent of classical Euler--Reynolds formulations arising in the analysis of weak limits of nonlinear transport equations. In the context of incompressible and compressible Euler equations, Reynolds stresses typically emerge either as defect measures associated with oscillations and
concentrations in sequences of approximate solutions, or as auxiliary unknowns introduced to relax the Euler system in underdetermined formulations \cite{BDIS15, DS13, DS14, DiP85, LPT94}.
In such settings, the Reynolds stress is not prescribed by the dynamics itself, but reflects a lack of strong compactness or is intentionally retained as a degree of freedom in the construction of weak solutions.

System \eqref{eq:ERA} should be compared with the Eulerian $p$-alignment system  developed in \cite{Tad23}. When $p=2$, \eqref{eq:ERA} coincides with the hydrodynamic $2$-alignment of \cite{Tad23}
expressed in terms of  $(\rho_t,u_t,{\mathbb P}_t)$,
 where ${\mathbb P}_t\succeq 0$  is a pressure tensor corresponding to Reynolds stress $\tau_t$, and initiated with mono-kinetic data, $(\rho_0,u_0,0)$. When $p>2$ there is an additional $p$-moment term (denoted ${\calI}_2$ in \cite[eq. (A.10)]{Tad23} and absorbed into the ``heat'' term ${\mathbf q}_\phi$), which corresponds to the nonlinear defect forcing $\rho_t{\calR}_t[\rho_t,u_t]$. In particular, the global energy inequality \eqref{ERA_energy} is consistent with the notion of dissipative solution (or entropic pressure) of the total energy $E=\frac{1}{2}|u_t|^2 + e_{{\mathbb P}}$ in \cite[eq. (1.12)]{Tad23}, where ${\rm tr}\,\theta_t$ coincides with the \emph{internal energy}, $ {\rm tr}\,\tau_t \rightarrow {\rm tr}\,{\mathbb P}_t =2\rho_t e_{{\mathbb P}}$.
 In this Eulerian framework, however, one lacks a closure for the pressure tensor ${\mathbb P}_t$.   

The ERA system \eqref{eq:ERA} differs in a fundamental way from these Eulerian formulations. The Reynolds stress $\tau_t$ is neither an independent unknown nor a modeling assumption: it is canonically generated by the exact Lagrangian flow through fibrewise velocity fluctuations induced by possible non-injectivity of the flow map. In particular, $\tau_t$ is uniquely determined by the Lagrangian dynamics (up to $\rho_t$-null sets), is nonnegative by construction, and vanishes initially as a direct consequence of $\eta_0={\rm id}$. As a result, the ERA system is \emph{not} underdetermined and requires no additional closure or admissibility criteria at the Eulerian level.  

In this respect, our formulation should also be compared with the recent work \cite{WZpre}, where global measure solutions to a pressureless Euler--alignment system are constructed through a vanishing-viscosity limit of a degenerate Navier--Stokes approximation, and the lack of strong compactness in the convective term is encoded by a matrix-valued concentration defect measure. In the linear case $p=2$, this defect is reminiscent of the Reynolds-type stress $\tau_t$ appearing in \eqref{eq:ERA}. However, the two viewpoints are conceptually different: in \cite{WZpre}, the defect arises as an \emph{a posteriori} weak-limit object associated with the convergence of approximate solutions, whereas in the present work the stress tensor is derived canonically from the exact Lagrangian flow through fibrewise disintegration of the reference measure. In particular, our approach identifies the defect as the fibre variance of the Lagrangian velocity explicitly, showing that it is not an additional unknown but a uniquely induced quantity determined by the underlying Lagrangian dynamics. The measure-theoretic construction of $\tau_t$ is established in Section \ref{ssec:LEc}; in particular, $\tau_t$ is nonnegative and does not depend on the particular representative of the disintegration.
 
This structural distinction shifts the analytical focus from the construction of Reynolds stresses to their \emph{dynamical evolution}. The central question is therefore not how to select $\tau_t$, but under which mechanisms it disappears. In the present framework, the vanishing of $\tau_t$ is tied to the collapse of the fibrewise velocity fluctuations in the disintegration of $\rho_0$ along $\eta_t$. In particular, whenever the flow remains injective, the fibres are singletons and $\tau_t$ vanishes identically. More generally, if the
disintegration is Dirac for almost every time, then $\tau_t=0$ for almost every $t$, and the Eulerian dynamics close to a mono-kinetic Euler--alignment system in the distributional sense. In regimes where alignment mechanisms enforce such injectivity, the Reynolds stress thus represents only a transient manifestation of microscopic velocity fluctuations, rather than a persistent defect or a free relaxation variable. 
 
\begin{remark}For $p=2$, the nonlinear defect force $\calR_p$ vanishes identically.  Indeed, since $G_2$ is linear, we have
\[
\calK_t(z,\zeta) = \intrr (\omega_t(y)-\omega_t(x)) \,\nu_{t,\zeta}(\dy) \nu_{t,z}(\dx) =0,
\]
since the fibre fluctuations satisfy $\intr \omega_t\,\nu_{t,z}(\dx)=0$. Hence, $\calR_2 \equiv 0$ and $(\rho, u, \tau)$ satisfies the Euler--alignment formulation with Reynolds stress:
\begin{align*}
\pa_t\rho_t+\nabla \cdot(\rho_t u_t)&=0, \\
\pa_t(\rho_t u_t)+\nabla \cdot(\rho_t u_t\otimes u_t + \tau_t)
&=\kappa\rho \intr \phi(z-\zeta) (u_t(\zeta)-u_t(z))\,\rho_t(\rd\zeta), 
\end{align*}
in the sense of distributions on $(0,\infty)\times\R^d$. 

In this case, the dissipation also admits a particularly clean Eulerian decomposition.  Writing $v_t(x)=u_t(\eta_t(x))+\omega_t(x)$ with zero fibre mean $\intr \omega_t\,\nu_{t,z}(\dx)=0$, the identity
\[ 
|v_t(x)-v_t(y)|^2 = |u_t(z)-u_t(\zeta)|^2 + 2(u_t(z)-u_t(\zeta)) \cdot(\omega_t(x)-\omega_t(y)) + |\omega_t(x)-\omega_t(y)|^2
\] 
shows that the mixed term vanishes after integrating against $\nu_{t,z}(\dx)\nu_{t,\zeta}(\dy)\rho_t(\dz)\rho_t(\rd\zeta)$.  Hence, the total dissipation splits into the sum of a macroscopic Euler--alignment part,
\[
D_{\rm{EA}}(t) =\kappa \intrr \phi(z-\zeta) |u_t(z)-u_t(\zeta)|^2\,\rho_t(\dz)\rho_t(\rd\zeta),
\]
and a purely microscopic Reynolds contribution,
\[
D_{{\rm Rey}}(t) =\kappa \iiiint_{\R^d \times \R^d \times \R^d \times \R^d} \phi(z-\zeta) |\omega_t(x)-\omega_t(y)|^2\,\nu_{t,z}(\dx)\nu_{t,\zeta}(\dy)\,\rho_t(\dz)\rho_t(\rd\zeta).
\]
Thus, when $p=2$, both the forcing and the dissipation exhibit a clean separation between macroscopic alignment effects and microscopic fibre fluctuations encoded in $\tau_t$. In particular, we have
\[
\frac12 \intr |u_t|^2 \rho_t(\dz) + \frac\kappa2\int_0^t \intrr \phi(z-\zeta)\,|u_t(z)-u_t(\zeta)|^2 \rho_t(\dz)\rho_t(\rd\zeta)\ds \leq \frac12\intr |u_0|^2 \rho_0(\dz).
\]
\end{remark}
 
\begin{remark} 
It is sometimes convenient to lift the Lagrangian dynamics to phase space by considering the pushforward measure
\[
\mu_t := (\eta_t,v_t)_\#\rho_0 \in \calP(\R^d\times\R^d).
\]
By construction, $\rho_t$ is the spatial marginal of $\mu_t$, while its first and second velocity moments recover the barycentric velocity $u_t$ and the Reynolds stress $\tau_t$.

Moreover, $\mu_t$ satisfies a kinetic transport equation associated with the alignment dynamics, namely \eqref{VA}, in the sense of distributions. This phase-space formulation provides a natural measure-valued description of the ERA system.

While we postpone the detailed derivation of this kinetic formulation to Appendix \ref{app:kinetic}, we emphasize that the lifted measures $\mu_t$ will play a crucial role in the subsequent asymptotic closure analysis. Their compactness properties under time translation allow us to extract limiting kinetic measures, whose structure, together with the decay of the velocity diameter, leads to the vanishing of the Reynolds stress and the nonlinear defect force under flocking.
\end{remark}



\subsubsection{Asymptotic closure: from Euler--Reynolds--alignment to Euler--alignment under flocking}
A central issue in the hydrodynamic description of alignment dynamics is whether, and in what sense, the macroscopic equations asymptotically reduce to a mono-kinetic Euler--alignment system.
 The ERA system obtained in Theorem \ref{thm:ERA} is, in general, not closed at the macroscopic level: the momentum balance involves the Reynolds stress $\tau_t$ and, for $p>2$, a genuinely nonlinear defect force $\calR_p[\rho_t,u_t]$. These two terms encode complementary microscopic effects. The tensor $\tau_t$ measures the fibrewise velocity dispersion generated by possible non-injectivity of the Lagrangian flow, while $\calR_p$ quantifies the mismatch between the nonlinear alignment interaction $G_p(\xi)=|\xi|^{p-2}\xi$ and its barycentric approximation.

Our third main result shows that both defect mechanisms are \emph{asymptotically suppressed by velocity flocking alone}. More precisely, we prove that the decay of the velocity diameter $\rd_v(t)$ forces both the Reynolds stress and the nonlinear defect force to vanish as $t\to\infty$, yielding an asymptotic closure of the ERA system to the Eulerian $p$-alignment dynamics \eqref{EA}. Notably, this closure requires no spatial confinement, no global injectivity of the Lagrangian flow map $\eta_t$, and no additional Eulerian regularity assumptions: the sole driving mechanism is the large-time alignment of velocities.

\begin{theorem} \label{thm:asymp_closure}
Let $(\rho_t,u_t,\tau_t)$ be the Eulerian fields associated with a global  Lagrangian solution constructed in Theorem \ref{thm:ext}.  Assume that $\rho_0$ is compactly supported and the communication kernel satisfies the heavy-tail condition \eqref{condi_phi}. Then the following hold.

\medskip
\noindent
\textbf{(Asymptotic closure)} The Reynolds stress vanishes asymptotically:
\[
\|\tau_t\|_{\calM(\R^d;\R^{d\times d})}\to0 \quad\text{as }t\to\infty.
\]
In addition, when $p>2$, the nonlinear defect force also vanishes asymptotically:
\[
\|\rho_t \calR_p[\rho_t,u_t]\|_{\calM(\R^d;\R^d)}\to0 \quad\text{as }t\to\infty.
\]
Consequently, the defect terms in the ERA system vanish as $t\to\infty$, and the dynamics become asymptotically consistent with the mono-kinetic Euler--alignment system \eqref{EA}.

\medskip
\noindent
\textbf{(Asymptotic flocking)} The diameter of Eulerian velocities decays to zero in the essential sense:
\[
\esssup_{(z,\zeta)\in S_t\times S_t}|u_t(z)-u_t(\zeta)|
\to 0 \quad \text{as }t\to\infty, \qquad S_t :=\text{supp}\,\rho_t.
\]

\medskip
\noindent
\textbf{(Asymptotic mono-kinetic dynamics)}
Let $T>0$ and let $\mu_t=(\eta_t,v_t)_\#\rho_0$ denote the lifted kinetic measure. Then there exist a sequence $R_n\to\infty$ and a family of probability measures $\{\rho_t^\ast\}_{t\in[0,T]}\subset\calP(\R^d)$ such that
\[
\mu_{R_n+t} \rightharpoonup \rho_t^\ast(\rd z)\,\delta_{\bar u}(\rd\xi) \quad
\text{in }\calM_{\mathrm{loc}}(\R^d\times\R^d)
\]
for every $t\in[0,T]$, where
\[
\bar u=\intr u_0(x)\rho_0(\rd x).
\]
Moreover,
\[
\rho_t^\ast=(z\mapsto z+t\bar u)_\#\rho_0^\ast,
\]
where $\rho_0^\ast\in\mathcal P(\R^d)$ denotes the weak limit of the spatial marginals at the initial shifted time.
\end{theorem}

\begin{remark} The limiting density is transported by the constant velocity $\bar u$ so that
\[
\pa_t\rho_t^\ast+\nabla\cdot(\rho_t^\ast \bar u)=0
\]
holds in the sense of distributions. Moreover, since $\bar u$ is spatially constant, the alignment interaction vanishes identically:
\[
\calA_p[\rho_t^\ast,\bar u]=0.
\]
Hence $(\rho_t^\ast,\bar u)$ is a global distributional solution to the Eulerian $p$-alignment system \eqref{EA}.
\end{remark}

\begin{remark}
The heavy-tail condition \eqref{condi_phi} in Theorem \ref{thm:asymp_closure} is assumed only as a sufficient condition ensuring, through Theorem \ref{thm:ext}, that the associated Lagrangian solution satisfies
\[
\rd_v(t)\to0 \quad\text{as }t\to\infty.
\]
Accordingly, the conclusions of Theorem \ref{thm:asymp_closure} remain valid for any global Lagrangian solution enjoying this asymptotic alignment property.
\end{remark}

\begin{remark}
The asymptotic closure mechanism in Euler $p$-alignment systems was originally developed  in \cite[Corollary 4.2]{Tad23} and \cite[Theorem 4]{Tadpre}. It states that if $(\rho_t,u_t,{\mathbb P}_t)$ is a weak $p$-alignment solution satisfying the dissipative energy inequality and the  heavy-tail condition, then one obtains the asymptotic emergence of monokinetic closure, namely,
\[
\intr \tr \,\mathbb P_t \,\dx \to 0  \quad \text{as }t\to\infty.
\]
The decay of velocity fluctuations in the linear velocity alignment case $p=2$ has been further investigated in \cite{Tadpre} through an entropy-based approach, establishing convergence toward mono-kinetic closure under the structural condition of uniform thickness. While the analysis in \cite{Tad23,Tadpre} focuses on velocity fluctuations of dissipative solutions, our approach is fundamentally different: the closure here is obtained at the level of weak, measure-valued limits induced by the Lagrangian flow. In this sense, Theorem \ref{thm:asymp_closure} provides a complementary mechanism for mono-kinetic emergence, driven purely by alignment, which collapses the fibrewise velocity fluctuations of both $\tau_t$ and $\rho_t\calR_p$ carried by the Lagrangian transport.
\end{remark}

A key structural ingredient in our analysis is the measure-valued compactness framework introduced in \cite{DiP85} and further developed in the kinetic formulation of \cite{LPT94}. In this perspective, weak limits of nonlinear transport dynamics are naturally described by Young measures: possible microscopic oscillations in the velocity variable are encoded by a probability kernel whose barycenter produces the limiting velocity, while residual oscillations appear as a nonnegative covariance tensor of Reynolds type. This barycentric--covariance decomposition is consistent with the structure emphasized in the computational theory of measure-valued solutions developed in \cite{FMT16}. 

In the present setting, this viewpoint is naturally implemented by lifting the Lagrangian dynamics to phase space through an associated kinetic measure $\mu_t$, whose spatial and velocity moments recover the Euler--Reynolds variables. This kinetic representation provides the compactness framework needed to pass to the long-time limit. Under velocity flocking, the associated velocity Young measures collapse to Dirac masses as $t\to\infty$, which forces both the Reynolds stress and the nonlinear defect force to vanish. As a result, any such long-time limit is described by a mono-kinetic Eulerian $p$-alignment state of the form $(\rho^\ast, \bar u)$.



\subsubsection{Global weak solutions to the Euler--alignment system} 
Our next results concern global-in-time weak solutions in the case of linear velocity alignment, that is, $p=2$. In this regime, the ERA system simplifies substantially, since the nonlinear defect force vanishes identically, $\calR_2 \equiv 0$. As a consequence, the only possible obstruction to a closed Eulerian description is the Reynolds stress $\tau_t$, which encodes the loss of injectivity of the Lagrangian flow through fibrewise velocity fluctuations. 
 
Recall that any Lagrangian solution to the alignment dynamics induces, via pushforward and disintegration along the flow map $\eta_t$, an Eulerian triple $(\rho_t,u_t,\tau_t)$ satisfying the
ERA system \eqref{eq:ERA}. If the flow $\eta_t$ remains injective, then the fibres $\eta_t^{-1}(\{z\})$ reduce to singletons and no microscopic velocity fluctuations can occur.
In this case, the Reynolds stress vanishes identically, $\tau_t\equiv0$, and the Eulerian dynamics close to the classical Euler--alignment system
\begin{align}\label{eq:EA-2}
\begin{aligned}
\pa_t\rho_t+\nabla \cdot(\rho_t u_t)&=0, \\
\pa_t(\rho_t u_t)+\nabla \cdot(\rho_t u_t\otimes u_t) &=\kappa \rho_t\intr \phi(z-\zeta)(u_t(\zeta)-u_t(z))\,\rho_t(\rd\zeta).
\end{aligned}
\end{align}
 
We establish global weak solutions to \eqref{eq:EA-2} by exploiting two distinct mechanisms, depending on the spatial dimension. In one spatial dimension, the Lagrangian formulation yields a sharp and complete characterization of injectivity, leading to an exact subcritical--supercritical dichotomy. In higher dimensions, we derive a conditional global existence result based on quantitative control of the Lagrangian deformation.

To state the one-dimensional result, let $\Phi$ be a $C^2$ primitive of the communication kernel $\phi$, and define the effective initial velocity by
\[
\widehat v(x) := u_0(x)-\kappa \int_\R \Phi(y-x)\,\rho_0(\dy).
\]

The following theorem provides a sharp dichotomy between global mono-kinetic Euler--alignment dynamics and the non-injective regime in one dimension.

\begin{theorem}\label{thm:1DEA}
Let $d=1$, and assume $p =2$ and $u_0 \in L^\infty(\rho_0)$. Let $(\rho_t, u_t, \tau_t)$ be the Eulerian variables induced by the Lagrangian solution $(\eta_t,v_t)$ of \eqref{Lag_p} through pushforward and disintegration as in Section \ref{ssec:ERA}. Then the following hold.

\medskip
\noindent\textbf{(Subcritical region: mono-kinetic regime)}  
If the effective velocity $\widehat v$ is non-decreasing, then $(\rho,u) \in C([0,\infty);\calP(\R)) \times L^\infty([0,\infty); L^\infty(\rho))$ is a global weak solution of the mono-kinetic 1D
Euler--alignment system 
\begin{align}\label{1D_EA}
\begin{aligned}
\pa_t \rho_t + \partial_x(\rho_t u_t) &= 0, \\
\pa_t(\rho_t u_t) + \partial_x(\rho_t u_t^2)
&= \kappa\,\rho_t \int_\R \phi(z-\zeta)(u_t(\zeta)-u_t(z))\,\rho_t(\rd\zeta),
\end{aligned}
\end{align}
for all $t\ge0$.

\medskip
\noindent\textbf{(Supercritical region: collision-induced Reynolds regime)} 
If $\widehat v$ is not non-decreasing, then collisions occur in finite time, and the Lagrangian flow loses injectivity. In this regime, the associated Eulerian description is given by the Euler--Reynolds--alignment system of Theorem \ref{thm:ERA}, and a nontrivial Reynolds stress may occur.
\end{theorem}

 \begin{remark} 
If $u_0$ is $C^1$ on the support of $\rho_0$, then
\[
\widehat v'(x)=u_0'(x)+\kappa(\phi*\rho_0)(x).
\]
Hence, the condition $\widehat v'(x)\ge0$ coincides exactly with the sharp subcritical threshold for global $C^1$ regularity of the 1D Euler--alignment system established in \cite{CCTT16}. When this condition fails, the Eulerian $C^1$ theory predicts finite-time blow-up of the velocity gradient $\pa_x u$. 

 From the Lagrangian viewpoint, however, this loss of regularity is explained by the loss of injectivity of the characteristic flow. The Lagrangian flow itself remains globally defined, but the associated Eulerian description may develop nontrivial Reynolds defects once different Lagrangian labels reach the same Eulerian position with different velocities. Thus, failure of the subcritical condition does not destroy the global Lagrangian dynamics, but it may prevent mono-kinetic closure at the Eulerian level.
\end{remark}

Global weak solutions to \eqref{eq:EA-2} in one spatial dimension have been constructed in several settings beyond the breakdown of classical solutions, that is, precisely in the \emph{super-critical 1D regime} where collisions occur and the mono-kinetic description can no longer be continued classically. For the Euler--alignment system, a global well-posedness theory for measure-valued weak solutions was established in \cite{LT23} by introducing an entropic selection principle through an associated scalar balance law, together with an approximation by sticky particle Cucker--Smale dynamics. This entropy-based characterization was subsequently complemented by a gradient-flow formulation in \cite{Gal25}, where the same sticky particle dynamics were identified as the unique $L^2$-gradient flow and shown to be equivalent to entropy solutions in one dimension. Related gradient-flow and Lagrangian formulations for pressureless Euler and Euler--Poisson systems were shown to be equivalent to entropy solutions in \cite{CGpre} and references therein, providing canonical continuations beyond collisions.

Theorem \ref{thm:1DEA} provides a complementary perspective in the one-dimensional setting. Rather than selecting a distinguished weak solution among many admissible continuations in the super-critical 1D regime, it identifies the subcritical regime in which the Lagrangian flow remains injective for all times, so that the induced Eulerian dynamics is genuinely mono-kinetic. Outside this regime, collisions may occur, and the Eulerian description is naturally given by the Euler--Reynolds--alignment system.
 
We next turn to the higher-dimensional case, where such a sharp characterization is no longer available. Instead, under a sufficiently large coupling strength, we derive a conditional global existence result based on quantitative control of the Lagrangian deformation.
 
 \begin{theorem} \label{thm:EA_global}
Let $d \geq 1$, and assume $p=2$ and $u_0 \in W^{1,\infty}(\rho_0)$. Let $(\rho_t, u_t)$ be the Eulerian variables induced by the Lagrangian solution $(\eta_t,v_t)$ of \eqref{Lag_p} through pushforward and disintegration as in Section \ref{ssec:ERA}. Suppose that:\newline
(i)   $\rho_0$ is compactly supported;\newline
(ii) $\phi$ satisfies the heavy-tail condition \eqref{condi_phi}; and\newline
(iii)  the coupling strength $\kappa > 0$ is large enough so that
 \begin{equation}\label{eq:large_kappa}
 \kappa >  \frac{\phi(\rd_\eta^\infty)\|\nabla u_0\|_{L^\infty(\rho_0)} + 2\|\nabla\phi\|_{L^\infty} \rd_v(0) }{\phi^2(\rd_\eta^\infty)},
 \end{equation}
where $\rd_\eta^\infty > 0$ is given by the relation \eqref{rel_deta}.\newline
 Then  $(\rho, u)$ is a global-in-time solution to \eqref{eq:EA-2} in the sense of distributions on $(0,\infty)\times\R^d$, with initial data 
\[
(\rho_t, u_t)|_{t=0} = (\rho_0, u_0).
\]
Moreover, we have
\[
\rho \in C([0,\infty);\calP(\R^d))  \quad u \in L^\infty(0,\infty; L^\infty(\rho)).
\]
In the class of Eulerian pairs arising from the Lagrangian flow of Theorem \ref{thm:ext}, this solution is unique.
\end{theorem}

\begin{remark}
If $\phi$ is integrable, then the condition \eqref{eq:large_kappa} should be modified to include  the effect of its ``short tail''
\[
\kappa > \max\lt\{\frac{\rd_v(0)}{\int_{\rd_\eta(0)}^\infty \phi(r)\,\dr}, \, \frac{\phi(\rd_\eta^\infty)\|\nabla u_0\|_{L^\infty(\rho_0)} + 2\|\nabla\phi\|_{L^\infty} \rd_v(0) }{\phi^2(\rd_\eta^\infty)} \rt\}.
\]
\end{remark}

 \begin{remark} 
The global result of Theorem \ref{thm:EA_global} is specific to the case of linear velocity alignment $p=2$. Nevertheless, for general nonlinear couplings $p\ge2$ and initial data
$u_0\in W^{1,\infty}(\rho_0)$, the Lagrangian formulation developed in this work still yields local-in-time weak solutions to the Eulerian $p$-alignment system, as stated in Theorem \ref{thm:EA_local}.

The restriction to $p=2$ in the global existence theory is not due to a lack of well-posedness of the underlying Lagrangian dynamics, but rather to the difficulty of obtaining global-in-time control of the velocity gradient. In the linear case, the gradient system closes with a constant damping rate and an integrable source term, which allows us to prove that
\[
\int_0^\infty \|\nabla v_t\|_{L^\infty(\rho_0)}\,\dt < \infty
\]
under a sufficiently large coupling strength, and hence to enforce global injectivity of the flow.

By contrast, for nonlinear velocity couplings $p\in(2,3)$, although flocking and algebraic decay of the velocity diameter still hold, the available estimates do not provide a uniform lower bound on the effective alignment strength. As a consequence, the gradient system does not admit a time-integrable damping structure, and the global-in-time injectivity of the flow remains open in this regime. A detailed discussion of this obstruction is given in Remark \ref{rem:p23} below.
\end{remark}

At the level of strong solutions, multi-dimensional  Euler alignment was recently treated in \cite[Theorem 3]{Tadpre}, under the assumption of limited initial velocity fluctuations $\rd_v(0)$,
\begin{subequations}\label{eqs:Tadpre}
\begin{equation}\label{eq:Tadpre_a}
8\|\nabla\phi\|_{L^\infty}\rd_v(0) \leq \kappa \phi^2(\rd_\eta^\infty),
\end{equation}
and for sub-critical initial data satisfying 
\begin{equation}\label{eq:Tadpre_b}
\lambda_{\textnormal{min}}(\nabla_{\!{}_S}u_0) + \kappa \phi(\rd_\eta^\infty)  > 0,
\end{equation}
\end{subequations}
where $\nabla_{\!{}_S}$ denotes the symmetric part of the velocity gradient. Theorem \ref{thm:EA_global} extends this strong existence result to the weak regime, under the (slightly) stronger \eqref{eq:large_kappa} compared with \eqref{eqs:Tadpre}.

In higher dimensions, the theory of global weak solutions for the Euler--alignment system remains limited outside regimes where strong structural or dissipative effects are present. A notable structural setting is provided by the unidirectional velocity framework, in which global measure-valued and weak solutions in arbitrary spatial dimensions were established in \cite{LLST22}, allowing for the formation of mass concentrations. There, the dynamics are governed by the scalar quantity $e=\nabla \cdot u+\phi*\rho$, and the analysis yields a refined geometric description of concentration phenomena through the pushforward of singular measures along the limiting Lagrangian flow.

Outside the unidirectional setting, global weak solution theories in multiple spatial dimensions remain extremely limited. To the best of our knowledge, the only available results concern measure-valued solutions constructed under strong singularity assumptions on the communication kernel. For linear velocity alignment ($p=2$), global measure-valued solutions were obtained in \cite{FP24} under the assumption that the singularity exponent exceeds the spatial dimension, exploiting the regularizing effect induced by the singular alignment force. This approach was subsequently extended to nonlinear velocity couplings ($p\neq 2$) in \cite{CFP25}, again in a strongly singular regime, where singular dissipation plays a crucial role in suppressing velocity dispersion and ensuring compactness.
 
By contrast, Theorem \ref{thm:EA_global} establishes global weak solutions for bounded and Lipschitz communication kernels without relying on singular dissipation mechanisms. Instead of constructing solutions via compactness arguments, we start from the globally well-posed Lagrangian dynamics and obtain a quantitative control of the Lagrangian deformation under a sufficiently large coupling strength. This control enforces global injectivity of the flow map $\eta_t$, and hence the Reynolds stress vanishes identically, $\tau_t\equiv0$. As a consequence, the ERA system closes globally in time and yields a global-in-time distributional solution to the Euler--alignment system.

Taken together, Theorems \ref{thm:1DEA} and \ref{thm:EA_global} place the Euler--alignment system with regular kernels into a unified framework that connects sharp one-dimensional thresholds, global weak solvability, and the structural role of Lagrangian injectivity across dimensions. They complement entropy-based one-dimensional theories and kinetic approaches for singular interactions, while providing a transparent dynamical interpretation of Reynolds defects and their disappearance.


 
\subsubsection{Uniform-in-time mean-field limit and Euler--alignment}\label{sssec:mf}
Our final result concerns a uniform-in-time quantitative mean-field limit for the $N$-particle Cucker--Smale system \eqref{CS} in the case of linear velocity alignment, that is, $p = 2$, formulated in phase space and measured in Wasserstein distance. More precisely, we establish a stability estimate that holds uniformly for all $t \geq 0$ and is independent of the number of particles. Under additional structural conditions on the limiting dynamics, this phase-space convergence can be further reduced to a mono-kinetic Eulerian description, yielding the Euler--alignment system \eqref{eq:EA-2}.

To describe the limiting dynamics, let $(\eta_t,v_t)$ denote the global Lagrangian solution of the alignment system \eqref{Lag_p} with $p=2$ provided by Theorem \ref{thm:ext}, and let $(\rho_t,u_t)$ be the associated Eulerian pair constructed in Theorem \ref{thm:ERA}. The Lagrangian formulation plays a central role in our analysis, as it provides a natural reference dynamics against which the particle system can be compared at all times.

We begin by studying a uniform-in-time mean-field limit from the particle system \eqref{CS} to the limiting Lagrangian dynamics \eqref{Lag_p}. This first step is naturally formulated at the level of trajectories. From a broader conceptual viewpoint, our approach is rooted in the classical deterministic coupling method developed in \cite{BH77, Dob79, Neu84}  for the derivation of mean-field limits in kinetic theory. Under suitable smoothness assumptions on the interaction kernel, these works establish stability estimates for particle approximations of general initial measures, typically measured in bounded Lipschitz or Wasserstein-type distances; see also the reviews \cite{CCH14b, Jab14, Spo91} and references therein. The stochastic extensions of this framework, relevant in the presence of diffusion or noise, were developed in \cite{Mel96, Szn91}. Uniform-in-time propagation of chaos implies that the continuum model describes the behavior of the particle system at all time scales with respect to the number of particles, results that are important in different contexts, see for instance \cite{DGPS} and the references therein.

In the present work, we remain entirely within a deterministic setting and follow this classical trajectory-based philosophy. We directly compare the $N$-particle Cucker--Smale dynamics with the limiting Lagrangian flow associated with the Vlasov--alignment equation. This viewpoint is particularly natural in the mono-kinetic regime relevant to pressureless Euler-type limits, where the macroscopic dynamics is most transparently described through Lagrangian characteristics.

The modulated Wasserstein quantities introduced below provide a quantitative measure of the discrepancy between particle trajectories and the limiting Lagrangian flow. They play the role of a deterministic coupling error and allow us to establish stability estimates that are uniform in time and independent of
the number of particles. In this sense, our analysis can be viewed as a Lagrangian, characteristic-based mean-field stability theory tailored to alignment dynamics and Eulerian closure.
 
With this preparation, we now turn to the quantitative comparison between the particle system and the limiting Lagrangian dynamics. To this end, we introduce suitable modulated quantities measuring the discrepancy in phase space between the empirical measure $\mu_t^N$ and the reference Lagrangian flow. These quantities will later allow us to control Wasserstein distances between the particle system and its macroscopic limit.

Let $q\ge1$ and define the modulated energies
\begin{align*}
\mathscr{E}_q(X^N|\eta)(t)  &:= \lt( \frac1N\sum_{i=1}^N \intr |x_i(t) - \eta_t(x)|^q \,\rho_0(\dx)\rt)^\frac1q, \cr
\mathscr{E}_q(V^N|v)(t) &:= \lt(\frac1N\sum_{i=1}^N \intr |v_i(t) - v_t(x)|^q \,\rho_0(\dx)\rt)^\frac1q.
\end{align*}
Here the integration with respect to $\rho_0$ reflects the comparison of each particle trajectory with the entire reference Lagrangian flow, in the spirit of a deterministic coupling.

Finally, without loss of generality, we may assume that the total momentum is matched at $t=0$:
\[
\frac1N \sum_{i=1}^N v_i(0) = \intr  u_0\,\rho_0(\dx), 
\]
so that by conservation of momentum for the particle system and the Lagrangian model, 
\[
\frac1N \sum_{i=1}^N v_i(t) = \intr v_t(x)\,\rho_0(\dx), \quad \forall\, t \geq 0. 
\]

The following theorem provides a uniform-in-time stability estimate for the modulated Wasserstein quantities introduced above, comparing the $N$-particle Cucker--Smale dynamics with the limiting Lagrangian alignment flow.

\begin{theorem}\label{thm:deri_LA} 
 Let $p=2$ and $q \in [1,\infty]$. Let $\{(x_i,v_i)\}_{i=1}^N$ be a global classical solution to \eqref{CS} satisfying
\[
\rd_{V^N}\in L^1(\R_+), \quad \rd_{V^N}(t):=\max_{1\leq i,j \leq N}|v_i(t)-v_j(t)|.
\]
Let $(\eta, v)$ be a global solution to the system \eqref{Lag_p} satisfying
\[
\sup_{t \ge 0}\rd_\eta(t) <\infty.
\]
Then there exists a constant $C>0$, independent of $N$, $q$, and $t$, such that:
\[
\sup_{t \ge 0}\lt( \mathscr{E}_q(X^N|\eta)(t)+\mathscr{E}_q(V^N|v)(t)\rt) \le C\lt(\mathscr{E}_q(X^N|\eta)(0)+\mathscr{E}_q(V^N|v)(0)\rt),
\]
and moreover
\[
\mathscr{E}_q(V^N|v)(t) \to 0 \quad \text{as } t \to \infty.
\]
\end{theorem}

\begin{remark}  
In Theorem \ref{thm:deri_LA}, the assumptions $p=2$,
\bq\label{asp0}
\rd_{V^N} \in L^1(\R_+), \quad\text{and}\quad \sup_{t\ge0} \rd_\eta(t)<\infty
\eq
ensure a uniform (in time and $N$) Gr\"onwall-type estimate. If these assumptions are not imposed, the above argument still yields a stability bound of the form
\[
\mathscr{E}_2(X^N|\eta)(t)+\mathscr{E}_2(V^N|v)(t)  \le C(t)\lt(\mathscr{E}_2(X^N|\eta)(0)+\mathscr{E}_2(V^N|v)(0)\rt),
\]
but the constant $C(t)$ generally depends on time. In particular, for general $p\ge2$, Appendix \ref{app:gene_mf} establishes such a finite-time stability estimate and the corresponding mean-field convergence toward the Lagrangian/kinetic alignment dynamics. Thus, the assumptions in Theorem \ref{thm:deri_LA} are precisely those that upgrade finite-time mean-field stability to a \emph{uniform-in-time} estimate.

On the other hand, when both the particle system and the Lagrangian limit model exhibit flocking, the assumptions \eqref{asp0} are automatically satisfied, and the stability estimate holds uniformly in time.
\end{remark}

The stability estimate established in Theorem \ref{thm:deri_LA} provides a uniform-in-time control on the discrepancy between the $N$-particle dynamics and the limiting nonlinear Lagrangian flow. Among rigorous mean-field analyses of alignment models, a closely related framework is due to \cite{CC21}, where a general modulated-energy method is developed to pass from Newtonian particle systems with alignment interactions (possibly combined with damping and external or interparticle potentials) to pressureless Euler-type models with nonlocal dissipation.
Their approach builds on the idea of measuring the discrepancy between the particle system and a macroscopic velocity field through a suitable modulated kinetic energy.

In the case where only alignment acts, namely in the absence of damping and external or interparticle potentials, the central quantity in \cite{CC21} is the discrete modulated kinetic energy
\[
\calE_N(Z^N(t) | U(t)) = \frac1{2N}\sum_{i=1}^N |v_i(t) - u_t(x_i(t))|^2 = \frac12 \intrr |v - u_t(x)|^2 \mu^N_t(\dx,\dv)
\]
which originates from the modulated energy concept introduced in \cite{Bre00, BMNP03, Mas01, Ser20}.

This quantity provides a metric-like control between the particle configuration $(x_i(t),v_i(t))_{i=1}^N$ and the macroscopic velocity field $u_t$, even in the absence of a convex pressure potential, as is typical in the pressureless Euler regime.

Restricting to alignment-only dynamics, \cite{CC21} derives a differential inequality of the form
\[
\frac{\rd}{\dt}\calE_N(t) + (\mbox{alignment dissipation}) \leq C \calE_N(t) + C{\rd}_{\rm BL}^2(\rho^N_t, \rho_t), \quad \rho^N_t = \pi_x {}_\# \mu^N_t,
\]
together with a transport inequality which plays a crucial role in the pressureless setting. Indeed, in the absence of any pressure term, there is no direct coercive control on the density $\rho_t$.
The transport structure of the dynamics, however, allows one to show that the discrepancy between the particle density $\rho_t^N$ and its macroscopic counterpart $\rho_t$ can still be controlled in terms of the modulated kinetic energy \cite{Cho21}:
\[
{\rd}^2_{\rm BL}(\rho^N_t, \rho_t) \leq C {\rd}_{\rm BL}^2(\rho^N_0, \rho_0) + C \int_0^t \calE_N(s)\,\ds,
\]
where ${\rd}_{\rm BL}$ denotes the bounded Lipschitz distance. Here the constant $C>0$ depends on $\|\nabla u\|_{L^\infty}$ and on the final time horizon $T$. A Gr\"onwall argument then yields convergence toward mono-kinetic macroscopic dynamics on finite time intervals, with constants that grow with
$\|u\|_{L^\infty(0,T;W^{1,\infty})}$ and with $T$.
In particular, one obtains
\begin{align*}
{\rd}_{\rm BL}^2(\mu^N_t, \rho_t \otimes \delta_{u_t}) 
&\leq C\calE_N(Z^N(0) | U(0)) + C {\rd}^2_{\rm BL}(\rho^N_0, \rho_0),
\end{align*}
for some $C>0$ independent of $N$, but depending implicitly on $T$. 

By contrast, the approach developed in the present work is formulated purely in Wasserstein geometry and follows an intrinsically Lagrangian perspective. Rather than comparing the particle system directly with an Eulerian velocity field, we measure the discrepancy between the particle dynamics and the limiting alignment flow at the level of characteristics. This viewpoint allows us to bypass any reliance on a priori $W^{1,\infty}$ bounds for the Eulerian velocity $u$ and to exploit instead the stability properties of the Lagrangian dynamics.

A further distinction with respect to \cite{CC21} lies in the choice of metric used to quantify convergence. The analysis in \cite{CC21} is formulated in terms of the bounded Lipschitz distance, which, on bounded domains, is equivalent to the $1$-Wasserstein distance and therefore captures only first-order transport effects. By contrast, the present framework relies directly on Wasserstein distances of higher order. In the case of linear velocity coupling ($p=2$), the modulated estimates can be performed at arbitrary Wasserstein orders, yielding a significantly stronger notion of convergence. In particular, this allows for quantitative control of higher-order moments as well as of the maximal displacement between the particle system and the limiting dynamics, uniformly in time. Such a strengthening of the convergence topology is not accessible within the bounded Lipschitz framework.

To turn the uniform-in-time stability bound of Theorem \ref{thm:deri_LA} into a mean-field convergence statement for empirical measures, we now formulate the limit in Wasserstein distance. We recall that for probability measures $\mu,\nu$ on $\R^d$ with finite $q$-th moment, the $q$-Wasserstein distance is defined by
\[
 \rd_q(\mu,\nu) := \inf_{\pi \in \Pi(\mu,\nu)} \lt(\iint_{\R^d\times\R^d} |x-y|^q \, \rd\pi(x,y)\rt)^{\frac1q},
\]
where $\Pi(\mu,\nu)$ denotes the set of all couplings of $\mu$ and $\nu$. In particular, we denote by $ \rd_\infty(\mu,\nu)$ the $\infty$-Wasserstein distance
\[
 \rd_\infty(\mu,\nu) := \inf_{\pi \in \Pi(\mu,\nu)} \, \sup_{(x,y)\in\supp \pi} |x-y|,
\]
which measures the maximal displacement between supports of $\mu$ and $\nu$.

We are now in a position to state the uniform-in-time mean-field convergence result toward the Eulerian alignment dynamics.

\begin{theorem} \label{thm:mf}
Let $p=2$ and $q \in (1,\infty]$. Assume that the hypotheses of Theorem \ref{thm:deri_LA} hold.   Suppose in addition that the disintegration of $\rho_0$ along $\eta_t$ is Dirac for almost every $t\ge0$.  Let $(\rho_t,u_t)$ be the Eulerian pair associated with the Lagrangian flow $(\eta_t,v_t)$ through the Lagrangian--Eulerian correspondence described in Section \ref{ssec:ERA}.

If the initial modulated energies satisfy
\[
\mathscr{E}_q(X^N|\eta)(0)+\mathscr{E}_q(V^N|v)(0)\to 0\quad\text{as }N\to\infty,
\]
the empirical measures $\mu_t^N$ converge uniformly in time toward the mono-kinetic measure $\rho_t\otimes\delta_{u_t}$, in the sense that
\[
\esssup_{t\ge0} \rd_q (\mu_t^N, \rho_t\otimes\delta_{u_t} ) \to 0\quad\text{as }N\to\infty.
\]
Moreover, the pair $(\rho_t,u_t)$ is a distributional solution to the Euler--alignment
system \eqref{eq:EA-2}.
\end{theorem}

\begin{remark}
In the one-dimensional case, if $u_0 \in L^\infty(\rho_0)$ and the effective velocity $\widehat v$ is non-decreasing, then Theorem \ref{thm:1DEA} implies that the Lagrangian flow remains injective for all times. Consequently, the disintegration of $\rho_0$ along $\eta_t$ is Dirac for every $t\ge0$, and under the assumptions of Theorem \ref{thm:deri_LA} the uniform-in-time mean-field convergence holds toward a mono-kinetic Eulerian limit satisfying the Euler--alignment system in the sense of distributions.

When $d\ge2$, if the initial data satisfy the assumptions of Theorem \ref{thm:EA_global}, then the associated Lagrangian flow enjoys global injectivity and boundedness properties. As a result, all the a priori assumptions required in Theorem \ref{thm:mf} are automatically satisfied.
\end{remark}

\begin{remark}
A natural way to ensure the assumption
\[
\mathscr{E}_q(X^N|\eta)(0) + \mathscr{E}_q(V^N|v)(0)\to 0 \quad \text{as } N \to \infty
\]
is to generate the initial particle configuration by sampling positions independently according to $\rho_0$ and assigning velocities consistently with the initial Lagrangian velocity field, namely
\[
X_i(0)\sim\rho_0 \ \text{i.i.d.},\quad v_i(0):=u_0(X_i(0)), \quad 1\le i\le N,
\]
where $u_0$ denotes the Eulerian velocity field associated with the initial Lagrangian data \eqref{ini:Lag_p}.

Then the random variables $(X_i(0),v_i(0))$ are i.i.d. with common law $\rho_0\otimes\delta_{u_0}$, and the empirical measure
\[
\mu_0^N:=\frac1N\sum_{i=1}^N \delta_{(X_i(0),v_i(0))}
\]
converges almost surely toward $\rho_0\otimes\delta_{u_0}$ as $N\to\infty$. As a consequence, the initial modulated energies satisfy
\[
\mathscr{E}_q(X^N|\eta)(0)+\mathscr{E}_q(V^N|v)(0)\to 0, \quad \forall\, q>1.
\]
\end{remark}

\begin{remark}
Independently of the injectivity of the Lagrangian flow map $\eta_t$, the argument developed above yields a uniform-in-time mean-field convergence toward the kinetic measure $(\eta_t,v_t)_\#\rho_0$, which is a distributional solution to the Vlasov--alignment equation \eqref{VA} with $p=2$.

The additional assumption is only required to identify this kinetic limit with a mono-kinetic Eulerian state. Indeed, when the disintegration of $\rho_0$ along $\eta_t$ is Dirac, the Lagrangian--Eulerian correspondence implies that
\[
(\eta_t,v_t)_\#\rho_0 = \rho_t \otimes \delta_{u_t}.
\]
In particular, injectivity of $\eta_t$ is a sufficient condition for this property, but not the only one.

We emphasize that this mono-kinetic structure is consistent with the choice of the initial particle approximation. The condition
\[
\mathscr{E}_q(X^N|\eta)(0)+\mathscr{E}_q(V^N|v)(0)\to 0
\quad\text{as }N\to\infty
\]
requires the particle system to approximate a single Lagrangian velocity field at the initial time. This setting is therefore more restrictive than classical mean-field limits for general kinetic initial data, but it is precisely tailored to capture the uniform-in-time convergence toward mono-kinetic Eulerian dynamics.
\end{remark}
 

 
\begin{figure}[t]
\begin{center}
\begin{tikzpicture}[
  font=\small,
  sbox/.style={  draw,  rounded corners=2pt,  align=center,  inner sep=5pt,  minimum width=28mm,
  minimum height=14mm},
  box/.style={draw, rounded corners=2pt, align=center,
    inner sep=6pt, minimum width=48mm, minimum height=18mm},
  bigbox/.style={draw, rounded corners=2pt, align=center,
    inner sep=8pt, minimum width=58mm, minimum height=22mm},
  arr/.style={-Latex, line width=0.6pt},
  darr/.style={-Latex, dashed, line width=0.6pt},
  lab/.style={midway, fill=white, inner sep=1.5pt}
]

\node[coordinate] (C) at (0,0) {};

\node[sbox] (L) at (C) { \textbf{Lagrangian system}\\[1mm]
$(\eta_t, v_t)$ };

\node[sbox] (P) at (-6.1,3.2) { \textbf{Particle system}\\[1mm]
$\{(x_i(t),v_i(t))\}_{i=1}^N$ };

\node[sbox] (K) at (6.1,3.2) { \textbf{Kinetic equation}\\[1mm]
$f_t(x,v)$ };

\node[sbox] (E) at (-6.1,-3.2) { \textbf{Euler--alignment system}\\[1mm]
$(\rho_t,u_t)$ };

\node[sbox] (ER) at (6.1,-3.2) { \textbf{Euler--Reynolds system}\\[1mm]
$(\rho_t,u_t,\tau_t)$ };

\node[align=left, anchor=north] at ($(L.south)+(0,-2mm)$) {%
\begin{minipage}{0.62\linewidth}
\[
\rho_t=\eta_{t\#}\rho_0,
\]
\[
m_t=\eta_{t\#}(v_t\rho_0),
\]
\[
u_t=\frac{\rd m_t}{\rd\rho_t}.
\]
\end{minipage}
};

\draw[arr] (P) -- (K)
  node[lab, above] {mean-field limit $(N\to\infty)$};

\draw[arr] (P) -- (L)
  node[lab, midway] {stability};

\draw[arr] (P) -- (E)
  node[lab, midway, align=center] {mean-field limit  \\ via Lagrangian closure};

\draw[arr] (L) -- (K)
  node[lab, midway, align=center] {kinetic lift \\ $(\eta_t,v_t)_\#\rho_0$};

\draw[arr] (K) -- (ER)
  node[lab, midway, align=center] {disintegration \\ \&  moments};

\draw[arr, line width=0.8pt] (L) -- (E)
  node[lab, midway] {Dirac disintegration};

\draw[arr] (L) -- (ER)
  node[lab, midway] {nontrivial fibres};

\draw[arr] (ER) -- (E)
  node[lab, below] {Dirac disintegration $\Rightarrow \tau_t=0$};

\end{tikzpicture}
\end{center}

\caption{Schematic relations between the particle dynamics, the kinetic description, and the Lagrangian continuum system. The Lagrangian system $(\eta_t,v_t)$ serves as a reference flow (structurally close to the particle dynamics) and provides a push-forward representation of the kinetic measure. When the disintegration of $\rho_0$ along $\eta_t$ is Dirac (for instance when the flow is injective), the dynamics close to the Euler--alignment system; otherwise the induced Eulerian description takes the form of an Euler--Reynolds system.}
\end{figure}
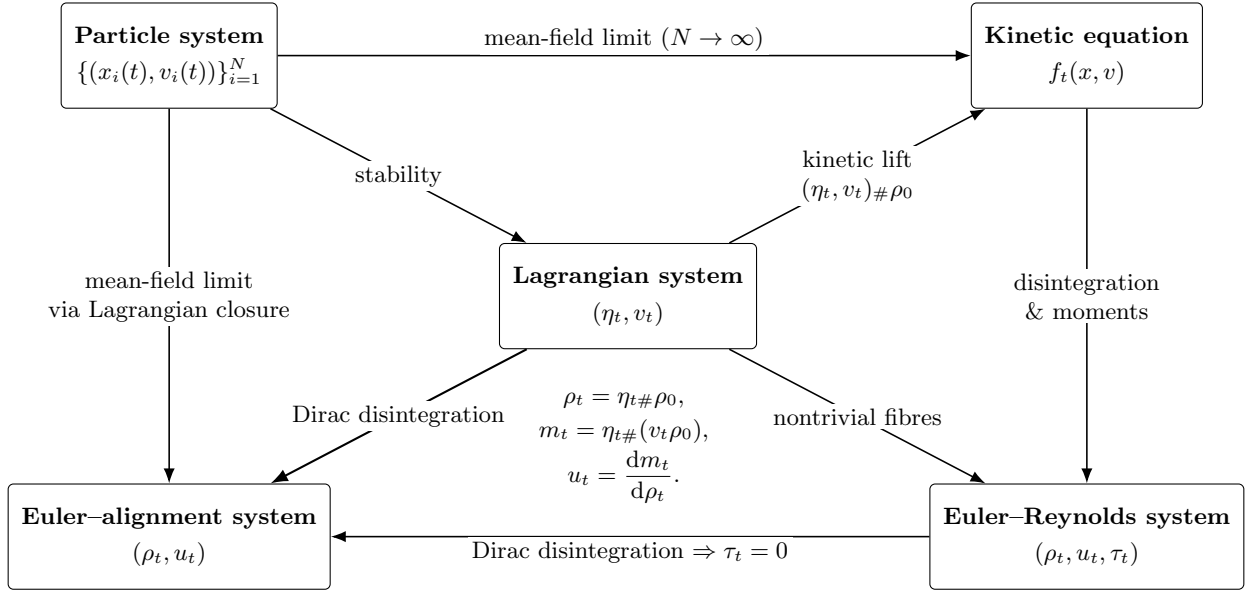



\subsection{Organization of the paper}
The rest of this paper is organized as follows. Section \ref{sec:LA} is devoted to the Lagrangian $p$-alignment system \eqref{Lag_p}; in particular, it proves global well-posedness and quantitative flocking estimates (Theorem \ref{thm:ext}). Section \ref{sec:ERA} develops the measure-theoretic Lagrangian--Eulerian correspondence and shows that the induced Eulerian triple solves the ERA system, together with the energy inequality (Theorem \ref{thm:ERA}). In Section \ref{sec:ac_ERA}, we establish the asymptotic closure mechanism under velocity flocking, showing that the Reynolds stress and, for $p>2$, the nonlinear defect force vanish in the long-time limit (Theorem \ref{thm:asymp_closure}). Section \ref{sec:1DEA} treats the one-dimensional linear case $p=2$ and characterizes the injective mono-kinetic regime through a sharp critical-threshold condition, while describing the loss of injectivity outside this regime (Theorem \ref{thm:1DEA}). Section \ref{sec:EA} addresses the multi-dimensional linear case $p=2$, deriving global weak solutions to the Euler--alignment system under a quantitative large-coupling condition that enforces closure (Theorem \ref{thm:EA_global}). Finally, Section \ref{sec:uni_mf} proves uniform-in-time mean-field stability and convergence results for the particle Cucker--Smale dynamics in the linear regime, yielding a uniform-in-time mono-kinetic Eulerian limit under an almost everywhere Dirac disintegration condition for the fibres of the Lagrangian flow (Theorems \ref{thm:deri_LA} and \ref{thm:mf}). Appendix \ref{app:kinetic} collects the phase-space (kinetic) reformulation of the Lagrangian flow and its moment relations with the Eulerian variables. Appendix \ref{app:gene_mf} provides complementary mean-field estimates for general $p\ge2$ on finite time intervals and discusses the resulting kinetic and mono-kinetic limits under additional structural assumptions.

%
%
%
%
%

 \section{Dynamics of Lagrangian $p$-alignment formulation}\label{sec:LA}
 
This section completes the proof of Theorem \ref{thm:ext}. We first prove global well-posedness of the Lagrangian system and then derive diameter-based estimates leading to flocking and explicit decay rates.

%
%
%
%
%
 \subsection{Global existence and uniqueness}
 
We prove the global existence and uniqueness part of Theorem \ref{thm:ext}. The argument is based on a fixed-point approach for the Lagrangian system and relies on uniform $L^\infty$ estimates.
 
\begin{proof}[Proof of Theorem \ref{thm:ext}: existence]
We work with the displacement variable
\[
w_t(x):=\eta_t(x)-x,
\]
so that $w(0,x)=0$ and $\pa_t w=v$. In terms of $(w,v)$, the Lagrangian $p$-alignment system becomes 
\bq\label{eq:w-formulation}
\frac{\rd}{\dt} \binom{w}{v} = \binom{v}{\calA_p[w,v]},
\quad
(w,v)\big|_{t=0}=(0,u_0),
\eq
where
\[
\calA_p[w,v](x) :=\kappa\intr \phi\lt((x+w(x))-(y+w(y))\rt) G_p(v(y)-v(x))\,\rho_0(\dy), \quad  G_p(\xi):=|\xi|^{p-2}\xi.
\]
The natural phase space for \eqref{eq:w-formulation} is 
\[
X:=L^\infty(\rho_0)\times L^\infty(\rho_0), \quad  \|(w,v)\|_X:=\|w\|_{L^\infty(\rho_0)}+\|v\|_{L^\infty(\rho_0)}.
\]
We first verify that the vector field on the right-hand side of \eqref{eq:w-formulation} is locally Lipschitz. Let $R>0$ and $(w_i,v_i)\in X$ satisfy $\|(w_i,v_i)\|_X \le R$, $i=1,2$. Since the first component $(w,v)\mapsto v$ is linear (hence Lipschitz), it remains to control the alignment operator. Writing
\[
\calA_p[w_1,v_1] - \calA_p[w_2,v_2] =: I + II,
\]
with
\[
I = \kappa\intr \lt\{\phi((x+w_1(x))-(y+w_1(y))) - \phi((x+w_2(x))-(y+w_2(y))) \rt\}  G_p(v_1(y)-v_1(x))\rho_0(\dy)
\]
and
\[
II = \kappa\intr   \phi((x+w_2(x))-(y+w_2(y))) \lt(  G_p(v_1(y)-v_1(x)) - G_p(v_2(y)-v_2(x))\rt) \rho_0(\dy),
\]
we estimate, using Lipschitz continuity of $\phi$,  
\[
|\phi((x+w_1(x))-(y+w_1(y))) - \phi((x+w_2(x))-(y+w_2(y)))|\le 2 \|\phi\|_{\rm Lip} \|w_1-w_2\|_{L^\infty(\rho_0)},
\]
while $|v_1(y)-v_1(x)|\le 2\|v_1\|_{L^\infty(\rho_0)} \le 2R$ ensures
\[
|I| \le 2\kappa \|\phi\|_{\rm Lip} (2R)^{p-1} \|w_1-w_2\|_{L^\infty(\rho_0)}.
\]
For $p\ge2$, note that 
\[
\nabla G_p(\xi) =(p-2)|\xi|^{p-4}\xi\otimes\xi + |\xi|^{p-2}I, 
\]
and this gives
\[
\|\nabla G_p(\xi)\|_{\rm op} \le C_p |\xi|^{p-2},
\]
for some constant $C_p>0$ depending only on $p$. Applying the mean-value form of Taylor's theorem, we get
\[
|G_p(a)-G_p(b)| \le \sup_{\theta\in[0,1]}\|\nabla G_p(\theta a+ (1-\theta) b)\|_{\rm op} |a-b| \le C_p (|a|+|b|)^{p-2}|a-b|.
\]
Thus, since $\|v_i\|_{L^\infty(\rho_0)} \le R$,
\[
|II| \le 2 \kappa \|\phi\|_{L^\infty} C_p (4R)^{p-2} \|v_1-v_2\|_{L^\infty(\rho_0)}.
\]
Hence, we have
\[
\|\calA_p[w_1,v_1]-\calA_p[w_2,v_2]\|_{L^\infty(\rho_0)} \le C(R,p,\phi,\kappa)\lt(\|w_1-w_2\|_{L^\infty(\rho_0)} + \|v_1-v_2\|_{L^\infty(\rho_0)}\rt),
\]
showing that the right-hand side of \eqref{eq:w-formulation} is locally Lipschitz on $X$. By the Picard--Lindel\"of theorem, there exists a unique maximal solution
\[
(w,v)\in C^1([0,T_{\max});X)
\]
for some $T_{\max}\in(0,\infty]$.

We now establish an $L^\infty$ maximum principle for $v$ using the upper Dini derivative. Fix a unit vector $e\in\R^d$ and set $s_t(x):=v_t(x)\cdot e$. Then, $s$ satisfies
\[
\pa_t s_t(x) =\intr K_t(x,y) (s_t(y)-s_t(x))\,\rho_0(\dy),
\]
where
\[
K_t(x,y):=\kappa \phi ((x+w_t(x))-(y+w_t(y)))|v_t(y)-v_t(x)|^{p-2}.
\]
Since $\kappa\ge0$ and $\phi\ge0$, we have $K_t\ge0$. For each fixed $t$, local existence guarantees  $\|v_t\|_{L^\infty(\rho_0)}<\infty$, and thus
\[
\|K_t\|_{L^\infty(\rho_0\otimes\rho_0)} \le \kappa \|\phi\|_{L^\infty} \lt(2\|v_t\|_{L^\infty(\rho_0)}\rt)^{p-2}<\infty.
\]
Define $M_e(t):=\esssup_{x}s_t(x)$, and let 
\[
D^+M_e(t):=\limsup_{h\downarrow0}\frac{M_e(t+h)-M_e(t)}{h}
\]
be its upper right Dini derivative. Fix $t\ge0$ and $\varepsilon>0$, and then choose $x_\e$ such that $s_t(x_\e)\ge M_e(t)-\e$. Then 
\[
s_t(y)-s_t(x_\e) \le \e \quad \text{for $\rho_0$-a.e.  $y$}, 
\]
and hence,
\[
\pa_t s_t(x_\e) =\intr K_t(x_\e,y)(s_t(y)-s_t(x_\e))\,\rho_0(\dy) \le \e \|K_t\|_{L^\infty(\rho_0\otimes\rho_0)}.
\]
By the standard Dini envelope argument (letting $\varepsilon\downarrow0$ at fixed $t$), we conclude $D^+M_e(t)\le0$ for all $t\geq 0$ and unit vector $e$.   Hence, we have
\bq\label{eq:v-Linf-decay}
\|v_t\|_{L^\infty(\rho_0)} =\sup_{|e|=1} M_e(t)\le \sup_{|e|=1} M_e(0)\le \|u_0\|_{L^\infty(\rho_0)} \quad\text{for all }t<T_{\max}.
\eq

Since $\pa_t w=v$, the bound \eqref{eq:v-Linf-decay} yields
\[
\|w_t\|_{L^\infty(\rho_0)} \le t \|u_0\|_{L^\infty(\rho_0)}, \quad \|\pa_t v_t\|_{L^\infty(\rho_0)} \le C(\phi,p,\kappa) \|u_0\|_{L^\infty(\rho_0)}^{p-1}.
\]
Thus $(w,v)$ remains bounded and uniformly Lipschitz in time on every finite interval, hence the continuation principle yields $T_{\max}=\infty$. Together with the local uniqueness from the Picard--Lindel\"of theorem, this gives a unique global solution. Restoring $\eta=x+w$, we obtain the desired solution of \eqref{Lag_p}. This completes the proof.
\end{proof}

Before proceeding to the flocking estimates, we record a few standard identities satisfied by sufficiently regular solutions of \eqref{Lag_p}, including conservation of momentum and dissipation of the kinetic energy.

\begin{lemma}\label{lem:apr} Let $(\eta,v)$ be a global classical solution to the system \eqref{Lag_p}. Then we have
\[
\frac{\rd}{\dt} \intr \eta_t(x)\,\rho_0(\dx) = \intr v_t(x)\,\rho_0(\dx), \quad \frac{\rd}{\dt}\intr v_t(x)\,\rho_0(\dx) =0,
\]
and
\[
\frac12 \frac{\rd}{\dt} \intr |v_t(x)|^2\,\rho_0(\dx) + \frac\kappa2\intrr \phi(\eta_t(x)-\eta_t(y))|v_t(x) - v_t(y)|^p \,\rho_0(\dx) \rho_0(\dy) =0.
\]
\end{lemma}
\begin{remark}
The center of mass can be explicitly given as
\[
 \intr \eta_t(x)\,\rho_0(\dx) =  \intr x\,\rho_0(\dx) + t  \intr u_0(x)\,\rho_0(\dx).
\]
Thus, by Galilean invariance, without loss of generality, one may assume
\[
 \intr \eta_t(x)\,\rho_0(\dx) =  \intr v_t(x)\,\rho_0(\dx) = 0, \quad \forall\, t \geq 0.
\]
\end{remark}

%
%
%
%
%

\subsection{Flocking estimates}

We now derive the quantitative flocking bounds in Theorem \ref{thm:ext}. Throughout this subsection, $(\eta,v)$ denotes the global Lagrangian solution constructed above. The argument follows the standard diameter method (cf.  \cite{BT24,CCH14,HHK10,NSTpre}), adapted to the present continuum Lagrangian setting.

\begin{proof}[Proof of Theorem \ref{thm:ext}: flocking dynamics]  
 Since 
 \[
 \pa_t (\eta_t(x) - \eta_t(y)) = v_t(x) - v_t(y), \quad x,y \in \supp \rho_0,
 \]
the envelope property of Dini upper derivative gives
 \[
 |D^+ \rd_\eta(t)| \leq \rd_v(t).
 \]
 
We next derive the differential inequality for $\rd_v(t)$. Let $x,y \in \supp \rho_0$ maximize $|v_t(x)-v_t(y)|$, and set 
\[
n := \frac{v_t(x) - v_t(y)}{|v_t(x) - v_t(y)|}.
\]
Projecting the equations along $n$ and using that $r\mapsto|r|^{p-2}r$ is odd and increasing ($p\ge2$), one obtains the sign relations
\[
n \cdot (v_t(z)-v_t(x))\le0, \quad n  \cdot (v_t(z)-v_t(y))\ge0 
\]
for $\rho_0$-a.e. $z$, and thus
\[
n \cdot G_p(v_t(z)-v_t(x))\le0,\quad n \cdot G_p(v_t(z)-v_t(y))\ge0.
\]
Now recall that $\phi$ is radial and nonincreasing in the distance. Since $|\eta_t(x)-\eta_t(z)|\le\rd_\eta(t)$ and $|\eta_t(y)-\eta_t(z)|\le\rd_\eta(t)$, we have
\[
\phi(\eta_t(x)-\eta_t(z)) \ge \phi(\rd_\eta(t)),\quad \phi(\eta_t(y)-\eta_t(z)) \ge \phi(\rd_\eta(t)) \quad\text{for $\rho_0$-a.e. }z.
\]
Using the strong monotonicity of $G_p$ (valid for all $p\ge2$),
\[
\big(G_p(\xi)-G_p(\zeta)\big) \cdot(\xi-\zeta) \ge  2^{\,2-p}\,|\xi-\zeta|^p, \quad \xi,\zeta\in\R^d,
\]
we arrive at
\bq\label{eq:dia_v}
D^+ \rd_v(t) \leq - 2^{2-p}\kappa\phi(\rd_\eta(t))\rd_v^{p-1}(t).
\eq
When $p=2$, Gr\"onwall's inequality yields
\[
\rd_v(t)\le \rd_v(0)\exp\left(-\kappa\int_0^t \phi(\rd_\eta(s))\,ds\right).
\]
When $p>2$, integration of \eqref{eq:dia_v} gives
\[
\rd_v(t) \le \lt(\rd_v(0)^{2-p} + (p-2)2^{2-p}\kappa \int_0^t \phi(\rd_\eta(s))\,\ds\rt)^{-\frac1{p-2}}. 
\]
Hence, in either case, velocity alignment follows provided
\[
\int_0^\infty \phi(\rd_\eta(s))\,\ds = \infty.
\]
On the other hand, since $\rd_v(t)\le \rd_v(0)$, we have
 \[
 \rd_\eta(t)\le \rd_\eta(0)+t \rd_v(0).
 \]
and therefore
\[
\int_0^\infty \phi(\rd_\eta(s))\,\ds \ge \int_0^\infty \phi (\rd_\eta(0)+s \rd_v(0))\,\ds = \frac1{\rd_v(0)}\int_{\rd_\eta(0)}^\infty \phi(r)\,\dr.
\]
Under the heavy-tail condition \eqref{condi_phi}, the right-hand side diverges, and thus $\rd_v(t)\to0$ as $t\to\infty$. This completes the proof.
 \end{proof}
 
  \begin{remark}[A sharper flocking estimate for $2\le p<3$]\label{rmk:p-sub3}
 Assume $p\in[2,3)$ and suppose that the initial data satisfy
  \[
\frac1{3-p} \rd_v^{3-p}(0) <  2^{2-p}\kappa \int_{\rd_\eta(0)}^\infty \phi(r)\,\dr.
 \]
 This condition is clearly weaker than the heavy-tail assumption \eqref{condi_phi}. Consider the Lyapunov functional
\[
\widehat\calL_p(t) := \frac1{3-p} \rd_v^{3-p}(t) + 2^{2-p}\kappa \int_{\rd_\eta(0)}^{\rd_\eta(t)} \phi(r)\,\dr.
\]
Then, using
\[
D^+ \rd_\eta(t) \leq \rd_v(t), \quad D^+ \rd_v(t) \leq - 2^{2-p}\kappa\phi(\rd_\eta(t))\rd_v^{p-1}(t),
\]
one obtains
\[
D^+ \widehat\calL_p(t)\le0.
\]
Hence
\[
\widehat\calL_p(t)\le \widehat\calL_p(0) \quad\text{for all }t\ge0.
\]
By the assumption on the initial data, there exists a unique $\rd_\eta^\infty>0$ such that
\[
\frac1{3-p} \rd_v^{3-p}(0) =  2^{2-p}\kappa \int_{\rd_\eta(0)}^{\rd_\eta^\infty} \phi(r)\,\dr,
\]
and consequently
\[
0 \leq  \frac1{3-p} \rd_v^{3-p}(t) \leq 2^{2-p}\kappa \int_{\rd_\eta(t)}^{\rd_\eta^\infty} \phi(r)\,\dr.
\]
In particular,
\[
\sup_{t \geq 0} \rd_\eta(t) \leq \rd_\eta^\infty.
\]
Therefore, we have
\[
D^+ \rd_v(t) \leq - 2^{2-p}\kappa \phi(\rd_\eta^\infty)\rd_v^{p-1}(t),
\]
which yields the explicit decay estimate \eqref{decay_v_sub3}. This argument is specific to the range $2\le p<3$, since for $p\ge3$ the coefficient $(3-p)^{-1}$ is nonpositive and the above Lyapunov structure no longer provides a useful coercive control.
 \end{remark}

%
%
%
%
%
\section{Euler--Reynolds--alignment system}\label{sec:ERA}

This section develops the Lagrangian--Eulerian correspondence stated in the introduction for the global $L^\infty$ Lagrangian flow $(\eta_t,v_t)$ constructed in Theorem \ref{thm:ext}. We first show how to associate to $(\eta_t,v_t)$ an Eulerian density $\rho_t$, momentum $m_t$, barycentric velocity $u_t$, and a nonnegative Reynolds stress $\tau_t$, using only pushforward and disintegration of the reference measure $\rho_0$. We then establish the basic time-regularity of these objects and verify that the resulting triple $(\rho_t,u_t,\tau_t)$ solves the ERA system \eqref{eq:ERA} in the sense of distributions, satisfies the global energy inequality \eqref{ERA_energy}, and is unique within the class of Eulerian triples induced by Lagrangian solutions.

%
%
%
%
%
\subsection{Lagrangian--Eulerian correspondence}\label{ssec:LEc}

Let $(\eta_t,v_t)$ be the global Lagrangian solution from Theorem \ref{thm:ext}. Since $\eta_t$ is only available in an $L^\infty$ framework and may be non-injective, Eulerian quantities cannot be recovered pointwise. We therefore use disintegration of $\rho_0$ along the fibres of $\eta_t$ to define the Eulerian density and momentum by pushforward, the barycentric velocity by Radon--Nikodym differentiation, and the Reynolds stress as the associated fibrewise covariance.  

We begin by recalling the disintegration theorem of \cite[Theorem 5.3.1]{AGS08} and then apply it to the flow map $\eta_t$ to define the Eulerian density, velocity, and Reynolds stress.

\begin{theorem}\label{thm:disint}
Let $X$, $\bar X$ be Radon separable metric spaces, $\mu \in \calP(X)$, let $\pi: X \to \bar X$ be a Borel map and let $\nu = \pi_\# \mu \in \calP(\bar X)$. Then there exists a $\nu$-a.e. uniquely determined Borel family of probability measures $\{\mu_{\bar x}\}_{\bar x \in \bar X} \subset \calP(X)$ such that
\[
\mu_{\bar x} (X \setminus \pi^{-1}(\{\bar x\})) = 0 \quad \text{for $\nu$-a.e. $\bar x \in \bar X$},
\]
and for every Borel map $f: X \to [0,+\infty)$,
\[
\int_X f(x)\,\mu(\dx) = \int_{\bar X} \lt(  \int_{\pi^{-1}(\bar x)} f(x) \,\mu_{\bar x}(\dx)\rt)  \nu(\rd\bar x).
\]
In particular, if $X = X_1 \times X_2$, $\bar X = X_1$, and $\mu \in \calP(X_1 \times X_2)$ with first marginal $\nu = \pi_1 {}_\# \mu$, then one can identify each fibre $\pi_1^{-1}(x_1) = \{x_1\} \times X_2$ with $X_2$ and find a Borel family $\{\mu_{x_1}\}_{x_1 \in X_1} \subset \calP(X_2)$ (unique $\nu$-a.e.) such that
\[
\mu = \int_{X_1} \mu_{x_1} \nu(\rd x_1).
\]
\end{theorem} 

\begin{lemma} \label{lem:ERA-prep}
Assume $p\geq 2$ and the hypotheses of Theorem \ref{thm:ext}. For each $t\ge0$, define
\[
\rho_t := (\eta_t)_\# \rho_0, \quad m_t := (\eta_t)_\#(v_t\,\rho_0).
\]
Then the following holds:
\begin{enumerate}[label=(\roman*)]
\item $m_t \ll \rho_t$ and the barycentric (mean) velocity
\[
u_t := \frac{\rd m_t}{\rd\rho_t}
\]
is well-defined and satisfies $u_t\in L^\infty(\rho_t;\R^d)$.

\item There exists a (unique $\rho_t$-a.e.) family of probability measures 
$\{\nu_{t,z}\}_{z\in \supp(\rho_t)}$ concentrated on $\eta_t^{-1}(\{z\})$ such that
\[
\rho_0 = \intr \nu_{t,z}\,\rho_t(\dz),
\]
that is, for every Borel set $A\subset\R^d$,
\[
\rho_0(A)=\intr \nu_{t,z}(A)\,\rho_t(\dz).
\]
In particular, for $\rho_t$-a.e.  $z$, the barycentric velocity admits the fibre-average representation
\[
u_t(z)=\intr v_t(x)\,\nu_{t,z}(\dx),
\]
where the integral may be taken over $\R^d$ since $\nu_{t,z}$ is supported on $\eta_t^{-1}(\{z\})$.

\item The Reynolds stress tensor is given by the fibre variance
\[
\tau_t := \rho_t \theta_t,  \quad \theta_t(z) := \intr (v_t(x)-u_t(z))\otimes(v_t(x)-u_t(z))\,\nu_{t,z}(\dx),
\]
and satisfies $\tau_t \ge 0$ (as a matrix-valued measure with symmetric positive semidefinite density $\theta_t$ $\rho_t$-a.e.) and $\tau_{t=0} = 0$. Moreover, for every $T>0$, the Reynolds stress satisfies the uniform bound
\[
\sup_{t\in[0,T]}\|\tau_t\|_{\calM(\R^d;\R^{d\times d})} \le \intr |u_0(x)|^2\,\rho_0(\dx).
\]
\end{enumerate}
\end{lemma}

\begin{remark}
If the Lagrangian flow map $\eta_t$ is injective (in fact, bi-measurable with measurable inverse on its image), then each fibre $\eta_t^{-1}(\{z\})$ is a singleton and hence $\nu_{t,z}=\delta_{\eta_t^{-1}(z)}$ for $\rho_t$-a.e.  $z$. Consequently, $\theta_t \equiv 0$ and thus $\tau_t \equiv 0$. Conversely, non-injectivity of $\eta_t$ does not necessarily imply that $\tau_t$ is nonzero: if all points in a fibre share the same velocity, the fibre variance still vanishes. Thus $\tau_t$ measures the velocity dispersion within fibres, not the geometric injectivity of $\eta_t$. In particular, the vanishing of $\tau_t$ is equivalent to the collapse of the disintegration to Dirac masses with zero fibre variance, whereas injectivity is only a sufficient condition for this to occur.
\end{remark}

\begin{proof}[Proof of Lemma \ref{lem:ERA-prep}]
We proceed in several steps. Throughout, $t\ge0$ is fixed.

\medskip
\noindent  {\bf Measurability and pushforwards.}
By Theorem \ref{thm:ext}, for each fixed $t\ge0$ the maps
$x\mapsto \eta_t(x)$ and $x\mapsto v_t(x)$ are Borel on $(\R^d,\rho_0)$, and
$v_t\in L^\infty(\rho_0;\R^d)$.  Hence the pushforward
\[
\rho_t(B):=\rho_0(\eta_t^{-1}(B))
\]
is a probability (Radon) measure on $\R^d$ for every Borel set $B$, and
\[
m_t(B):=\int_{\eta_t^{-1}(B)} v_t(x)\,\rho_0(\dx)
\]
defines a finite $\R^d$-valued Radon measure. 

\medskip
\noindent {\bf Absolute continuity $m_t\ll\rho_t$ and definition of $u_t$.}
Let $A\subset\R^d$ be Borel with $\rho_t(A)=0$. Then $\rho_0(\eta_t^{-1}(A))=0$ by definition of pushforward. Consequently,
\[
|m_t|(A) =\int_{\eta_t^{-1}(A)} |v_t(x)|\,\rho_0(\dx) \leq \|v_t\|_{L^\infty(\rho_0)}\int_{\eta_t^{-1}(A)} \rho_0(\dx) =0,
\]
due to $v_t\in L^\infty(\rho_0)$.
Thus $m_t\ll \rho_t$. By the Radon--Nikodym theorem, there exists a (class of) Borel function(s)
$u_t\in L^1(\rho_t;\R^d)$ such that
\[
m_t = u_t\,\rho_t, \quad\text{i.e.,}\quad u_t = \frac{\rd m_t}{\rd \rho_t}\ \ \rho_t\text{-a.e.}
\]

\medskip
\noindent {\bf $L^\infty$ bound on $u_t$ and barycentric identity.}
Apply Theorem \ref{thm:disint} with 
\[
X=\R^d,\quad \bar X=\R^d,\quad \mu=\rho_0,\quad \pi=\eta_t,
\]
so that the pushforward measure is $\nu=\pi_\#\mu=\rho_t$.  This yields a $\rho_t$-a.e.  uniquely determined Borel family of probability measures 
$\{\nu_{t,z}\}_{z\in\supp(\rho_t)}$ such that
\begin{equation*}
\rho_0=\intr \nu_{t,z}\,\rho_t(\dz), \quad \supp\,\nu_{t,z}\subset   \eta_t^{-1}(\{z\}).
\end{equation*}
Define the fibre barycentre 
\[
\bar v_t(z):=\intr v_t(x)\,\nu_{t,z}(\dx),\quad \rho_t\text{-a.e. }z.
\]
Since $v_t\in L^\infty(\rho_0)$, there exists a $\rho_0$-null set $N_t$ such that
\[
|v_t(x)|  \le \|v_t\|_{L^\infty(\rho_0)} \le \|u_0\|_{L^\infty(\rho_0)}   \quad \text{for all } x \in \R^d \setminus N_t
\]
due to \eqref{eq:v-Linf-decay}. By the disintegration identity, we get
\[
0 = \rho_0(N_t) = \intr \nu_{t,z}(N_t) \,\rho_t(\dz),
\]
and thus, $\nu_{t,z}(N_t) = 0$ for $\rho_t$-a.e. $z$. Hence, for $\rho_t$-a.e $z$
\[
|\bar v_t(z)| \le \intr |v_t(x)|\,\nu_{t,z}(\dx)  \leq \|v_t\|_{L^\infty(\rho_0)} \le \|u_0\|_{L^\infty(\rho_0)}.  
\]
In particular, $\bar v_t\in L^\infty(\rho_t)$ and
\[
\|\bar v_t(z)\|_{L^\infty(\rho_t)} \le \|u_0\|_{L^\infty(\rho_0)}.  
\]

Next, for any $F\in C_b(\R^d)$, by Fubini and disintegration,
\[
\intr F(z)\,\bar v_t(z)\,\rho_t(\dz) = \intr \left(\intr F(z)\,v_t(x)\,\nu_{t,z}(\dx)\right)\rho_t(\dz) = \intr F(\eta_t(x))\,v_t(x)\,\rho_0(\dx).
\]
On the other hand, for the vector measure $m_t$,
\[
\intr F(z)\,m_t(\dz) = \intr F(\eta_t(x)) v_t(x)\,\rho_0(\dx).
\]
Since $m_t=u_t\,\rho_t$ and $F$ is arbitrary in $C_b$, we conclude that $u_t=\bar v_t$ $\rho_t$-a.e., and
\bq\label{eq:bary-identity-lemma}
\intr  F(z)\,u_t(z)\,\rho_t(\dz) = \intr  F(\eta_t(x))\,v_t(x)\,\rho_0(\dx),\quad \forall\,F\in C_b(\R^d).
\eq

\medskip
\noindent {\bf Fibre covariance and definition of $\theta_t$, $\tau_t$.} With the disintegration $\{\nu_{t,z}\}$ fixed as above, define for $\rho_t$-a.e.  $z$,
\[
\theta_t(z):=\intr (v_t(x)-u_t(z))\otimes (v_t(x)-u_t(z))\,\nu_{t,z}(\dx)\in\R^{d\times d}.
\]
Measurability of $z\mapsto\theta_t(z)$ follows from the standard kernel measurability of disintegrations and the fact that $(z,x)\mapsto (v_t(x)-u_t(z))\otimes(v_t(x)-u_t(z))$ is Borel and integrably bounded. Define the matrix-valued measure $\tau_t:=\rho_t\,\theta_t$, i.e.,
\[
\intr \Psi : \rd\tau_t := \intr \Psi(z):\theta_t(z)\,\rho_t(\dz)
\quad\text{for all }\Psi\in C_b(\R^d;\R^{d\times d}).
\]
Each $\theta_t(z)$ is a covariance matrix, hence symmetric and positive semidefinite; therefore $\tau_t$ is a finite symmetric positive semidefinite matrix-valued Radon measure.

\medskip
\noindent {\bf Finiteness of $\tau_t$ and the variance identity.}
Finiteness follows from the trace estimate. Using the identity
\[
\intr {\rm tr}\,\theta_t(z)\,\rho_t(\dz)
= \intrr |v_t(x)-u_t(z)|^2\,\nu_{t,z}(\dx) \rho_t(\dz),
\]
expand the square and use \eqref{eq:bary-identity-lemma} with $F\equiv 1$ and $F\equiv u_t$ (the latter obtained by approximating $u_t$ with its truncations $u_t^R:=u_t {\bf 1}_{{|u_t|\le R}}$ and passing to the limit via dominated convergence and the $L^2$-isometry $\|g\circ\eta_t\|_{L^2(\rho_0)} = \|g\|_{L^2(\rho_t)}$) to obtain the variance decomposition
\[
\intr  {\rm tr}\,\theta_t(z)\, \rho_t(\dz) = \intr |v_t(x)|^2\,\rho_0(\dx) - \intr |u_t(z)|^2\,\rho_t(\dz)
\le \intr |v_t|^2\,\rho_0(\dx) < \infty.
\] 
Thus $\tau_t$ is finite. More generally, for any $\Psi\in C_b(\R^d;\R^{d\times d})$, expanding fibrewise yields
\begin{align*}
\intr \Psi:\theta_t\,\rho_t(\dz)
&= \intr \Psi(\eta_t(x)):(v_t\otimes v_t)\,\rho_0(\dx)
 - \intr \Psi(z): (u_t(z)\otimes u_t(z))\,\rho_t(\dz)\\
&= \intr \Psi(\eta_t(x)):\lt(v_t\otimes v_t - u_t(\eta_t(x))\otimes u_t(\eta_t(x))\rt)\rho_0(\dx),
\end{align*}
which is the claimed identity (independence of the choice of disintegration follows from equality of all integrals against test fields $\Psi$).

\medskip
\noindent {\bf Initial traces and positivity.}
At $t=0$, $\eta_0={\rm id}$; therefore each fibre $\eta_0^{-1}(\{z\})=\{z\}$ is a singleton and $\nu_{0,z}=\delta_z$. Hence,
\[
u_{t=0}(z)=\intr v_0(x)\,\delta_z(\dx)=v_0(z)=u_0(z),
\]
\[
\theta_{t=0}(z)=\intr (v_0(x)-u_0(z))\otimes(v_0(x)-u_0(z)) \,\delta_z(\dx)=0,
\]
and thus $\tau_{t=0}=0$. Positivity of $\tau_t$ follows from positivity of $\theta_t(z)$ for $\rho_t$-a.e. $z$.
\end{proof}

 We next provide the temporal regularity of the Eulerian quantities $(\rho_t,m_t)$  constructed in Lemma \ref{lem:ERA-prep}. Since the Lagrangian flow map $t\mapsto\eta_t$ and velocity $t\mapsto v_t$ are continuous in $L^\infty(\rho_0)$ by Theorem \ref{thm:ext}, one expects the corresponding pushforward objects to vary continuously in time. This is essential for formulating the Eulerian system in a weak sense.

\begin{lemma} \label{lem:cont}
Assume $p\geq 2$ and the hypotheses of Theorem \ref{thm:ext}.  Let $(\rho_t,m_t)$ be the Eulerian quantities associated with the Lagrangian solution $(\eta_t,v_t)$ as in Lemma \ref{lem:ERA-prep}. Then the following properties hold:
\begin{align*}
t &\mapsto \rho_t \quad\text{is narrowly continuous in }\calP(\R^d), \cr
t &\mapsto m_t \quad\text{is weak$^\ast$ continuous in }\calM(\R^d;\R^d).
\end{align*}
\end{lemma}

\begin{proof}
Fix $t>0$. Since $t\mapsto w_t:= \eta_t - {\rm id}$ is continuous in $L^\infty(\rho_0)$, we have
\[
\lim_{s\to t}\|w_s-w_t\|_{L^\infty(\rho_0)}=0,
\]
hence for $\rho_0$-a.e.  $x$,
$
\eta_s(x)=x+w_s(x)\to x+w_t(x)=\eta_t(x).
$
Let $\xi \in C_b(\R^d)$. Then
\[
\intr \xi\,\rho_s(\dx) = \intr \xi(\eta_s(x))\,\rho_0(\dx).
\]
Since $\xi$ is bounded and $\xi(\eta_s(x))\to \xi(\eta_t(x))$ for $\rho_0$-a.e.  $x$, the dominated convergence theorem yields
\[
\intr \xi(\eta_s(x))\,\rho_0(\dx)\to \intr \xi(\eta_t(x))\,\rho_0(\dx), \quad \text{i.e., $t\mapsto\rho_t$ is narrowly continuous.}
\]

For $m_t$, take $\varphi\in W^{1,\infty}(\R^d;\R^d)$. Then, we obtain
\begin{align*}
\lt| \intr \varphi\cdot (m_s(\dx) - m_t(\dx))  \rt| &= \lt|\intr \lt(\varphi(\eta_s(x))\cdot v_s(x) - \varphi(\eta_t(x))\cdot v_t(x)\rt) \rho_0(\dx)  \rt| \cr
&\leq \intr |\varphi(\eta_s)-\varphi(\eta_t)|\,|v_s|\, \rho_0(\dx) + \intr |\varphi(\eta_t)|\,|v_s-v_t|\, \rho_0(\dx)\cr
&\leq \|\varphi\|_{\rm Lip}\|v_s\|_{L^\infty(\rho_0)} \|w_s-w_t\|_{L^\infty(\rho_0)} + \|\varphi\|_{L^\infty}\|v_s-v_t\|_{L^\infty(\rho_0)}\cr
&\to 0 \quad \mbox{as } s\to t.
\end{align*}
Thus, $t\mapsto m_t$ is weak$^\ast$ continuous. This completes the proof.
\end{proof}

%
%
%
%
%

\subsection{Global Eulerian solutions}

In this subsection, we verify that the Eulerian objects $(\rho_t,u_t,\tau_t)$ constructed in Lemma \ref{lem:ERA-prep} indeed satisfy the ERA system \eqref{eq:ERA} in the sense of distributions.
The construction relies exclusively on the underlying Lagrangian flow $(\eta_t,v_t)$ provided by Theorem \ref{thm:ext} without \emph{a priori} Eulerian regularity assumptions. We also show that this Eulerian solution satisfies the corresponding global energy inequality and is unique within the class of Eulerian triples induced by Lagrangian solutions of Theorem \ref{thm:ext}, thereby completing the proof of Theorem \ref{thm:ERA}.
 
\begin{proof}[Proof of Theorem \ref{thm:ERA}]
Let $(\eta_t,v_t)$ be the global Lagrangian solution provided by Theorem \ref{thm:ext}. All Eulerian objects $\rho_t$, $m_t$, $u_t$, $\theta_t$ and $\tau_t:=\rho_t\,\theta_t$ are already constructed in Lemma \ref{lem:ERA-prep}, where we also recorded: $m_t\ll\rho_t$ with $u_t=\frac{\rd m_t}{\rd\rho_t}\in L^\infty(\rho_t)$, $\tau_t$ is symmetric positive semidefinite and finite, and the barycentric identity
\bq\label{eq:bary-identity-proof}
\intr F(z) u_t(z)\,\rho_t(\dz) =\intr F(\eta_t(x)) v_t(x)\,\rho_0(\dx) \quad\text{for all bounded Borel }F:\R^d\to\R^d,
\eq
together with the fibre support property of the disintegration $\{\nu_{t,z}\}$.

Fix $T>0$. In what follows, we take test functions
\[
\xi\in C_c^\infty((0,T)\times\R^d),\quad \varphi\in C_c^\infty((0,T)\times\R^d;\R^d),
\]
so that no boundary terms in time appear.  

\medskip
\noindent {\bf  Step 1: Continuity equation in the sense of distributions.} By Theorem \ref{thm:ext}, $t\mapsto \eta_t- {\rm id}$ and $t\mapsto v_t$ are $C^1$ in $L^\infty(\rho_0)$, hence $t\mapsto \intr \xi(t,\eta_t(x))\,\rho_0(\dx)$ is absolutely continuous and
\[
\frac{\rd}{\dt}\intr \xi(t,z)\,\rho_t(\dz)=\frac{\rd}{\dt}\intr \xi(t,\eta_t(x))\,\rho_0(\dx) = \intr \lt(\pa_t\xi+\nabla\xi\cdot v_t\rt)(t,\eta_t(x))\,\rho_0(\dx).
\]
Using \eqref{eq:bary-identity-proof} with $F=\nabla\xi(t,\cdot)$, we obtain
\[
\frac{\rd}{\dt}\intr \xi(t,z)\,\rho_t(\dz)
=\intr (\pa_t\xi) \rho_t(\dz)+\intr (\nabla\xi)\cdot u_t\,\rho_t(\dz).
\]
Integrating in time over $(0,T)$ and using $\xi\in C_c^\infty((0,T)\times\R^d)$ yield
\[
\int_0^T\intr (\pa_t\xi) \rho_t(\dz)dt+\int_0^T\intr (\nabla\xi)\cdot u_t\,\rho_t(\dz)dt=0 ,
\]
which is the distributional form of $\pa_t\rho+\nabla \cdot(\rho u)=0$ on $(0,T)\times\R^d$.

\medskip
\noindent {\bf  Step 2: Momentum equation in the sense of distributions.} Define for $t \in (0,T)$
\[
I(t):=\intr \varphi(t,\eta_t(x))\cdot v_t(x)\,\rho_0(\dx)=\intr \varphi(t,z)\cdot m_t(\dz).
\]
Differentiating in $t$ and using $\pa_t\eta_t=v_t$ and the Lagrangian equation for $v_t$,
\[
\pa_t v_t(x)=\kappa\intr \phi(\eta_t(x)-\eta_t(y)) G_p(v_t(y)-v_t(x))\,\rho_0(\dy),
\]
we get
\begin{align*}
\frac{\rd}{\dt} I(t) &=\intr \pa_t\varphi(t,\eta_t(x))\cdot v_t(x)\,\rho_0(\dx)
 +\intr \nabla\varphi(t,\eta_t(x)) : (v_t(x)\otimes v_t(x))\,\rho_0(\dx)\\
&\quad +\kappa\intrr \varphi(t,\eta_t(x))\cdot \phi(\eta_t(x)-\eta_t(y))G_p(v_t(y)-v_t(x))\,\rho_0(\dy) \rho_0(\dx).
\end{align*}
Antisymmetrizing the last term via the evenness of $\phi$ (swap $(x,y)$ and average) yields
\begin{equation}\label{new}  
\frac12\intrr \phi(\eta_t(x)-\eta_t(y)) \lt(\varphi(t,\eta_t(x))-\varphi(t,\eta_t(y))\rt)\cdot G_p(v_t(y)-v_t(x))\,\rho_0(\dy) \rho_0(\dx).
\end{equation}
We now pass to Eulerian variables in each term.

\smallskip
\noindent  \emph{Step 2.(a) Transport and quadratic terms.}
By \eqref{eq:bary-identity-proof} with $F=\pa_t\varphi(t,\cdot)$, 
\[
\intr \pa_t\varphi(t,\eta_t)\cdot v_t\,\rho_0(\dx)
=\intr \pa_t\varphi(t,z)\cdot u_t(z)\,\rho_t(\dz)=\intr \pa_t\varphi\cdot m_t(\dz).
\]
For the quadratic term we use the macro/fluctuation decomposition (Lemma \ref{lem:ERA-prep}): writing
\[
v_t=u_t\circ\eta_t+\omega_t \quad \text{with }\intr \omega_t\,\nu_{t,z}(\dx)=0
\]
on each fibre, expanding $v_t\otimes v_t$ and integrating fibrewise gives
\[
\intr \nabla\varphi(t,\eta_t): v_t\otimes v_t\,\rho_0(\dx) =\intr \nabla\varphi(t,z): u_t\otimes u_t \, \rho_t(\dz) + \intr \nabla\varphi(t,z): \tau_t (\dz).
\]

\smallskip
\noindent \emph{Step 2.(b) Nonlocal alignment term.}
Disintegrate the product measure 
\[
\rho_0(\dx)\rho_0(\dy)=\nu_{t,z}(\dx)\,\nu_{t,\zeta}(dy)\,\rho_t(\dz)\rho_t(\rd\zeta)
\] 
and write, for $x \in \supp \nu_{t,z}$, $y \in \supp \nu_{t,\zeta}$,
\[
v_t(x) = u_t(z) + \omega_t(x), \quad v_t(y) = u_t(\zeta) + \omega_t(y),
\]
where the fibre fluctuations satisfy
\[
\intr \omega_t(x) \nu_{t,z}(\dx) = 0, \quad \intr \omega_t(y) \nu_{t,\zeta}(\dy) = 0.
\]
Then, on each pair of fibres $(z,\zeta)$,
\begin{align*}
&\int_{\eta_t^{-1}(z)} \int_{\eta_t^{-1}(\zeta)} G_p (v_t(y)-v_t(x))\,\nu_{t,z}(\dx) \nu_{t,\zeta}(\dy) \cr
&\quad= \int_{\eta_t^{-1}(z)} \int_{\eta_t^{-1}(\zeta)}
G_p (u_t(\zeta)-u_t(z) + \omega_t(y)-\omega_t(x))\,\nu_{t,z}(\dx) \nu_{t,\zeta}(\dy).
\end{align*}
Adding and subtracting $G_p(u_t(\zeta)-u_t(z))$ inside the integral, we get
\[
\int_{\eta_t^{-1}(z)} \int_{\eta_t^{-1}(\zeta)} G_p (v_t(y)-v_t(x))\,\nu_{t,z}(\dx) \nu_{t,\zeta}(\dy) = G_p(u_t(\zeta)-u_t(z)) + \calK_t(z,\zeta),
\]
where $\calK$ is given as in \eqref{calK}. Hence, the antisymmetrized nonlocal term \eqref{new} equals tp
\[
\intr \varphi(t,z)\cdot \lt(\kappa\intr \phi(z-\zeta)
\lt( G_p (u_t(\zeta)-u_t(z)) + \calK_t(z,\zeta)\rt)\rho_t(\rd\zeta)\rt)\rho_t(\dz).
\]
 
Collecting (a)--(b) gives, for all $\varphi\in C_c^\infty((0,T)\times\R^d;\R^d)$,
\begin{align*}
\frac{\rd}{\dt}\intr \varphi \cdot \,m_t(\dz) &=\intr \pa_t\varphi\cdot m_t(\dz) +\intr \nabla\varphi: u_t\otimes u_t \rho_t(\dz) +\intr \nabla\varphi:  \tau_t(\dz) \cr
&\quad +\intr \varphi(t,z)\cdot \lt(\kappa\intr \phi(z-\zeta)
\lt( G_p (u_t(\zeta)-u_t(z)) + \calK_t(z,\zeta)\rt)\rho_t(\rd\zeta)\rt)\rho_t(\dz),
\end{align*}
and integrating in time over $(0,T)$ yields the distributional form of
\[
\pa_t(\rho u)+\nabla \cdot(\rho u\otimes u+\tau) =\rho (\calA_p[\rho,u]+\calR_p[\rho,u]).
\]
Since $T>0$ was arbitrary, both distributional identities hold on $(0,\infty)$.

\medskip
\noindent {\bf Step 3: Energy inequality.} 
Since $(\eta_t,v_t)$ is a global classical solution of the Lagrangian $p$-alignment system \eqref{Lag_p}, Lemma \ref{lem:apr} yields for all $t\geq 0$,
\begin{align}\label{lag-energy}
\begin{aligned}
&\frac12 \intr |v_t|^2\,\rho_0(\dx)+\frac{\kappa}{2} \int_0^t \intrr
\phi(\eta_s(x)-\eta_s(y)) |v_s(x)-v_s(y)|^p\,\rho_0(\dx)\rho_0(\dy)\ds \cr
&\quad \le \frac12 \intr |u_0|^2\,\rho_0(\dx).
\end{aligned}
\end{align}
 
Using the fibre decomposition
\[
v_t(x)=u_t(z)+\omega_t(x), \quad z=\eta_t(x),
\]
with zero fibre mean
\[
\intr \omega_t(x)\,\nu_{t,z}(\dx)=0,
\]
we compute
\[
|v_t(x)|^2=|u_t(z)|^2+2u_t(z)\cdot \omega_t(x)+|\omega_t(x)|^2.
\]
Integrating first with respect to $\nu_{t,z}$ and using the zero-mean property of $\omega_t$, we obtain
\[
\intr |v_t(x)|^2\,\nu_{t,z}(\dx) = |u_t(z)|^2 + \intr |\omega_t(x)|^2\,\nu_{t,z}(\dx).
\]
Integrating next with respect to $\rho_t(\dz)$ and recalling that
\[
\tr\,\theta_t(z)=\intr|\omega_t(x)|^2\,\nu_{t,z}(\dx),
\]
we arrive at
\begin{equation}\label{eq:v-decompose}
\intr  \tr\,\theta_t(z)\, \rho_t(\dz) = \intr |v_t(x)|^2\,\rho_0(\dx) - \intr |u_t(z)|^2\,\rho_t(\dz),
\end{equation}
and thus,
\[
\intr \left(|u_t|^2 + \tr\,\theta_t(z)\right)\rho_t(\dz) = \intr|v_t|^2\,\rho_0(\dx).
\]

Using the fibre-variance identity from Lemma \ref{lem:ERA-prep}, we have
\begin{equation}\label{eq:v-decompose}
\intr  {\rm tr}\,\theta_t(z)\, \rho_t(\dz) = \intr |v_t(x)|^2\,\rho_0(\dx) - \intr |u_t(z)|^2\,\rho_t(\dz),
\end{equation}
and thus,
\[
\intr \left(|u_t|^2+{\rm tr}\,\theta_t(z)\right)\rho_t(\dz) = \intr|v_t|^2\,\rho_0(\dx).
\]
Substituting this identity into \eqref{lag-energy}, while writing $\eta_r(x)=z$ and $\eta_r(y)=\zeta$ and disintegrating $\rho_0(\dx)\rho_0(\dy)$ along the fibres, yields the Eulerian energy inequality \eqref{ERA_energy}.

\medskip
\noindent {\bf Step 4: Regularity and uniqueness in the Lagrangian--compatible class.}
By Theorem \ref{thm:ext} and Lemma \ref{lem:cont},  $t\mapsto \rho_t$ is narrowly continuous and $t\mapsto m_t$ is weak$^\ast$ continuous. All structural properties of $u_t$ and $\tau_t$ are those listed in Lemma \ref{lem:ERA-prep}.

Finally, if $(\tilde\rho,\tilde u,\tilde\tau)$ is obtained from another Lagrangian solution with the same initial data $(\rho_0,u_0)$, uniqueness for the Lagrangian ODE (Theorem \ref{thm:ext}) gives $(\eta,v)=(\tilde\eta,\tilde v)$, hence $\rho=\tilde\rho$, $m=\tilde m$, and by the barycentric/variance identities $u=\tilde u$, $\tau=\tilde\tau$. This completes the proof.
\end{proof}

%
%
%
%
%

\section{Asymptotic closure of Euler--Reynolds--alignment to Euler--alignment under flocking}\label{sec:ac_ERA}

In the previous sections, we derived the ERA system induced by the global Lagrangian flow and identified the two defect terms preventing macroscopic closure: the Reynolds stress $\tau_t$ and, for $p>2$, the nonlinear defect force $\calR_p[\rho_t,u_t]$. Both terms arise from microscopic velocity fluctuations along Lagrangian fibres. Equivalently, they vanish precisely in the mono-kinetic regime, where the Eulerian velocity is uniquely determined at each spatial point.

The purpose of this section is twofold. First, we show that under the sole assumption of velocity flocking, the two defect terms in the Euler--Reynolds--alignment system, the Reynolds stress $\tau_t$ and the nonlinear defect force $\calR_p[\rho_t,u_t]$, vanish asymptotically as $t\to\infty$. This yields an asymptotic suppression of the obstruction to Eulerian macroscopic closure, without requiring injectivity of the flow map, spatial confinement, or additional Eulerian regularity.

Second, using the kinetic lifting of the Lagrangian flow and a compactness argument for time-translates, we show that any subsequential long-time limit is mono-kinetic and is transported by the conserved mean velocity $\bar u$.

%
%
%
%
%

\subsection{Decay of Reynolds stress and nonlinear defect force}

We begin by showing that velocity flocking forces the decay of both defect terms in the ERA system. More precisely, the uniform decay of the velocity diameter implies that velocities along the
same Lagrangian fibre become asymptotically indistinguishable. As a consequence, the fibrewise velocity variance encoded in $\tau_t$ vanishes, and the nonlinear defect force $\calR_p$, which measures the mismatch between nonlinear alignment interactions and their barycentric approximation, disappears as well.

The following lemma provides the quantitative formulation of this decay at the level of distributions.

\begin{lemma}\label{lem:asymp_closure}
Let $p \geq 2$, assume the hypotheses of Theorem \ref{thm:ext}, and let $(\rho_t,u_t,\tau_t)$ be the Eulerian objects induced by the Lagrangian solution as in Theorem \ref{thm:ERA}. Then the Reynolds stress vanishes asymptotically in the sense that
\[
\|\tau_t\|_{\calM(\R^d;\R^{d\times d})} \to 0 \quad \text{as }t\to\infty.
\]
In addition, when $p>2$, the nonlinear defect force also vanishes asymptotically:
\[
\|\rho_t \calR_p[\rho_t,u_t]\|_{\calM(\R^d;\R^d)} \to 0 \quad \text{as }t\to\infty.
\]
Moreover, the Eulerian velocity diameter decays to zero in the essential sense:
\[
\esssup_{(z,\zeta)\in S_t\times S_t}|u_t(z)-u_t(\zeta)|
\to 0 \quad \text{as }t\to\infty, \qquad S_t=\text{supp}\,\rho_t.
\]
\end{lemma}

\begin{proof}
We follow the systematization argument of \cite[\S4]{Tad23},\cite[\S3]{Tadpre}, expressing  the energy balance stated in Lemma \ref{lem:apr}, as an equivalent  statement for the decay of \emph{velocity fluctuations}
\[
\frac14 \frac{\rd}{\dt} \intrr |v_t(x)-v_t(y)|^2\,\rho_0(\dx)\rho_0(\dy) + \frac\kappa2\intrr \phi(\eta_t(x)-\eta_t(y))|v_t(x) - v_t(y)|^p \,\rho_0(\dx) \rho_0(\dy) =0.
\]
By H\"{o}lder inequality (recall that $\rho_0$ is a probability measure)
\[
\frac{\rd}{\dt} \delE(t) \leq -2\kappa \phi(\rd_\eta(t)) \delE^{p/2}(t), \qquad \delE(t):=\intrr |v_t(x)-v_t(y)|^2\,\rho_0(\dx)\rho_0(\dy),
\]
and in view of our assumption, $\delE(t) \stackrel{t\rightarrow \infty}{\longrightarrow}0$.
The result follows by noting that $\delE(t)$ quantifies fluctuations of the Eulerian velocities; specifically
\begin{equation}\label{eq:delE-Eulerian}
\delE(t)=\intrr \Big(|u_t(x)-u_t(y)|^2 + {\rm tr}\, \theta_t(x) + {\rm tr}\, \theta_t(y)\Big)\rho_t(\dx)\rho_t(\dy).
\end{equation}
Indeed, $\displaystyle 
\delE(t)=\int \limits_{\R^d}|v_t(x)|^2\rho_0(\dx) -2\int \limits_{\R^d}  v_t(x)\rho_0(\dx)\int \limits_{\R^d}  v_t(y)\rho_0(\dy)+\int \limits_{\R^d}|v_t(y)|^2\rho_0(\dy)$.
For the two quadratic terms we use  the energy decomposition in \eqref{eq:v-decompose}
\[
\intr|v_t(z)|^2\rho_0(\dz)=\intr|u_t(z)|^2\rho_t(\dz)+\intr {\rm tr}\,\theta_t(z)\rho_t(\dz),
\]
while for the mixed term we use the barycentric identity 
\eqref{eq:bary-identity-proof} with $F\equiv1$
\[
\intr  v_t(x)\rho_0(\dx)\intr  v_t(y)\rho_0(\dy)=\intr  u_t(x)\rho_t(\dx)\intr  u_t(y)\rho_t(\dy),
\]
and \eqref{eq:delE-Eulerian} follows.
Hence, $\delE(t) \stackrel{t\rightarrow \infty}{\longrightarrow}0$ implies   an asymptotic mono-kinetic closure,  $\displaystyle \|\tau_t\|_{\calM(\R^d;\R^{d\times d})}=\intr {\rm tr}\,\theta_t(x)\rho_t(\dx) \stackrel{t\rightarrow \infty}{\longrightarrow}0$ (as well as average asymptotic  flocking $\displaystyle \intrr |u_t(x)-u_t(y)|^2 \rho_t(\dx)\rho_t(\dy)\stackrel{t\rightarrow \infty}{\longrightarrow}0$). 

For $p>2$, there is an additional nonlinear defect  kernel $\calK_t(z,\zeta)$ in \eqref{calK} 
\[
\calK_t(z,\zeta) =\intrr \big(G_p(a+b)-G_p(a)\big)\,\nu_{t,\zeta}(\dy) \nu_{t,z}(\dx),
\]
which involves the two terms  $a:=u_t(\zeta)-u_t(z)$ and $b:=\omega_t(y)-\omega_t(x)$. Since $u_t$ is the barycenter of $v_t$ along each fibre,
\[
|a| =|u_t(\zeta)-u_t(z)| =\lt|\intrr (v_t(y)-v_t(x))\,\nu_{t,\zeta}(\dy) \nu_{t,z}(\dx)\rt|
\le \rd_v(t).
\]
Hence, we have
\[
\esssup_{(z,\zeta)\in S_t\times S_t}|u_t(z)-u_t(\zeta)|
\to 0 \quad \text{as }t\to\infty, \qquad S_t=\text{supp}\,\rho_t.
\]
Likewise, for $x\in\supp\nu_{t,z}$ and $y\in\supp\nu_{t,\zeta}$, $|\omega_t(x)| = |v_t(x)-u_t(z)| \le \rd_v(t)$, hence
\[
|b| = |\omega_t(y)-\omega_t(x)| \le 2\rd_v(t).
\]
By Theorem \ref{thm:ext}, we have  $\rd_v(t)\stackrel{t\rightarrow\infty}{\longrightarrow}0$, and hence $G_p(a+b)$ and $G_p(a)$, each is of the vanishing order $\rd_v^{p-1}(t)\stackrel{t\rightarrow \infty}{\longrightarrow}0$. Therefore, 
\[
|\calK_t(z,\zeta)| \le|G_p(a+b)-G_p(a)| \le C_p\, \rd_v^{p-1}(t),
\]
 which in turn implies
$\displaystyle 
\|\rho_t\calR_p[\rho_t,u_t]\|_{\calM(\R^d;\R^d)}\leq \|\calR_p[\rho_t,u_t]\|_{L^\infty} \lesssim \rd_v^{p-1}(t)\stackrel{t\rightarrow \infty}{\longrightarrow}0$.
\end{proof}
 
%
%
%
%
%

\subsection{Asymptotic mono-kinetic dynamics}

Lemma \ref{lem:asymp_closure} already proves the first assertion of Theorem \ref{thm:asymp_closure}, namely the asymptotic vanishing of the Reynolds stress and the nonlinear defect force. It remains to establish the second assertion, concerning the structure of subsequential long-time limits of the lifted kinetic measure.

To this end, we consider time-translates of the kinetic lifting and pass to the limit on a fixed time window. The key point is that velocity flocking forces the velocity marginal to collapse to the Dirac mass $\delta_{\bar u}$, so that every such long-time limit is mono-kinetic and transported by the conserved mean velocity.

\begin{proof}[Proof of Theorem \ref{thm:asymp_closure}]
The first assertion is exactly Lemma \ref{lem:asymp_closure}. We therefore only prove the second assertion. We proceed in several steps.

\medskip
\noindent\textbf{Step 1. Time translation and kinetic weak formulation on a fixed window.}
Fix $T>0$. For each $R>0$, introduce the translated lifted measure
\[
\mu^{(R)}_s := (\eta_{R+s},v_{R+s})_\#\rho_0, \quad 0\le s\le T.
\]
By construction, $\mu^{(R)}_s\in\calP(\R^d\times\R^d)$.

As shown in Appendix~\ref{app:kinetic}, the lifted measure $\mu_t$ associated with the Lagrangian dynamics satisfies the kinetic equation
\[
\partial_t\mu_t +\nabla_z\cdot(\xi\mu_t) +\nabla_\xi\cdot\bigl(F_p[\mu_t]\mu_t\bigr)=0 \quad\text{in }\calD'((0,\infty)\times\R^d\times\R^d),
\]
where
\[
F_p[f](z,\xi) = \kappa \iint_{\R^d\times\R^d} \phi(z-z')\,G_p(\xi'-\xi)\,f(\rd z',\rd\xi'),
\quad
G_p(w)=|w|^{p-2}w.
\]

Consequently, the translated measure $\mu^{(R)}$ satisfies the weak formulation
\begin{align}\label{eq:kinetic_weak_R}
\begin{aligned}
&\int_0^T\intrr \lt(-\partial_s\Phi(s,z)\cdot\xi -\nabla_z\Phi(s,z):(\xi\otimes\xi)\rt) \mu^{(R)}_s(\rd z,\rd\xi) \,\rd s \\ 
&\quad = \int_0^T\intrr \Phi(s,z)\cdot F_p[\mu^{(R)}_s](z,\xi) \,\mu^{(R)}_s(\rd z,\rd\xi) \,\rd s
\end{aligned}
\end{align}
for every $\Phi\in C_c^\infty((0,T)\times\R^d;\R^d)$.

In particular, taking test functions depending only on $(s,z)$ yields the spatial continuity equation
\[
\int_0^T\intrr
\lt(-\partial_s\varphi(s,z)-\xi\cdot\nabla_z\varphi(s,z)\rt)  \mu^{(R)}_s(\rd z,\rd\xi)\,\rd s=0
\]
for all $\varphi\in C_c^\infty((0,T)\times\R^d)$.
 
\medskip
\noindent\textbf{Step 2. Local compactness on the support of the test function.}
Let $K\subset\R^d$ be a compact set such that $\supp\Phi\subset (0,T)\times K$. Since we only test the kinetic formulation \eqref{eq:kinetic_weak_R} against $\Phi$, all space--time integrals involve only $(s,z)\in[0,T]\times K$.

By the $L^\infty$ maximum principle established in the proof of Theorem \ref{thm:ext} (see \eqref{eq:v-Linf-decay}), we have the uniform bound
\[
\|v_t\|_{L^\infty(\rho_0)} \le \|u_0\|_{L^\infty(\rho_0)}
\quad\text{for all } t\ge0,
\]
and hence $|v_t(x)|\le M$ for $\rho_0$-a.e.  $x$ with $M:=\|u_0\|_{L^\infty(\rho_0)}$. Consequently, for every $R>0$ and $s\in[0,T]$, the lifted measure
\[
\mu^{(R)}_s=(\eta_{R+s},v_{R+s})_\#\rho_0
\]
is supported in $\R^d\times\overline{B(0,M)}$, namely
\bq\label{eq:compact_support_xi}
\supp\mu^{(R)}_s \subset \R^d\times\overline{B(0,M)} .
\eq
In particular, for each $s\in[0,T]$, the restriction of $\mu^{(R)}_s$ to $K\times\R^d$ is a finite positive Radon measure supported in the compact set $K\times\overline{B(0,M)}$, and hence the family $\{\mu^{(R)}_s\}_{R>0}$ is bounded in $\calM(K\times\overline{B(0,M)})$.

Using \eqref{eq:compact_support_xi} and the fact that $\mu^{(R)}_s$ are probability measures, we may extract a sequence $R_n\to\infty$ and a limit
\[
\mu^\ast \in L^\infty\bigl(0,T;\calM(K\times\overline{B(0,M)})\bigr), 
\]
such that $\mu^{(R_n)}\rightharpoonup\mu^\ast$ weakly-$\ast$ in $L^\infty (0,T;\calM(K\times\overline{B(0,M)}) )$. That is, for every test function
$\psi\in L^1(0,T;C(K\times\overline{B(0,M)}))$, we have
\[
\int_0^T \intrr \psi(s,z,\xi)\,\mu^{(R_n)}_s(\dz,\rd\xi)\,\ds \to \int_0^T \intrr \psi(s,z,\xi)\,\mu^\ast_s(\dz,\rd\xi)\,\ds.
\]

For a.e.  $s\in(0,T)$, let $\rho^\ast_s$ denote the $z$--marginal of $\mu^\ast_s$. By the disintegration theorem (Theorem \ref{thm:disint}), there exists a $\rho^\ast_s$--a.e.  uniquely determined family of probability measures $\{\tilde\nu^\ast_{s,z}\}_{z\in K}\subset\calP(\R^d)$ such that
\[
\mu^\ast_s(\dz,\rd\xi)=\rho^\ast_s(\dz)\,\tilde\nu^\ast_{s,z}(\rd\xi).
\]

\medskip
\noindent\textbf{Step 3. Equicontinuity in time and compactness.} We show that the family $\{\mu^{(R)}\}_{R>0}$ is equicontinuous in time in the dual space $W^{-1,\infty}(K\times\overline{B(0,M)})$.

Let $\psi\in C^1(K\times\overline{B(0,M)})$. From the kinetic formulation \eqref{eq:kinetic_weak_R}, we obtain
\begin{align*}
\frac{\rd}{\rd s}\intrr\psi(z,\xi)\,\mu^{(R)}_s(\dz,\rd\xi) &= \intrr \xi\cdot\nabla_z\psi(z,\xi)\,\mu^{(R)}_s(\dz,\rd\xi)\cr
&\quad + \intrr F_p[\mu^{(R)}_s](z,\xi)\cdot\nabla_\xi\psi(z,\xi)\,\mu^{(R)}_s(\dz,\rd\xi).
\end{align*}
Since $|\xi|\le M$ on the support of $\mu^{(R)}_s$ and the alignment force 
$F_p[\mu^{(R)}_s]$ is uniformly bounded on $K\times\overline{B(0,M)}$, we deduce
\[
\lt|\frac{\rd}{\rd s}\intrr\psi(z,\xi)\,\mu^{(R)}_s(\dz,\rd\xi) \rt|
\le C\|\psi\|_{W^{1,\infty}(K\times\overline{B(0,M)})},
\]
with a constant $C >0$ independent of $R$.  Hence $\{\mu^{(R)}\}_{R>0}$ is equicontinuous in  $W^{-1,\infty}(K\times\overline{B(0,M)})$.

Combining this equicontinuity with the weak-$\ast$ compactness obtained in Step 2,
an Arzel\'a--Ascoli argument yields, up to extraction of a subsequence,
\[
\mu^{(R_n)}_s \rightharpoonup \mu^\ast_s
\quad\text{in }\calM(K\times\overline{B(0,M)})
\quad\text{for every } s\in[0,T].
\]

\medskip
\noindent\textbf{Step 4. Collapse of the velocity marginal.}
For each $R>0$ and $t\in[0,T]$, let
\[
\lambda^{(R)}_t := (\pi_\xi)_\#\mu^{(R)}_t = (v_{R+t})_\#\rho_0
\in \calP(\overline{B(0,M)})
\]
be the velocity marginal of $\mu^{(R)}_t$, where $\pi_\xi(z,\xi):=\xi$ denotes the projection onto the velocity variable.

Since $\pi_\xi$ is continuous and $\mu^{(R_n)}_t \rightharpoonup \mu^\ast_t$ narrowly in
$\calM(K\times\overline{B(0,M)})$ for every $t\in[0,T]$, it follows that
\[
\lambda^{(R_n)}_t \rightharpoonup \lambda^\ast_t := (\pi_\xi)_\#\mu^\ast_t
\quad\text{narrowly in }\calP(\overline{B(0,M)})
\]
for every $t\in[0,T]$.

We now show that $\lambda^\ast_t$ is in fact a Dirac mass concentrated at the conserved mean velocity. Since $\lambda^{(R_n)}_t=(v_{R_n+t})_\#\rho_0$, every two points in $\supp\lambda^{(R_n)}_t$
are of the form $v_{R_n+t}(x)$ and $v_{R_n+t}(y)$ for some $x,y\in\supp\rho_0$.
Hence
\[
\diam(\supp\lambda^{(R_n)}_t) \le \rd_v(R_n+t).
\]
Therefore, for every fixed $t\in[0,T]$,
\[
\diam(\supp\lambda^{(R_n)}_t)\to0 \quad\text{as }n\to\infty.
\]

Next we use conservation of total momentum to identify the center of concentration. By Lemma \ref{lem:apr}, 
the total momentum is conserved along the Lagrangian dynamics:
\[
\intr v_s(x)\,\rho_0(\dx)=\intr u_0(x)\,\rho_0(\dx)=:\bar u \quad\text{for all }s\ge0.
\]
Equivalently,
\[
\intr \xi\,\lambda^{(R_n)}_t(\rd\xi)=\bar u \quad\text{for every }n\in\N,\ t\in[0,T].
\]

We claim that
\[
\lambda^{(R_n)}_t \rightharpoonup \delta_{\bar u} \quad\text{narrowly in }\calP(\overline{B(0,M)})
\]
for every fixed $t\in[0,T]$.
To prove this, let $g\in C(\overline{B(0,M)})$ be arbitrary.
Since $\bar u$ is the barycenter of $\lambda^{(R_n)}_t$, and the support of $\lambda^{(R_n)}_t$
has diameter tending to zero, the whole support must concentrate around $\bar u$.
Indeed, for any $\xi\in\supp\lambda^{(R_n)}_t$,
\[
|\xi-\bar u|
=
\lt|\xi-\intr \xi'\,\lambda^{(R_n)}_t(\rd\xi')\rt|
\le
\intr |\xi-\xi'|\,\lambda^{(R_n)}_t(\rd\xi')
\le
\diam(\supp\lambda^{(R_n)}_t).
\]
Taking the supremum over $\xi\in\supp\lambda^{(R_n)}_t$, we obtain
\[
\supp\lambda^{(R_n)}_t \subset
\overline{B\bigl(\bar u,\diam(\supp\lambda^{(R_n)}_t)\bigr)}.
\]
Hence,
\[
\lt| \intr g(\xi)\,\lambda^{(R_n)}_t(\rd\xi)-g(\bar u) \rt|
\le
\sup_{|\xi-\bar u|\le \diam(\supp\lambda^{(R_n)}_t)}
|g(\xi)-g(\bar u)|.
\]
Since $g$ is uniformly continuous on the compact set $\overline{B(0,M)}$ and
$\diam(\supp\lambda^{(R_n)}_t)\to0$, the right-hand side converges to $0$.
Therefore,
\[
\intr g(\xi)\,\lambda^{(R_n)}_t(\rd\xi)\to g(\bar u),
\]
which proves the claim.

By the uniqueness of the narrow limit, we conclude that
\[
\lambda^\ast_t=\delta_{\bar u} \quad\text{for every }t\in[0,T].
\]

Finally, since $\lambda^\ast_t=(\pi_\xi)_\#\mu^\ast_t$ is the $\xi$-marginal of $\mu^\ast_t$
and is equal to the Dirac mass $\delta_{\bar u}$, the measure $\mu^\ast_t$ must be concentrated on the slice
$\R^d\times\{\bar u\}$. Hence, we have
\[
\mu^\ast_t(\rd z,\rd\xi)=\rho^\ast_t(\rd z)\,\delta_{\bar u}(\rd\xi)
\quad\text{for every }t\in[0,T].
\]

\medskip
\noindent\textbf{Step 5. Passage to the limit and identification of the limit dynamics.}
Let $\varphi\in C_c^\infty((0,T)\times\R^d)$ be arbitrary, and let
$K\subset\R^d$ be a compact set such that
\[
\supp\varphi\subset (0,T)\times K.
\]
Since $\mu^{(R_n)}$ is a weak solution of the kinetic equation, testing against
functions depending only on $(s,z)$ yields
\[
\int_0^T\intrr
\lt(-\partial_s\varphi(s,z)-\xi\cdot\nabla_z\varphi(s,z)\rt) \mu^{(R_n)}_s(\rd z,\rd\xi)\,\rd s=0.
\]
Indeed, the force term disappears because $\varphi$ is independent of $\xi$.

Now the function
\[
(s,z,\xi)\mapsto -\partial_s\varphi(s,z)-\xi\cdot\nabla_z\varphi(s,z)
\]
belongs to $L^1(0,T;C(K\times\overline{B(0,M)}))$.
Hence, by the weak-$\ast$ convergence obtained in Step 2, we may pass to the limit:
\[
\int_0^T\intrr \lt(-\partial_s\varphi(s,z)-\xi\cdot\nabla_z\varphi(s,z)\rt) \mu^\ast_s(\rd z,\rd\xi)\,\rd s=0.
\]
Using the monokinetic form from Step 4,
\[
\mu^\ast_s(\rd z,\rd\xi)=\rho^\ast_s(\rd z)\,\delta_{\bar u}(\rd\xi),
\]
we obtain
\[
\int_0^T\intr \lt(-\partial_s\varphi(s,z)-\bar u\cdot\nabla_z\varphi(s,z)\rt) \rho^\ast_s(\rd z)\,\rd s=0.
\]
Therefore, we have
\[
\partial_t\rho^\ast + \nabla_z \cdot (\rho^\ast \bar u)=0 \quad\text{in }\calD'((0,T)\times\R^d).
\]

Settting
\[
\rho^\ast_0 := (\pi_z)_\#\mu^\ast_0 \in \calP(\R^d),
\]
the above continuity equation with constant velocity $\bar u$ yields
\[
\rho^\ast_t=(z\mapsto z+t\bar u)_\#\rho^\ast_0 \quad\text{for all }t\in[0,T].
\]
As $T>0$ was arbitrary, the same representation holds for all $t\ge0$.
This completes the proof.
\end{proof}

 \begin{remark}[On the relation between $\nu_{t,z}$ and $\tilde\nu_{t,z}$]
Two different families of fibre measures appear in our analysis and play distinct but related roles, corresponding to two different disintegration
procedures.

The family $\{\tilde\nu_{t,z}\}_{z\in\R^d}$ arises from the disintegration of the kinetic lifted measure $(\eta_t,v_t)_\#\rho_0$ with respect to its spatial marginal $\rho_t=(\eta_t)_\#\rho_0$, namely
\[
(\eta_t,v_t)_\#\rho_0(\dz,\rd\xi)=\rho_t(\dz)\,\tilde\nu_{t,z}(\rd\xi).
\]
The measure $\tilde\nu_{t,z}\in\calP(\R^d)$ represents the conditional distribution of velocities at the Eulerian position $z$.

On the other hand, the family $\{\nu_{t,z}\}_{z\in\R^d}$ is obtained from the disintegration of the reference measure $\rho_0$ with respect to the flow map $\eta_t$, that is,
\[
\rho_0(\dx)=\intr\nu_{t,z}(\dx) \rho_t(\dz),
\]
and describes the distribution of Lagrangian labels $x$ such that $\eta_t(x)=z$. These two families are related through the velocity map by the pushforward identity
\[
\tilde\nu_{t,z} = (v_t)_\# \nu_{t,z}.
\]
In particular, if the Lagrangian flow map $\eta_t$ is injective, then for $\rho_t$-a.e.  $z$ the fibre $\nu_{t,z}$ reduces to the Dirac mass $\delta_{\eta_t^{-1}(z)}$, and consequently
\[
\tilde\nu_{t,z}=\delta_{v_t(\eta_t^{-1}(z))}=\delta_{u_t(z)}.
\]
In this case the kinetic measure is mono-kinetic and the Reynolds stress vanishes. Thus injectivity of $\eta_t$ is a sufficient condition for mono-kinetic closure.

The converse implication, however, does not hold in general. It may happen that $\tilde\nu_{t,z}$ is a Dirac mass even though $\nu_{t,z}$ is not, corresponding to the situation where several Lagrangian labels collide in space while carrying identical velocities. This distinction explains why mono-kinetic closure at the Eulerian level does not necessarily imply injectivity of the underlying Lagrangian flow.
\end{remark}


\section{One-dimensional Lagrange--alignment formulation}\label{sec:1DEA}
In this section, we specialize the Lagrangian formulation of the alignment dynamics to the one-dimensional case. We consider the system
\begin{align}\label{Lag_1}
\begin{aligned}
\pa_t \eta_t(x) &= v_t(x), \quad t >0, \ x \in \R, \\
\pa_t v_t(x) &= \kappa \int_\R \phi (\eta_t(x)-\eta_t(y))(v_t(y)-v_t(x))\,\rho_0(\dy).
\end{aligned}
\end{align}

Following \cite{CCKTpre, CZ21, HKPZ19}, it is natural to write the above system in the equivalent ``renormalized'' form. Recall the effective initial velocity
\[
\widehat v(x) = u_0(x)-\kappa \int_\R \Phi(y-x)\,\rho_0(\dy),
\]
where $\Phi\in C^2(\R)$ is a primitive of $\phi$. Then \eqref{Lag_1} rewrites as
\bq\label{eq:rLag}
\pa_t \eta_t(x) = \widehat v(x) + \kappa \int_\R \Phi (\eta_t(y)-\eta_t(x))\,\rho_0(\dy).
\eq
By Theorem \ref{thm:ext} the system \eqref{Lag_1}--\eqref{eq:rLag} admits a unique global Lagrangian solution $(\eta_t,v_t)$.


\subsection{Order preservation and injectivity}

We now characterize the exact condition under which the Lagrangian flow remains injective for all $t\ge 0$. Recall from Section \ref{sec:ERA} that injectivity of $\eta_t$ implies that the Reynolds stress tensor $\tau_t$ vanishes, so that the ERA system reduces to the mono-kinetic Euler--alignment equations.  

\begin{proposition}[Characterization of injectivity]\label{prop:inj_reg}
Let $(\eta_t,v_t)$ denote the global Lagrangian solution, whose flow component $\eta_t$ satisfies \eqref{eq:rLag}. Then the following are equivalent:
\begin{enumerate}
\item[(i)] $\eta_t$ is injective for all $t\ge 0$; equivalently, the flow is 
order preserving:
\[
x>y \quad\Longrightarrow\quad \eta_t(x)>\eta_t(y)\quad\forall\,t\ge 0.
\]

\item[(ii)] The effective velocity $\widehat v$ is non-decreasing:
\[
x>y \quad\Longrightarrow\quad \widehat v(x)\ge \widehat v(y).
\]
\end{enumerate}
\end{proposition}
\begin{proof}
\emph{(ii)\,$\Rightarrow$\,(i).}  
Fix $x>y$. From \eqref{eq:rLag}, we find
\begin{align*}
\pa_t (\eta_t(x) - \eta_t(y)) &= \widehat v(x) - \widehat v(y) + \kappa \int_\R \lt(\Phi(\eta_t(z)-\eta_t(x)) - \Phi(\eta_t(z)-\eta_t(y))\rt)\rho_0(\dz)\cr
&\geq -\kappa(\eta_t(x) - \eta_t(y)) \int_\R \phi(\xi_t(x,y,z))\,\rho_0(\dz),
\end{align*}
where $\xi_t(x,y,z)$ is between $\eta_t(z)-\eta_t(x)$ and $\eta_t(z)-\eta_t(y)$. Then, applying Gr\"onwall's lemma gives
\[
\eta_t(x) - \eta_t(y) \geq (x-y) \exp\lt(  -\kappa \int_0^t  \int_\R \phi(\xi_s(x,y,z))\,\rho_0(\dz) \ds\rt) \geq (x-y) e^{-\kappa \|\phi\|_{L^\infty} t}.
\]
Thus order is preserved.

\medskip
\emph{(i)\,$\Rightarrow$\,(ii).}
If $\widehat v$ is not non-decreasing, pick $x>y$ with 
$\widehat v(x)<\widehat v(y)$.  
For as long as $\eta_t(x) - \eta_t(y) > 0$, we obtain
\begin{align*}
\pa_t (\eta_t(x) - \eta_t(y)) &= \widehat v(x) - \widehat v(y) + \kappa \int_\R \lt(\Phi(\eta_t(z)-\eta_t(x)) - \Phi(\eta_t(z)-\eta_t(y))\rt)\rho_0(\dz)\cr
&\leq  \widehat v(x) - \widehat v(y),
\end{align*}
since $\Phi$ is nondecreasing.
This gives
\[
\eta_t(x) - \eta_t(y) \leq (x-y) + \lt( \widehat v(x) - \widehat v(y)\rt) t,
\]
and hence $\eta_t(x) - \eta_t(y)$ must reach $0$ in finite time, contradicting injectivity. This completes the proof.
\end{proof}
 
\begin{remark}[At most one collision per pair]\label{rmk:one_c}
Assume that $\widehat v$ is not non-decreasing, and fix $x>y$ with $\widehat v(x)<\widehat v(y)$.  Set
\[
d(t):=\eta_t(x)-\eta_t(y),\quad \Delta\widehat v:=\widehat v(x)-\widehat v(y)<0.
\]
Subtracting \eqref{eq:rLag} at $x$ and $y$ and using the mean value theorem (with $\Phi'=\phi$) yields, for all $t\ge0$,
\[
d'(t)
=\Delta\widehat v -\kappa \lt(\int_\R\phi(\xi_t(x,y,z))\,\rho_0(\dz)\rt) d(t) = \Delta\widehat v-\kappa A(t)\,d(t),
\]
where $\xi_t(x,y,z)$ lies between $\eta_t(z)-\eta_t(x)$ and $\eta_t(z)-\eta_t(y)$, and
\[
A(t):=\int_\R\phi (\xi_t(x,y,z) )\,\rho_0(\dz)\in[0,\|\phi\|_{L^\infty}].
\]
Thus $d$ solves the linear ODE
\[
d'(t)+\kappa A(t)\,d(t)=\Delta\widehat v.
\]
By the integrating factor formula, we get
\[
d(t)=e^{-\kappa\int_0^t A(\tau)\,d\tau}\lt( d(0)+\Delta\widehat v\int_0^t e^{\kappa\int_0^s A(\tau)\,d\tau}\,\ds\rt).
\]
Since $\Delta\widehat v<0$ and the integrand is strictly positive, the bracketed term is strictly decreasing in $t$. Consequently, it can vanish at most once, and therefore
\[
C_{x,y}:=\{t>0:\eta_t(x)=\eta_t(y)\}=\{t>0:d(t)=0\}
\]
contains at most one point. In particular, a second collision cannot occur.

Moreover, as long as $d(t)>0$ one has $d'(t)=\Delta\widehat v-\kappa A(t)d(t)\le \Delta\widehat v$, and hence the first collision time $t_c$ (if it exists) satisfies the explicit upper bound
\[
t_c \le \frac{d(0)}{-\Delta\widehat v}=\frac{x-y}{\widehat v(y)-\widehat v(x)}.
\]
\end{remark}


 \subsection{Reduction to Euler--alignment in the injective regime}
 
The one-dimensional structure allows a complete characterization of when the Lagrangian flow remains injective, using Proposition \ref{prop:inj_reg}. When the effective velocity $\widehat v$ is non-decreasing, the flow remains injective for all times, and the induced Eulerian state remains mono-kinetic, so that the Reynolds stress vanishes identically and the ERA system reduces to the Euler--alignment equations.

When the monotonicity of $\widehat v$ fails, collisions may occur, and the flow may lose injectivity. Remark \ref{rmk:one_c} shows that each ordered pair of Lagrangian labels can collide at most once. However, this property alone does not imply that the disintegration of $\rho_0$ along $\eta_t$ remains Dirac for almost every time. In general, several Lagrangian labels may map to the same Eulerian position while carrying different velocities, and the Eulerian description is then given by the ERA system with a nontrivial Reynolds stress.
 
\begin{proof}[Proof of Theorem \ref{thm:1DEA}]
If the effective velocity $\widehat v$ is non-decreasing, then by Proposition \ref{prop:inj_reg}, the flow map $\eta_t$ is injective for all $t\ge0$, the solution remains mono-kinetic, and the
Reynolds stress vanishes identically. Hence, by Theorem \ref{thm:ERA}, the Eulerian pair $(\rho_t,u_t)$ satisfies the mono-kinetic one-dimensional Euler--alignment system \eqref{1D_EA} in the sense of distributions.
\end{proof}

%
%
%
%
%

\section{Euler--alignment system}\label{sec:EA}

This section is devoted to the construction of weak solutions to the Euler--alignment system
\begin{align}\label{eq:EA}
\begin{aligned}
\pa_t\rho_t+\nabla \cdot(\rho_t u_t)&=0, \\
\pa_t(\rho_t u_t)+\nabla \cdot(\rho_t u_t\otimes u_t)
&=\rho_t \calA_p[\rho_t,u_t]
\end{aligned}
\end{align}
starting from the Lagrangian $p$-alignment dynamics \eqref{Lag_p}.

Given a Lagrangian solution $(\eta_t,v_t)$, the pushforward and disintegration procedure associates Eulerian variables $(\rho_t,u_t,\tau_t)$ solving the ERA system \eqref{eq:ERA}, where the Reynolds stress $\tau_t$ measures the fibrewise velocity dispersion generated by the disintegration of $\rho_0$ along the flow map $\eta_t$. In particular, injectivity of $\eta_t$ is a sufficient condition for $\tau_t$ to vanish, but the essential closure mechanism is the collapse of the disintegration to Dirac masses. The vanishing of $\tau_t$ is therefore the key mechanism allowing the reduction of \eqref{eq:ERA} to the closed Eulerian system \eqref{eq:EA}.

Our approach relies on quantitative control of the Lagrangian deformation. Since
\[
\nabla \eta_t(x) =  I_d + \int_0^t \nabla v_s(x)\,\ds,
\]
a sufficient condition for injectivity of the flow is that the deformation of the identity remains strictly smaller than unity in operator norm, i.e.,
\[
\lt\|\int_0^t \nabla v_s(\cdot)\,\ds \rt\|_{L^\infty(\rho_0)} < 1,
\]
which implies that the flow map $\eta_t$ is injective on this time interval. Consequently, the disintegration along $\eta_t$ is Dirac, so that both the Reynolds stress $\tau_t$ and the nonlinear defect force $\calR_p$ vanish identically, and the ERA system reduces to \eqref{eq:EA}.

We proceed as follows.  We first establish global well-posedness of the Lagrangian $p$-alignment system in $W^{1,\infty}(\rho_0)$ and derive a priori bounds on $\|\nabla \eta_t\|_{L^\infty(\rho_0)}$ and $\|\nabla v_t\|_{L^\infty(\rho_0)}$ (Theorem \ref{thm:Wp1} below).  These bounds yield local-in-time weak solutions to \eqref{eq:EA}.  We then specialize to the linear velocity coupling $p=2$ and show that, under a sufficiently large coupling strength, the injectivity condition holds globally in time, thereby proving Theorem \ref{thm:EA_global}.

%
%
%
%
%
\subsection{Global well-posedness for Lagrangian $p$-alignment formulation  in $W^{1,\infty}(\rho_0)$}  
We now prove the global well-posedness result for the Lagrangian $p$-alignment system mentioned above.

\begin{theorem} \label{thm:Wp1}
Let $p\ge2$ and assume $\phi \in C^1 \cap W^{1,\infty}(\R^d)$. For any $u_0\in W^{1,\infty}(\rho_0)$, the Lagrangian $p$-alignment system \eqref{Lag_p}--\eqref{ini:Lag_p} admits a unique global solution
\[
(\eta - {\rm id},v)\in C^1([0,\infty);W^{1,\infty}(\rho_0)\times W^{1,\infty}(\rho_0)).
\]
Moreover, we have
\[
\|\nabla \eta_t \|_{L^\infty(\rho_0)} +  \|\nabla v_t\|_{L^\infty(\rho_0)} \leq (1 + \|\nabla u_0\|_{L^\infty(\rho_0)}) e^{C_0 t}
\]
for some $C_0 > 0$ depending on $\kappa, p, \|\phi\|_{W^{1,\infty}}$, and $\|u_0\|_{L^\infty}$.
\end{theorem}

\begin{proof}
Local well-posedness in $W^{1,\infty}(\rho_0)$ follows by the same argument as in Theorem \ref{thm:ext}, thus it suffices to derive a priori estimates in  $\dot W^{1,\infty}(\rho_0)$.

Let $G_p(\xi):=|\xi|^{p-2}\xi$. For all $p\ge2$, one has
\[
|G_p(\xi)|\le |\xi|^{p-1},\quad |\nabla G_p(\xi)|\le C_p |\xi|^{p-2}\quad(\text{with }C_p\sim p-1).
\]
By Theorem \ref{thm:ext} (maximum principle), set 
\[
M:=\sup_{t\ge0}\|v_t\|_{L^\infty(\rho_0)}\le \|u_0\|_{L^\infty(\rho_0)}. 
\]
Thus, for all $x,y \in \supp \rho_0$, we have
\[
 |v_t(y)-v_t(x)|\le 2M, \quad \|(\nabla G_p)(v_t(y)-v_t(x))\|\le C_p (2M)^{p-2}.
\]

Differentiate the system \eqref{Lag_p} with respect to $x$, we find
\begin{align}\label{LA_grad}
\begin{aligned}
\pa_t \nabla \eta_t &= \nabla v_t, \cr
\pa_t \nabla v_t &= \kappa \intr \nabla \eta_t(x) \nabla\phi(\eta_t(x) - \eta_t(y)) G_p(v_t(y) - v_t(x))\,\rho_0(\dy) \cr
&\quad - \kappa \intr \phi(\eta_t(x) - \eta_t(y)) (\nabla G_p)(v_t(y) - v_t(x)) \nabla v_t(x)\,\rho_0(\dy)\cr
&=: I + II.
\end{aligned}
\end{align}
We then estimate $I$ and $II$ as
\begin{align*}
| I | &\leq \kappa \|\nabla \eta_t\|_{L^\infty(\rho_0)}\|\nabla \phi\|_{L^\infty} (2M)^{p-1},\cr
| II | &\leq \kappa \|\nabla v_t\|_{L^\infty(\rho_0)} \|\phi\|_{L^\infty} C_p (2M)^{p-2}.
\end{align*}
This yields
\[
\|\nabla \eta_t \|_{L^\infty(\rho_0)} +  \|\nabla v_t\|_{L^\infty(\rho_0)} \leq 1 + \|\nabla u_0\|_{L^\infty(\rho_0)} + C_0\int_0^t (\|\nabla \eta_s \|_{L^\infty(\rho_0)} +  \|\nabla v_s\|_{L^\infty(\rho_0)})\,\ds,
\]
and applying Gr\"onwall's lemma further gives
\[
\|\nabla \eta_t \|_{L^\infty(\rho_0)} +  \|\nabla v_t\|_{L^\infty(\rho_0)} \leq (1 + \|\nabla u_0\|_{L^\infty(\rho_0)}) e^{C_0 t}
\]
for some $C_0 > 0$ depending on $\kappa, p, \|\phi\|_{W^{1,\infty}}, M$. This completes the proof.
\end{proof}

%
%
%
%
%
 
\subsection{Local-in-time existence of Eulerian $p$-alignment formulation}  

As a direct consequence of the gradient bounds obtained in Theorem \ref{thm:Wp1}, there exists a time $T_*>0$ such that
\[
\lt\|\int_0^t \nabla v_s\,\ds \rt\|_{L^\infty(\rho_0)} < 1 \quad \forall\, t \in [0,T_*).
\]
On this time interval, the Lagrangian flow $\eta_t$ remains injective and the ERA system reduces to the Eulerian $p$-alignment equations \eqref{eq:EA}.  This yields the following local-in-time existence result.

\begin{theorem} \label{thm:EA_local}
Assume $p\geq 2$ and the hypotheses of Theorem \ref{thm:Wp1}. Let $(\rho_t,u_t)$ be the Eulerian pair associated with the Lagrangian flow $(\eta_t,v_t)$ through the Lagrangian--Eulerian correspondence described in Section \ref{ssec:LEc}. Then  $(\rho, u)$ is a local-in-time solution to the Eulerian $p$-alignment system \eqref{eq:EA} in the sense of distributions on $(0,T_*)\times\R^d$ for some $T_* > 0$, with initial data 
\[
(\rho_t, u_t)|_{t=0} = (\rho_0, u_0).
\]
Moreover, we have
\[
\rho \in C([0,T_*);\calP(\R^d)),\quad   u \in L^\infty(0,T_*; L^\infty(\rho)).
\]
In the class of Eulerian pairs arising from the Lagrangian flow of Theorem \ref{thm:Wp1}, this solution is unique.
\end{theorem}
%
%
%
%
%
 \subsection{Global-in-time existence of Euler--alignment system}  
We now turn to the proof of global-in-time existence for the Euler--alignment system in the case of linear velocity coupling $p=2$. As discussed above, global solvability at the Eulerian level reduces to establishing global injectivity of the Lagrangian flow $\eta_t$. Our strategy is therefore to obtain an \emph{integrability-in-time} estimate for $\|\nabla v_t\|_{L^\infty(\rho_0)}$, which guarantees that the deformation of the flow remains uniformly small for all times.

The argument proceeds in two steps. We first establish a technical Gr\"onwall-type lemma for a coupled system of differential inequalities. We then show that, under a sufficiently large coupling strength $\kappa$, the gradient system associated with the Lagrangian dynamics fits precisely into this framework, yielding the desired integrability. This proves Theorem \ref{thm:EA_global}.

 \begin{lemma}\label{lem:tech} Let $f,g \in C^1(\R_+; \R_+)$ satisfy the following a system of differential inequalities:
\begin{align*}
f'(t) &\leq g(t),\quad t > 0,\cr
g'(t) &\leq -a g(t) + b f(t) e^{-at},
\end{align*}
where $a,b > 0$. If $a > \sqrt{b}$, we have
\[
\int_0^\infty g(t)\,\dt \leq \frac{ag(0) + bf(0)}{(a^2-b)}. 
\]
 \end{lemma}
 \begin{remark}\label{rmk_tech0} The same differential inequality was studied in \cite[Lemma 3.1]{HKZ18} to get the exponential decay of $g$ without the assumption $a > \sqrt b$. However, for simplicity, we obtain a sharper upper bound on $g$ under this assumption.
 \end{remark}
 \begin{proof}[Proof of Lemma \ref{lem:tech}] Observe that $f,g$ satisfy
 \[
 f(t) \leq f(0) + \int_0^t g(s)\,\ds, \quad g(t) \leq g(0) e^{-at} + b e^{-at} \int_0^t f(s)\,\ds.
 \]
 Thus, we obtain
\begin{align*}
 g(t) &\leq g(0) e^{-at} + b e^{-at} \int_0^t \lt(f(0) + \int_0^s g(\tau)\,\rd\tau \rt) \ds \cr
 &= g(0) e^{-at} + b f(0) t e^{-at} + b e^{-at} \int_0^t (t-\tau) g(\tau)\,\rd\tau.
\end{align*}
Integrating the above over $[0,t]$ gives
\begin{align*}
\int_0^t g(s)\,\ds &\leq \frac{g(0)}{a} (1 - e^{-at}) + \frac{bf(0)}{a^2}(1 - e^{-at}(1+at)) + b\int_0^t e^{-as} \int_0^s(s-\tau) g(\tau)\,\rd\tau \ds\cr
&=: I_t + II_t + III_t,
\end{align*}
where we estimate $III_t$ as
\[
III_t = b \int_0^t g(\tau) \int_\tau^t e^{-as} (s-\tau)\,\ds \rd\tau = \frac{b}{a^2}\int_0^t g(s) (e^{-as} - e^{-at} - a(t-s)e^{-at})\,\ds \leq \frac{b}{a^2}\int_0^t g(s) \,\ds.
\]
This gives
\[
\lt(1 - \frac{b}{a^2} \rt)\int_0^t g(s)\,\ds \leq I_t + II_t.
\]
This together with
\[
I_t \to \frac{g(0)}{a}, \quad II_t \to \frac{bf(0)}{a^2} \quad \text{as } t \to \infty,
\]
concludes the desired result.
 \end{proof}
 
We now apply Lemma \ref{lem:tech} to the gradient system associated with the Lagrangian $p$-alignment dynamics in the linear case $p=2$.

\begin{lemma}\label{lem:uni_g} Assume $p=2$ and the hypotheses of Theorem \ref{thm:Wp1}. Suppose that the initial data satisfy
\[
\rd_v(0) <  \kappa \int_{\rd_\eta(0)}^\infty \phi(r)\,\dr, \quad   \|\nabla\phi\|_{L^\infty} \rd_v(0) < \kappa \phi^2(\rd_\eta^\infty),
\]
where $\rd_\eta^\infty > 0$ is given by the relation,
\[
\rd_v(0) =   \kappa \int_{\rd_\eta(0)}^{\rd_\eta^\infty} \phi(r)\,\dr.
\]
 Then we have
 \[
 \int_0^\infty \|\nabla v_t\|_{L^\infty(\rho_0)}\,\dt \leq \frac{\phi(\rd_\eta^\infty)\|\nabla u_0\|_{L^\infty(\rho_0)}+ \|\nabla\phi\|_{L^\infty} \rd_v(0)}{\kappa \phi^2(\rd_\eta^\infty) - \|\nabla\phi\|_{L^\infty} \rd_v(0)}.
 \]
\end{lemma}

\begin{remark} Note that 
\[
\rd_v(0) =   \kappa \int_{\rd_\eta(0)}^{\rd_\eta^\infty} \phi(r)\,\dr
 \] 
 implies that $\rd_\eta^\infty \to \rd_\eta(0)$ as $\kappa \to \infty$. On the other hand, 
 \[
   \|\nabla\phi\|_{L^\infty} \rd_v(0) < \kappa \phi^2(\rd_\eta^\infty) = \frac{\rd_v(0) \phi^2(\rd_\eta^\infty)}{\int_{\rd_\eta(0)}^{\rd_\eta^\infty} \phi(r)\,\dr},
\]
and the right-hand side diverges to $+\infty$ as $\kappa \to \infty$. This shows that the conditions in Lemma \ref{lem:uni_g} are satisfied for sufficiently large $\kappa > 0$.
\end{remark}

\begin{proof}[Proof of Lemma \ref{lem:uni_g}] It follows from \eqref{LA_grad} that
\begin{align*}
\pa_t \nabla \eta_t &= \nabla v_t, \cr
\pa_t \nabla v_t &= \kappa \intr \nabla \eta_t(x) \nabla\phi(\eta_t(x) - \eta_t(y)) (v_t(y) - v_t(x))\,\rho_0(\dy)   - \kappa \intr \phi(\eta_t(x) - \eta(y))   \rho_0(\dy)\nabla v_t(x)\,.
\end{align*}
We first readily find
\[
D^+ \|\nabla \eta_t\|_{L^\infty(\rho_0)} \leq \|\nabla v_t\|_{L^\infty(\rho_0)},
\]
where $D^+$ denotes the upper right Dini derivative. The justification of this $L^\infty$ estimate follows from the same $\varepsilon$-maximizer argument used in the proof of Theorem \ref{thm:ext}.

For $\|\nabla v_t\|_{L^\infty(\rho_0)}$, we note that 
\[
\lt| \intr \nabla \eta(x) \nabla\phi(\eta(x) - \eta(y))  (v(y) - v(x))\,\rho_0(\dy)\rt| \leq \|\nabla \eta_t\|_{L^\infty(\rho_0)}\|\nabla\phi\|_{L^\infty} \rd_v.
\] 
Moreover, since $|\eta(x)-\eta(y)|\le \rd_\eta^\infty$ and $\phi$ is non-increasing,
\[
\intr \phi(\eta(x)-\eta(y))\,\rho_0(\dy)\ge \phi(\rd_\eta^\infty),
\]
so that the second term provides a linear damping of $\nabla v$. This together with Theorem \ref{thm:ext} yields
\[
D^+ \|\nabla v_t\|_{L^\infty(\rho_0)} \leq \kappa \|\nabla\phi\|_{L^\infty}\|\nabla \eta_t\|_{L^\infty(\rho_0)} \rd_v(0) e^{- \kappa \phi(\rd_\eta^\infty) t} - \kappa \phi(\rd_\eta^\infty)\|\nabla v_t\|_{L^\infty(\rho_0)}.
\] 
Thus, if we set $c_1 :=  \|\nabla\phi\|_{L^\infty} \rd_v(0)$, we arrive at
\begin{align*}
D^+ \|\nabla \eta_t\|_{L^\infty(\rho_0)} &\leq \|\nabla v_t\|_{L^\infty(\rho_0)}, \cr
D^+ \|\nabla v_t\|_{L^\infty(\rho_0)} &\leq - \kappa \phi(\rd_\eta^\infty)\|\nabla v_t\|_{L^\infty(\rho_0)} + \kappa c_1 \|\nabla \eta_t\|_{L^\infty(\rho_0)}   e^{- \kappa \phi(\rd_\eta^\infty) t} .
\end{align*}
Since the assumption on the initial data implies
\[
\kappa \phi(\rd_\eta^\infty) > \sqrt{\kappa c_1},
\]
applying Lemma \ref{lem:tech}, we have the desired bound estimate.
  \end{proof}

 \begin{remark}[On the case $p\in(2,3)$]\label{rem:p23}
The argument of Lemma \ref{lem:uni_g} is specific to the linear velocity coupling $p=2$,  where the gradient system for $(\nabla\eta_t,\nabla v_t)$ closes with a constant damping rate and an exponentially decaying source term.  This allows us to apply Lemma \ref{lem:tech} and deduce that
\[
\int_0^\infty \|\nabla v_t\|_{L^\infty(\rho_0)}\,\dt < \infty
\]
under suitable assumptions on the configurations.

When $p\in(2,3)$, Theorem \ref{thm:ext} still provides flocking and the algebraic decay of the velocity diameter,
\[
\rd_v(t) \lesssim \big(1+t\big)^{-\frac{1}{p-2}}.
\]
However, the proof of this flocking estimate only yields \emph{upper} bounds on the averaged alignment modulus
\[
\intr\phi(\eta_t(x)-\eta_t(y)) |v_t(y)-v_t(x)|^{p-2}\,\rho_0(\dy),
\]
and does not provide any uniform \emph{lower} bound that could play the role of a time-dependent damping coefficient in the gradient system. In particular, we are not able to deduce an inequality of the type
\[
D^+\|\nabla v_t\|_{L^\infty(\rho_0)} \le -\mu(t) \|\nabla v_t\|_{L^\infty(\rho_0)} + \text{(integrable source)}, \quad  \mu(t)\gtrsim \frac{1}{1+t},
\]
which would be the natural analogue of Lemma \ref{lem:tech} in the nonlinear regime.

As a consequence, for $p\in(2,3)$ the flocking estimate alone does not allow us to decide whether
\[
\int_0^\infty \|\nabla v_t\|_{L^\infty(\rho_0)}\,\dt
\]
is finite or infinite in general. Establishing either integrability or non-integrability of $\|\nabla v_t\|_{L^\infty(\rho_0)}$ for $p\in(2,3)$ would require additional non-degeneracy assumptions on the distribution of velocities, beyond the diameter decay provided by Theorem \ref{thm:ext}.
\end{remark}
 
 \begin{proof}[Proof of Theorem \ref{thm:EA_global}]Under the assumptions of Theorem \ref{thm:EA_global}, we obtain from Lemma \ref{lem:uni_g} that 
  \[
 \int_0^\infty \|\nabla v_t\|_{L^\infty(\rho_0)}\,\dt \leq \frac{\phi(\rd_\eta^\infty)\|\nabla u_0\|_{L^\infty(\rho_0)}+ \|\nabla\phi\|_{L^\infty} \rd_v(0)}{\kappa \phi^2(\rd_\eta^\infty) - \|\nabla\phi\|_{L^\infty} \rd_v(0)} < 1.
 \]
Consequently, we have
\[
\Big\|\int_0^t \nabla v_s\,\ds\Big\|_{L^\infty(\rho_0)} < 1
\quad \text{for all } t \ge 0,
\]
and the Lagrangian flow $\eta_t$ remains injective for all times. Hence the induced disintegration is Dirac, the Reynolds stress vanishes identically, and the associated Eulerian pair $(\rho_t,u_t)$ solves the Euler--alignment system globally in time. This completes the proof.
 \end{proof}

%
%
%
%
%

 \section{Uniform-in-time mean-field limits}\label{sec:uni_mf}
 
This section is devoted to the proofs of Theorems \ref{thm:deri_LA} and \ref{thm:mf}. Our analysis proceeds in two steps.

First, we establish a uniform-in-time stability estimate between the $N$-particle Cucker--Smale system \eqref{CS} and the limiting nonlinear Lagrangian dynamics \eqref{Lag_p} in the case of linear velocity coupling, i.e.  $p=2$. This step is formulated entirely at the level of characteristics and does not require any injectivity assumption on the Lagrangian flow. The resulting estimate yields a quantitative, uniform-in-time control of the error between the particle trajectories and the limiting Lagrangian flow.

Second, we combine this Lagrangian stability estimate with the Lagrangian--Eulerian correspondence of Section \ref{ssec:ERA} to obtain a uniform-in-time mean-field convergence toward a mono-kinetic Eulerian state. At this stage, an almost everywhere Dirac disintegration condition along the Lagrangian flow is required in order to identify the kinetic limit with an Eulerian pair $(\rho_t,u_t)$, which satisfies the Euler--alignment system \eqref{eq:EA-2} in the sense of distributions.

%
%
%
%
%

\subsection{Uniform-in-time stability to Lagrange--alignment flow}
In this subsection, we prove the uniform-in-time stability estimate at the level of Lagrangian characteristics, which constitutes the first step in the proof of Theorem \ref{thm:deri_LA}. The argument relies on a simple system of differential inequalities satisfied by the modulated Lagrangian quantities, whose abstract structure is isolated in the following lemma.

\begin{lemma}\label{lem1} Let $f,g \in C^1(\R_+; \R_+)$ and $h \in L^1(\R_+;\R_+)$ satisfy the following system of differential inequalities:
\begin{align*}
f'(t) &\leq g(t),\quad t > 0,\cr
g'(t) &\leq -c_0 g(t) + h(t) f(t)
\end{align*}
for some $c_0 > 0$. Then there exists $C = C(c_0, \|h\|_{L^1})$ such that
\[
f(t) + g(t) \leq C (f(0) + g(0)), \quad t \geq 0.
\] 
Moreover, we have
\[
g(t) \to 0 \quad \mbox{as } t \to \infty.
\]
\end{lemma}

\begin{proof}We introduce the Lyapunov-type functional
\[
\calL (t) := f(t) + \frac{1}{c_0}g(t).
\]
Using the differential inequalities satisfied by $f$ and $g$, we compute
\[
\frac{\rd}{\dt}\calL \leq g+\frac{1}{c_0}  \lt(-c_0 g+h f\rt) =\frac{1}{c_0} hf \leq \frac{1}{c_0} h\calL.
\]
An application of Gr\"onwall's inequality yields
\[
f(t) + \frac1{c_0}g(t) \leq \lt( f(0) + \frac1{c_0}g(0)\rt)e^{\frac1{c_0}\|h\|_{L^1}}.
\]
As a consequence, we have
\[
f(t) + g(t) \leq \max\lt\{c_0, \frac1{c_0}\rt\}e^{\frac1{c_0}\|h\|_{L^1}}\lt(f(0) + g(0)\rt).
\]
 
To analyze the large-time behavior of $g$, we first note that the above bound implies the uniform estimate
\[
f(t) \leq M_0, \quad t \geq 0,
\]
for some constant $M_0>0$. We then consider the auxiliary function
\[
\ell(t) := g(t) + \int_t^\infty h(s) f(s) \,\ds. 
\]
Since $f$ is uniformly bounded in time, $g \in C^1(\R_+)$, and $h \in L^1(\R_+)$, $\ell:\R_+ \to \R_+$ is well-defined and differentiable. Moreover, we obtain
\[
\ell'(t) = g'(t) - h(t) f(t) \leq -c_0 g(t) \leq 0,
\]
and thus, $\ell$ is decreasing, i.e. $\ell(t) \downarrow \ell_\infty$ as $t \to \infty$ for some $\ell_\infty \geq 0$. We claim that $\ell_\infty = 0$. Suppose not, i.e., $\ell_\infty > 0$. Then for $ \ell_\infty > \e >0$, there exists $T >0$ such that
\[
\int_T^\infty h(t) f(t) \,\dt \leq \frac\e2
\]
since $h\in L^1(\R_+)$ and $f$ is bounded. This gives that for $t \geq T$
\[
g(t) = \ell(t) - \int_t^\infty h(s) f(s) \,\ds \ge \ell_\infty - \frac\e2 \ge \frac{\ell_\infty}2.
\]
Thus,
\[
\ell'(t) \leq -c_0 g(t)  \leq - \frac{c_0 \ell_\infty}{2}.
\]
This implies that $\ell(t)$ becomes negative in finite time, which contradicts the fact that $\ell(t)\ge0$ for all $t\ge0$. Therefore $\ell_\infty=0$, and since $g(t)\le \ell(t)$, we conclude that $g(t)\to0$ as $t\to\infty$.
\end{proof}

\begin{remark}
As mentioned in Remark \ref{rmk_tech0}, a similar estimate was obtained  in \cite[Lemma 3.1]{HKZ18} for $h(t) = \gamma e^{-c_0 t}$ for $\gamma > 0$. For our purpose, we refine the proof and relax the assumption on $h$.
\end{remark}

We now apply Lemma \ref{lem1} to the modulated Lagrangian quantities introduced in Section \ref{sssec:mf}, which measure the discrepancy between the particle system and the limiting Lagrangian flow. For $q\ge1$, these quantities are defined by
\begin{align*}
\mathscr{E}_q(X^N|\eta)(t)  &= \lt( \frac1N\sum_{i=1}^N \intr |x_i(t) - \eta_t(x)|^q \,\rho_0(\dx)\rt)^\frac1q, \cr
\mathscr{E}_q(V^N|v)(t) &= \lt(\frac1N\sum_{i=1}^N \intr |v_i(t) - v_t(x)|^q \,\rho_0(\dx)\rt)^\frac1q.
\end{align*}
We also assume 
\[
\frac1N \sum_{i=1}^N v_i(0) = \intr  u_0\,\rho_0(\dx).
\]
  
With these preparations in hand, we are ready to prove the uniform-in-time Lagrangian stability estimate stated in Theorem \ref{thm:deri_LA}.

\begin{proof}[Proof of Theorem \ref{thm:deri_LA}]
 Differentiating and using H\"older's inequality, we get
\begin{align*}
\frac{\rd}{\dt} \mathscr{E}_q^q(X^N|\eta)(t) &= \frac qN\sum_{i=1}^N \intr |x_i(t) - \eta_t(x)|^{q-2}(x_i(t) - \eta_t(x)) \cdot (v_i(t) - v_t(x))\,\rho_0(\dx) \cr
& \leq q\mathscr{E}_q^{q-1}(X^N|\eta)(t) \mathscr{E}_q(V^N|v)(t),
\end{align*}
and hence
\bq\label{eq_pos}
\frac{\rd}{\dt} \mathscr{E}_q(X^N|\eta)(t) \leq \mathscr{E}_q(V^N|v)(t).
\eq

We next estimate the velocity error. Differentiating $\mathscr{E}_q^q(V^N|v)$ gives
\begin{align*}
&\frac{\rd}{\dt} \mathscr{E}_q^q(V^N|v)(t) \cr
&\quad = \kappa\frac qN \sum_{i=1}^N \intr |v_i - v_t(x)|^{q-2} (v_i - v_t(x)) \cr
&\hspace{2cm} \cdot \lt\{\frac1N \sum_{j=1}^N\phi (x_j - x_i)(v_j - v_i) - \intr \phi(\eta_t(x) - \eta_t(y)) (v_t(y) - v_t(x))\,\rho_0(\dy) \rt\}\,\rho_0(\dx)\cr
&\quad = \kappa\frac q{N^2} \sum_{i,j=1}^N \intrr  |v_i - v_t(x)|^{q-2}(v_i - v_t(x))\cdot (v_j - v_i)  \cr
&\hspace{4cm} \times  \lt(\phi(x_i - x_j) - \phi(\eta_t(x) - \eta_t(y)) \rt)  \,\rho_0(\dx)\rho_0(\dy) \cr
&\quad + \kappa\frac q{N^2} \sum_{i,j=1}^N \intrr  |v_i - v_t(x)|^{q-2}(v_i - v_t(x)) \cr
&\hspace{4cm} \cdot   \phi(\eta_t(x) - \eta_t(y))\lt\{ (v_j - v_i) - (v_t(y) - v_t(x)) \rt\}  \rho_0(\dx)\rho_0(\dy) \cr
&\quad=: I + II,
\end{align*}
 where we use Lipschitz continuity of $\phi$ to obtain
 \begin{align*}
 |I| &\leq \kappa  \|\phi\|_{\rm Lip} \rd_{V^N}(t)\frac q{N^2} \sum_{i,j=1}^N \intrr |v_i - v_t(x)|^{q-1} \lt( |x_i - \eta_t(x)| + |x_j  - \eta_t(y)|\rt)\,\rho_0(\dx)\rho_0(\dy)\cr
 &\leq  \kappa q \|\phi\|_{\rm Lip} \rd_{V^N}(t) \mathscr{E}_q(X^N|\eta)(t)  \mathscr{E}_q^{q-1}(V^N|v)(t).
 \end{align*}

Symmetrizing in $(i,j)$ and $(x,y)$ the expression of $II$ and using 
\[
\phi\ge\phi_m:= \inf_{t \geq 0}\phi(\rd_\eta(t))>0\,,
\] since by assumption \eqref{asp0}, we can ensure that
$$
\sup_{t\ge0} \rd_\eta(t)<\infty\,,
$$ 
together with the monotonicity
  \[
( |x|^{q-2}x - |y|^{q-2}y) \cdot (x-y) \geq 0, \quad x,y \in \R^d, \ q> 1,
  \]
we estimate $II$ as
 \begin{align*}
 II &= -\kappa\frac q{2N^2} \sum_{i,j=1}^N \intrr \phi(\eta_t(x) - \eta_t(y)) \lt\{(v_i - v_t(x)) - (v_j - v_t(y))\rt\} \cr
& \hspace{3cm} \cdot \lt\{(v_i - v_t(x))|v_i - v_t(x)|^{q-2} - (v_j - v_t(y))|v_j - v_t(y)|^{q-2} \rt\}\,\rho_0(\dx)\rho_0(\dy)\cr
 &\leq -\kappa\frac{ q\phi_m}{2N^2} \sum_{i,j=1}^N \intrr  \lt\{(v_i - v_t(x)) - (v_j - v_t(y))\rt\} \cr
& \hspace{3cm} \cdot \lt\{(v_i - v_t(x))|v_i - v_t(x)|^{q-2} - (v_j - v_t(y))|v_j - v_t(y)|^{q-2} \rt\}\rho_0(\dx)\rho_0(\dy)\cr
 &\leq -\kappa\frac{ q\phi_m}{N^2} \sum_{i,j=1}^N \intrr  \lt\{(v_i - v_t(x)) - (v_j - v_t(y))\rt\}  \cdot \lt\{(v_i - v_t(x))|v_i - v_t(x)|^{q-2}  \rt\} \rho_0(\dx)\rho_0(\dy)\cr
 &= - \kappa q\phi_m  \mathscr{E}_q^q(V^N|v)(t).
  \end{align*}
  
Combining the above estimates yields
\[
\frac{\rd}{\dt}\mathscr{E}_q(V^N|v)(t) \leq -\kappa\phi_m  \mathscr{E}_q(V^N|v)(t) + \kappa  \|\phi\|_{\rm Lip} \rd_{V^N}(t) \mathscr{E}_q(X^N|\eta)(t).
\]
Together with \eqref{eq_pos}, and applying Lemma \ref{lem1}, we deduce
\[
\mathscr{E}_q(X^N|\eta)(t) + \mathscr{E}_q(V^N|v)(t) \leq C\lt(\mathscr{E}_q(X^N|\eta)(0) + \mathscr{E}_q(V^N|v)(0) \rt),
\]
where $C>0$ is independent of $q$, $N$, and $t$. This completes the proof.
 \end{proof}

%
%
%
%
%
\subsection{Uniform-in-time mean-field limit to Euler--alignment system}
In this part, we convert the uniform-in-time Lagrangian stability estimates established in the previous subsection into uniform-in-time mean-field convergence results at the Eulerian level. The key step consists in relating the modulated Lagrangian quantities to Wasserstein distances between the empirical particle measure and the corresponding Eulerian state.

We first derive a basic Wasserstein estimate under the mono-kinetic reduction induced by the Dirac structure of the disintegration along the Lagrangian flow map $\eta_t$, in which case the Eulerian state is represented by the measure $\rho_t\otimes\delta_{u_t}$.

\begin{lemma}\label{lem:d_pZ_p}
Let $q \in [1,\infty]$, and let $(\rho_t,u_t)$ be the Eulerian pair associated with the Lagrangian flow $(\eta_t,v_t)$ through the Lagrangian--Eulerian correspondence described in Section \ref{ssec:LEc}. Assume that $\mathscr{E}_q(X^N|\eta)(t) < \infty$ and $\mathscr{E}_q(V^N|v)(t) < \infty$.   Suppose in addition that the disintegration of $\rho_0$ along $\eta_t$ is Dirac $\rho_t$-a.e.  Then we have
\[
{\rd}_q \lt(\mu_t^N, \rho_t\otimes\delta_{u_t}\rt)  \le 2^{1 - \frac1q}\lt(\mathscr{E}_q(X^N|\eta)(t)^q+\mathscr{E}_q(V^N|v)(t)^q\rt)^{\frac1q},
\]
with the usual interpretation when $q=\infty$.
\end{lemma}
\begin{proof}
Consider the product coupling
\[
\pi := (\rho_t\otimes\delta_{u_t})\otimes\mu_t^N \in \Pi(\rho_t\otimes\delta_{u_t},\mu_t^N).
\]
Since $|(x-y,v-w)|^q \le 2^{q-1}(|x-y|^q+|v-w|^q)$, we have
\begin{align*}
{\rd}_q^q \lt(\mu_t^N, \rho_t\otimes\delta_{u_t}\rt)
&\le 2^{q-1}\iiint_{\R^d \times \R^d \times \R^d} \lt(|x-y|^q+|u_t(x)-w|^q\rt) \rho_t(dx)\mu_t^N(\dy,\dw)\cr
&\le 2^{q-1}\intrr |\eta_t(x)-y|^q \,\rho_0(dx)\rho_t^N(\dy) \cr
&\quad + 2^{q-1}\iiint_{\R^d \times \R^d \times \R^d} |u_t(\eta_t(x))-w|^q\, \rho_0(\dx)\mu_t^N(\dy,\dw)\cr
& = 2^{q-1}\lt(\mathscr{E}_q(X^N|\eta)(t)^q + \mathscr{E}_q(V^N|v)(t)^q\rt).
\end{align*}
This completes the proof.
\end{proof}

\begin{remark}
Lemma \ref{lem:d_pZ_p} is specific to the mono-kinetic regime, in which the
disintegration of $\rho_0$ along $\eta_t$ is Dirac and the velocity field satisfies $v_t(x)=u_t(\eta_t(x))$ for $\rho_0$-a.e.  $x$. If the disintegration is not Dirac, the fibre decomposition
\[
v_t(x)=u_t(\eta_t(x))+\omega_t(x), \quad  \intr \omega_t\,\nu_{t,z}(\dx)=0,
\]
induces a nontrivial Reynolds stress $\tau_t=\rho_t\,\theta_t$ and a nonlinear defect force $\calR_p[\rho_t,u_t]$ in the ERA system of Theorem \ref{thm:ERA}. In this general case, the microscopic state generated by the Lagrangian flow is the kinetic measure $(\eta_t,v_t)\#\rho_0$, rather than the mono-kinetic ansatz $\rho_t\otimes\delta_{u_t}$, and the above argument cannot be used directly to control $\rd_q(\mu_t^N,\rho_t\otimes \delta_{u_t})$. Instead, one would need a stability estimate at the level of the full ERA dynamics, where the stress tensor and defect force appear explicitly.
\end{remark}

\begin{remark} The finiteness of the modulated quantities $\mathscr{E}_q(X^N|\eta)(t)$ and $\mathscr{E}_q(V^N|v)(t)$ is equivalent to the moment conditions $\rho_t \in \calP_q(\R^d)$ and $u_t \in L^q(\rho_t)$.
\end{remark}

\begin{remark}As a direct consequence of the definition of $\mathscr{E}_q(X^N|\eta)$, the Wasserstein distance between the particle density $\rho_t^N$ and its macroscopic counterpart $\rho_t$ can be controlled as follows. Since 
\[
\rho_t^N= \intr \mu^N_t(\dw) = \frac1N\sum_{i=1}^N\delta_{x_i(t)} \in \calP_q(\R^d) 
\]
and $\pi=\rho_t^N \otimes \rho_t $ belongs to $\Pi(\rho_t^N,\rho_t)$, we have
\[
{\rd}_q^q(\rho_t^N,\rho_t) \le \intrr |x-y|^q\,\rho_t(\dx)\rho_t^N(\dy)= \mathscr{E}_q^q(X^N|\eta)(t).
\]
\end{remark}
 
We next estimate the discrepancy between the particle momentum density and its Eulerian counterpart in the bounded Lipschitz distance.
\begin{lemma} Let $q \in [1,\infty]$, $m_t^N := \intr v\,\mu_t^N(\dv) = \frac1N\sum_{i=1}^N v_i(t)\,\delta_{x_i(t)}$, and $m_t := \rho_t u_t$. Then we have
\[
{\rd}_{\rm BL} (m_t, m_t^N) \leq \lt(\frac1N\sum_{i=1}^N |v_i(t)|^{\frac{q}{q-1}} \rt)^{1 - \frac1q} \mathscr{E}_q(X^N|\eta)(t) + \mathscr{E}_q(V^N|v)(t).
\]
\end{lemma}
\begin{proof}
For any $\varphi \in W^{1,\infty}(\R^d, \R^d)$, we compute
\begin{align*}
&\intr \varphi(z) \cdot \, (m^N-m)(\dz)\cr
&\quad = \frac1N\sum_{i=1}^N \varphi(x_i)\cdot v_i  -  \intr \varphi(\eta(x))\cdot v(x)\,\rho_0(\dx)\\
&\quad = \frac1N\sum_{i=1}^N \intr (\varphi(x_i)-\varphi(\eta(x)) )\cdot v_i\,\rho_0(\dx)
 +  \frac1N\sum_{i=1}^N \intr \varphi(\eta(x))\cdot (v_i-v(x) )\,\rho_0(\dx)\\
&\quad =:  I + II.
\end{align*}
For $I$, by Lipschitz continuity of $\varphi$ and the kinetic energy bound, we get
\begin{align*}
|I| &\le \|\varphi\|_{\rm Lip}\lt(\frac1N\sum_{i=1}^N |v_i|^{\frac{q}{q-1}} \rt)^{1 - \frac1q}\lt(\frac1N\sum_{i=1}^N \intr |x_i-\eta(x)|^q\,\rho_0(\dx)\rt)^\frac1q\cr
& =  \|\varphi\|_{\rm Lip}\lt(\frac1N\sum_{i=1}^N |v_i|^{\frac{q}{q-1}} \rt)^{1 - \frac1q}\mathscr{E}_q(X^N|\eta).
\end{align*}
For $II$, we find
\[
|II| \le \|\varphi\|_{L^\infty}\,\frac1N\sum_{i=1}^N \intr |v_i-v|\,\rho_0(\dx) \le \|\varphi\|_{L^\infty} \mathscr{E}_q(V^N|v).
\]
Combining the above estimates concludes the desired result.
\end{proof}

%
%
%
%
%

We are now in a position to prove the uniform-in-time mean-field convergence result stated in Theorem \ref{thm:mf}. The proof combines the uniform-in-time Lagrangian stability estimate established in Theorem \ref{thm:deri_LA} with the Wasserstein bounds derived above.
 
\begin{proof}[Proof of Theorem \ref{thm:mf}]
By Theorem \ref{thm:deri_LA}, there exists a constant $C>0$, independent of $N$ and $t$, such that
\[
\mathscr{E}_q(X^N|\eta)(t) + \mathscr{E}_q(V^N|v)(t) \le  C\lt(\mathscr{E}_q(X^N|\eta)(0)+\mathscr{E}_q(V^N|v)(0)\rt) \quad \forall\,t\ge0.
\]
Since the disintegration of $\rho_0$ along $\eta_t$ is Dirac for almost every $t\ge0$, the Eulerian state associated with the Lagrangian flow is mono-kinetic and given by
$\rho_t\otimes\delta_{u_t}$. We may therefore apply Lemma \ref{lem:d_pZ_p} wtih arbitrary order $q$ which yields
\[
\rd_q \left(\mu_t^N, \rho_t\otimes\delta_{u_t}\right)  \le  2^{1 - \frac1q}\lt( \mathscr{E}_q(X^N|\eta)(t)^q+\mathscr{E}_q(V^N|v)(t)^q \rt)^{\frac1q}  \quad \text{a.e. } t \geq 0.
\]
Taking the essential supremum over $t\ge0$ gives
\[
\esssup_{t\ge0}\rd_q \bigl(\mu_t^N, \rho_t\otimes\delta_{u_t}\bigr)\to 0 \qquad \mbox{as } N\to\infty.
\]
Moreover, since the disintegration is Dirac for almost every $t\ge0$, the associated Reynolds stress vanishes for almost every time, and hence $(\rho_t,u_t)$ satisfies the Euler--alignment system \eqref{eq:EA-2} in the sense of distributions. This completes the proof.
\end{proof}

%
%
%
%
%
%

\section*{Acknowledgments}
JAC was supported by the Advanced Grant Nonlocal-CPD (Nonlocal PDEs for Complex Particle Dynamics: Phase Transitions, Patterns and Synchronization) of the European Research Council Executive Agency (ERC) under the European Union’s Horizon 2020 research and innovation programme (grant agreement No. 883363). JAC was also partially supported by the EPSRC grant number EP/V051121/1. 
The work of YPC was supported by NRF grant no. 2022R1A2C1002820 and RS-2024-00406821. The work of ET was supported by ONR grant  N00014-2412659 and NSF grant DMS-2508407.  The author is also grateful for the hospitality of the Laboratoire Jacques-Louis Lions (LJLL) at Sorbonne University, where part of this work was completed.

%
%
%
%

\appendix

\section{Kinetic formulation associated with the Lagrangian flow}\label{app:kinetic}

This appendix provides a kinetic description of the dynamics generated by the Lagrangian flow $(\eta_t,v_t)$. More precisely, we show that the phase-space lifting
\[
\mu_t := (\eta_t,v_t)_\#\rho_0 \in \calP(\R^d\times\R^d)
\]
satisfies the Vlasov--alignment model \eqref{VA}, without any injectivity assumption on the flow, and explain how the ERA system is recovered from this kinetic equation by taking low-order velocity moments. This formulation clarifies the role of fibre disintegration and highlights the origin of Reynolds-type defect terms at the Eulerian level.
%
%
%
%
%
%
\subsection{Weak formulation of the kinetic transport equation}
Let $\Psi\in C_c^\infty([0,T)\times\R^d\times\R^d)$. By definition of the pushforward, testing $\mu_t$ against $\Psi$ corresponds to evaluating $\Psi$ along the Lagrangian trajectories:
\bq\label{eq:kin_test}
\intrr \Psi(t,z,\xi)\,\mu_t(\dz,\rd\xi)
= \intr \Psi(t,\eta_t(x),v_t(x))\,\rho_0(\dx).
\eq
Differentiating \eqref{eq:kin_test} in time and using that $(\eta_t,v_t)$ solves the Lagrangian $p$-alignment system \eqref{Lag_p} (with the regularity ensured by Theorem \ref{thm:ext}), the chain rule yields
\begin{align}\label{eq:kin_chain}
\frac{\rd}{\dt}\intrr \Psi\,\mu_t(\dz,\rd\xi)
&= \intr \Bigl(\pa_t\Psi + (\nabla_z\Psi)\cdot v_t(x)
      +(\nabla_\xi\Psi)\cdot F_p(\eta_t(x),v_t(x))\Bigr)\,\rho_0(\dx),
\end{align}
where the force field is given by
\[
F_p[\mu_t](z,\xi):=\kappa\intrr \phi(z-z')\,G_p(\xi'-\xi)\,\mu_t(\dz',\rd\xi'),
\quad G_p(\xi)=|\xi|^{p-2}\xi.
\]
Using again the identity $\mu_t=(\eta_t,v_t)_\#\rho_0$, each term in \eqref{eq:kin_chain} can be rewritten in kinetic (phase-space) variables. Indeed,
\[
\intr (\nabla_z\Psi)(t,\eta_t(x),v_t(x))\cdot v_t(x)\,\rho_0(\dx)
= \intrr (\nabla_z\Psi)(t,z,\xi)\cdot \xi\,\mu_t(\dz,\rd\xi),
\]
and similarly,
\[
\intr (\nabla_\xi\Psi)(t,\eta_t(x),v_t(x))\cdot F_p[\mu_t](\eta_t(x),v_t(x))\,\rho_0(\dx)
= \intrr (\nabla_\xi\Psi)(t,z,\xi)\cdot F_p[\mu_t](z,\xi)\,\mu_t(\dz,\rd\xi).
\]

Integrating \eqref{eq:kin_chain} over $t\in(0,T)$ and using the compact support of $\Psi$ in $[0,T)$, we obtain the weak formulation
\bq\label{eq:kin_weak}
\int_0^T\intrr \Bigl(\pa_t\Psi + (\nabla_z\Psi)\cdot \xi + (\nabla_\xi\Psi)\cdot F_p[\mu_t]\Bigr)\,\mu_t(\dz,\rd\xi)\,\dt
= -\intrr \Psi(0,z,\xi)\,\mu_0(\dz,\rd\xi).
\eq

In distributional form, \eqref{eq:kin_weak} corresponds to the kinetic transport equation
\bq\label{eq:kinetic_eq}
\pa_t\mu_t + \nabla_z\cdot(\xi\,\mu_t) + \nabla_\xi\cdot\bigl(F_p[\mu_t]\,\mu_t\bigr)=0
\quad \text{in }\calD'((0,T)\times\R^d\times\R^d),
\eq
which holds for all $p\ge2$.
%
%
%
%
%
%
\subsection{Recovery of the ERA system by velocity moments}

Let $\rho_t=(\pi_z)_\#\mu_t$ denote the spatial marginal of $\mu_t$, and define the momentum measure as the first velocity moment
\[
m_t(\dz):=\intr \xi\,\mu_t(\dz,\rd\xi).
\]
Disintegrating $\mu_t$ with respect to $\rho_t$,
\[
\mu_t(\dz,\rd\xi)=\rho_t(\dz)\,\tilde\nu_{t,z}(\rd\xi),
\]
we introduce the barycentric velocity
\[
u_t(z):=\intr \xi\,\tilde\nu_{t,z}(\rd\xi),
\]
and the fibre covariance
\[
\theta_t(z)
:=\intr (\xi-u_t(z))\otimes(\xi-u_t(z))\,\tilde\nu_{t,z}(\rd\xi),
\quad
\tau_t:=\rho_t\,\theta_t.
\]

With these definitions, taking the zeroth and first velocity moments of the kinetic equation \eqref{eq:kinetic_eq} yields, in the sense of distributions, the continuity equation for $\rho_t$ and the momentum balance equation of the ERA system, with Reynolds stress $\tau_t$ and defect force induced by the non-mono-kinetic structure of $\mu_t$. The stress tensor $\tau_t$ and the associated defect force encode the loss of closure at the Eulerian level when the kinetic measure is not mono-kinetic.

Whenever the disintegration of $\rho_0$ along $\eta_t$ is Dirac, one has $\tau_t\equiv0$, and the defect force $\calR_p[\rho_t,u_t]$ also vanishes identically. In particular, injectivity of the Lagrangian flow $\eta_t$ is a sufficient condition for this to occur. As a consequence, the ERA system closes and reduces to the Eulerian $p$-alignment equations.

%
%
%
%
%

 \section{Mean-field limit from Lagrangian to Vlasov/Eulerian $p$-alignment systems}\label{app:gene_mf}
 
This appendix is devoted to a general mean-field analysis of the $p$-alignment dynamics for $p\ge2$. In contrast to the linear case $p=2$ treated in the main text, the estimates derived here yield stability bounds with constants depending on time. As a consequence, the mean-field convergence is obtained only on finite time intervals and naturally leads to kinetic limits for general initial data.
 
 \begin{theorem}\label{thm:deri_pLA} Let $p\geq 2$. Let $\{(x_i,v_i)\}_{i=1}^N$ and $(\eta, v)$ be global classical solutions to systems \eqref{CS} and \eqref{Lag_p}, respectively. Assume that $\rho_0$ has a compact support and $u_0 \in L^\infty(\rho_0)$, so that the initial velocity diameter is finite. Then there exists a constant $C(t)>0$, independent of $N$, such that
 \[
\mathscr{E}_2(X^N|\eta)(t)+\mathscr{E}_2(V^N|v)(t)  \le C(t)\lt(\mathscr{E}_2(X^N|\eta)(0)+\mathscr{E}_2(V^N|v)(0)\rt), \quad t \geq 0.
\]
 \end{theorem}
 
 \begin{proof} We first observe from the maximum principle for the particle system (see, e.g. \eqref{eq:dia_v}) that
 \[
\rd_{V^N}(t) \leq \rd_{V^N}(0), \quad t >0.
 \]
 Since the case $p=2$ can be easily obtained by almost the same argument as in the proof of Theorem \ref{thm:deri_LA}, we only consider $p > 2$.
 
Note that the position estimate remains the same:
\[
\frac{\rd}{\dt} \mathscr{E}_2(X^N|\eta)(t) \leq \mathscr{E}_2(V^N|v)(t).
\]

For the velocity error, we differentiate $\mathscr{E}_2^2(V^N|v)$ and use the nonlinear operator $G_p(z)=|z|^{p-2}z$:
\begin{align*}
&\frac{\rd}{\dt} \mathscr{E}_2^2(V^N|v)(t) \cr
&\quad = \frac {2\kappa}N \sum_{i=1}^N \intr  (v_i - v_t(x))\cdot \bigg\{\frac1N \sum_{j=1}^N\phi (x_j - x_i)G_p(v_j - v_i)  \cr
&\hspace{6cm} - \intr \phi(\eta_t(x) - \eta_t(y))  G_p(v_t(y) - v_t(x))\,\rho_0(\dy) \bigg\}\,\rho_0(\dx)\cr
&\quad = \frac {2\kappa}{N^2} \sum_{i,j=1}^N \intrr  (v_i - v_t(x)) \lt(\phi(x_i - x_j) - \phi(\eta_t(x) - \eta_t(y)) \rt)   \cdot G_p(v_j - v_i) \,\rho_0(\dx)\rho_0(\dy) \cr
&\quad + \frac {2\kappa}{N^2} \sum_{i,j=1}^N \intrr  (v_i - v_t(x)) \phi(\eta_t(x) - \eta_t(y)) \cr
&\hspace{7cm} \cdot  \lt\{ G_p(v_j - v_i)  -  G_p(v_t(y) - v_t(x)) \rt\}   \rho_0(\dx)\rho_0(\dy) \cr
&\quad =: I + II.
\end{align*}

Using $\|\phi\|_{\rm Lip}$ and $\rd_{V^N}(t)$, we get
 \begin{align*}
|I| &\leq \|\phi\|_{\rm Lip} \rd^{p-1}_{V^N}(t)\frac {2\kappa}{N^2} \sum_{i,j=1}^N \intrr |v_i - v_t(x)|  \lt( |x_i - \eta_t(x)| + |x_j  - \eta_t(y)|\rt)\,\rho_0(\dx)\rho_0(\dy)\cr
 &\leq 4\kappa\|\phi\|_{\rm Lip} \rd^{p-1}_{V^N}(0) \mathscr{E}_2(X^N|\eta)(t) \mathscr{E}_2(V^N|v)(t).
 \end{align*}

Using $\phi\ge0$ and the uniform monotonicity of $G_p$,
  \[
( G_p(x) - G_p(y)) \cdot (x-y) \geq 2^{2-p}|x-y|^p, \quad p \ge 2,
  \]
 we estimate $II$ as 
 \begin{align*}
II &= -\frac \kappa{N^2} \sum_{i,j=1}^N \intrr \phi(\eta_t(x) - \eta_t(y)) \lt\{(v_i - v_t(x)) - (v_j - v_t(y))\rt\} \cr
& \hspace{5cm} \cdot \lt\{G_p(v_i - v_j) - G_p(v_t(x) - v_t(y))   \rt\} \rho_0(\dx)\rho_0(\dy)\cr
 &\leq 0.
  \end{align*}
   
Combining the above estimates gives
\[
\frac{\rd}{\dt} \mathscr{E}_2(V^N|v)(t) \leq  2\kappa\|\phi\|_{\rm Lip} \rd^{p-1}_{V^N}(0) \mathscr{E}_2(X^N|\eta)(t),
\]
and subsequently,
\[
\frac{\rd}{\dt} \lt(  \mathscr{E}_2(X^N|\eta)(t) +  \mathscr{E}_2(V^N|v)(t) \rt) \leq \max\lt\{ 1, \, 2\kappa\|\phi\|_{\rm Lip} \rd^{p-1}_{V^N}(0) \rt\} \lt(  \mathscr{E}_2(X^N|\eta)(t) +  \mathscr{E}_2(V^N|v)(t) \rt).
\]

Finally, applying Gr\"onwall's lemma yields the claimed time-dependent stability estimate.
 \end{proof}
 
The time-dependent stability estimate of Theorem \ref{thm:deri_pLA} allows us to derive a corresponding mean-field convergence result on finite time intervals. Since the argument follows the same steps as in the proof of Theorem \ref{thm:mf}, with the uniform-in-time bound replaced by the estimate of Theorem \ref{thm:deri_pLA}, we only state the result and omit the proof.

 \begin{theorem}
Let $T>0$ and $p\geq 2$. Assume that the hypotheses of Theorem \ref{thm:deri_pLA} hold. If the initial modulated energies satisfy
\[
\mathscr{E}_2(X^N|\eta)(0)+\mathscr{E}_2(V^N|v)(0)\to 0\quad\text{as }N\to\infty,
\]
then the empirical measures $\mu_t^N$ associated with the particle system \eqref{CS} converge to a kinetic measure $(\eta_t,v_t)_\#\rho_0$, which is a distributional solution of the Vlasov--alignment equation \eqref{VA} with $p\ge2$, in the sense that
\[
\sup_{t \in [0,T]}
\rd_2 \bigl(\mu_t^N,(\eta_t,v_t)_\#\rho_0\bigr)
\to 0\quad\text{as }N\to\infty.
\]

Suppose in addition that the disintegration of $\rho_0$ along $\eta_t$ is Dirac for almost every $t \in [0,T]$. Then the kinetic limit $(\eta_t,v_t)_\#\rho_0$ is mono-kinetic for almost every $t\in[0,T]$, and the empirical measures $\mu_t^N$ converge toward the associated Eulerian state $\rho_t\otimes\delta_{u_t}$ in the sense that
\[
\esssup_{t \in [0,T]} \rd_2 \bigl(\mu_t^N, \rho_t\otimes\delta_{u_t}\bigr) \to 0\quad\text{as }N\to\infty,
\]
where $(\rho_t,u_t)$ denotes the Eulerian pair associated with the Lagrangian flow $(\eta_t,v_t)$ through the Lagrangian--Eulerian correspondence described in Section \ref{ssec:LEc}, which satisfies \eqref{EA} in the sense of distributions on $(0,T)\times\R^d$.
\end{theorem}

%
%
%
%

\bibliographystyle{abbrv}
\bibliography{CCT_LA}

%
%
%
%
%
%
%
%
%
%
%

\end{document}